\documentclass[12pt]{amsart}

\usepackage{amsmath, amssymb, amsfonts, amsthm, amscd}
\usepackage{manfnt}
\usepackage{mathtools}
\usepackage[mathcal]{euscript}  
\usepackage{fullpage}
\usepackage{url}
\usepackage[all]{xy}
   \SelectTips{cm}{10}
\usepackage{txfonts}
\usepackage[shortlabels]{enumitem}
\SetEnumerateShortLabel{C}{(C\arabic*)}
\usepackage{bbm}
\usepackage[colorlinks=true,linkcolor=blue,urlcolor=blue,citecolor=blue, pagebackref]{hyperref}
\usepackage[usenames, dvipsnames, svgnames]{xcolor}

\newtheorem{thm}{Theorem}[section]
\newtheorem{lemma}[thm]{Lemma}
\newtheorem{prop}[thm]{Proposition}
\newtheorem{cor}[thm]{Corollary}

\newtheorem{prop-conj}[thm]{Proposition-Conjecture}

\newtheorem{thmalph}{Theorem}

\newtheorem{conjalph}[thmalph]{Conjecture}

\theoremstyle{definition}
\newtheorem{defn}[thm]{Definition}

\theoremstyle{definition}
\newtheorem{rmk}[thm]{Remark}

\theoremstyle{definition}

\theoremstyle{definition}
\newtheorem{assumption}[thm]{Assumption}

\theoremstyle{definition}

\theoremstyle{definition}

\theoremstyle{definition}

\theoremstyle{definition}
\newtheorem{eg}[thm]{Example}

\DeclareFontFamily{U}{wncy}{}
    \DeclareFontShape{U}{wncy}{m}{n}{<->wncyr10}{}
    \DeclareSymbolFont{mcy}{U}{wncy}{m}{n}
    \DeclareMathSymbol{\Sha}{\mathord}{mcy}{"58}

\newcommand{\Q}{\mathbb{Q}}
\newcommand{\Qb}{\overline{\mathbb{Q}}}
\newcommand{\Z}{\mathbb{Z}}
\newcommand{\CC}{\mathbb{C}}
\newcommand{\RR}{\mathbb{R}}

\newcommand{\Qlb}{\overline{\mathbb{Q}}_\ell}
\newcommand{\Qpb}{\overline{\mathbb{Q}}_p}

\newcommand{\Fp}{\mathbb{F}_p}

\DeclareMathOperator{\Hom}{Hom}

\DeclareMathOperator{\End}{End}
\DeclareMathOperator{\Aut}{Aut}

\DeclareMathOperator{\Sym}{Sym}
\DeclareMathOperator{\Ad}{Ad}
\DeclareMathOperator{\ad}{ad}
\DeclareMathOperator{\rec}{rec}

\DeclareMathOperator{\im}{im}
\DeclareMathOperator{\coker}{coker}

\DeclareMathOperator{\Spec}{Spec}
\DeclareMathOperator{\cO}{\mathcal O}

\DeclareMathOperator{\Frac}{Frac}
\DeclareMathOperator{\Lie}{Lie}

\newcommand{\gal}[1]{\Gamma_{#1}} 
\newcommand{\Gal}{\mathrm{Gal}} 

\newcommand{\into}{\hookrightarrow}
\newcommand{\onto}{\twoheadrightarrow}

\newcommand{\mc}{\mathcal}
\newcommand{\mf}{\mathfrak}
\newcommand{\mr}{\mathrm}
\newcommand{\mbf}{\mathbf}
\newcommand{\mbb}{\mathbb}

\newcommand{\br}{\bar{\rho}} 
\newcommand{\fg}{\mathfrak{g}}

\newcommand{\fgder}{\mathfrak{g}^{\mathrm{der}}}

\DeclareMathOperator{\Lift}{\mathrm{Lift}}

\newcommand{\kbar}{\bar{\kappa}}

\newcommand{\tF}{\widetilde{F}}

\newcommand{\wt}{\widetilde}
\newcommand{\wh}{\widehat}

\newcommand{\vpi}{\varpi}

\newcommand{\un}[1]{\underline{#1}}
\newcommand{\ov}{\overline}

\title{Lifting and automorphy of reducible mod $p$  Galois representations over global fields}

\thanks{We would like to thank Jack Thorne  for
  helpful  correspondence. We are very grateful to the referee for a particularly careful reading of our paper and for many helpful comments. N.F. was supported by the DAE, Government of India, PIC
  12-R\&D-TFR-5.01-0500. C.K.  would like to thank TIFR, Mumbai for
  its hospitality, in periods when some of the work was carried
  out. S.P. was supported by NSF grants DMS-1700759, DMS-1752313, and DMS-2120325.}

\begin{document}
\author[N.~Fakhruddin]{Najmuddin Fakhruddin}
\address{School of Mathematics, Tata Institute of Fundamental Research, Homi Bhabha Road, Mumbai 400005, INDIA}
\email{naf@math.tifr.res.in}
\author[C.~Khare]{Chandrashekhar Khare}
\address{UCLA Department of Mathematics, Box 951555, Los Angeles, CA 90095, USA}
\email{shekhar@math.ucla.edu}
\author[S.~Patrikis]{Stefan Patrikis}
\address{Department of Mathematics, Ohio State University, 100 Math Tower, 231 West 18th Ave., Columbus, OH 43210, USA}
\email{patrikis.1@osu.edu}

\begin{abstract}
We prove the modularity of most reducible, odd representations $\br: \gal{\Q} \to \mr{GL}_2(k)$ with $k$ a finite field of characteristic an odd  prime  $p$.
This is an analogue  of Serre's celebrated  modularity conjecture (which concerned  irreducible, odd representations $\br: \gal{\Q} \to \mr{GL}_2(k)$) for reducible, odd  representations. Our proof lifts $\br$ to an irreducible geometric $p$-adic representation $\rho$  which is known to arise from a newform by  results of Skinner--Wiles and Pan. We likewise prove automorphy of many reducible representations $\br \colon \gal{F} \to \mr{GL}_n(k)$ when $F$ is a global function field of characteristic different from $p$, by establishing a $p$-adic lifting theorem and invoking the work of L.~Lafforgue. Crucially, in both cases we show that the actual representation $\br$, rather than just its semisimplification, arises from reduction of the geometric representation attached to a cuspidal automorphic representation. Our main theorem establishes a geometric lifting result for mod $p$  representations  $\br:\gal{F} \to G(k)$ of Galois groups of  global  fields $F$, valued in reductive groups $G(k)$, and assumed to be odd when $F$ is a number field. Thus we find that lifting theorems, combined with automorphy lifting results pioneered by Wiles in the number field case and the results in the global Langlands correspondence proved by Drinfeld and L.~Lafforgue in the function field case,  give the only known method to access  modularity of mod $p$  Galois representations both in reducible and irreducible cases.

\end{abstract}

\maketitle

\section{Introduction}

 Serre's modularity conjecture
\cite{serre:conjectures}  asserts that  an odd, irreducible, continuous
representation $\br \colon \Gal(\Qb/\Q) \to \mr{GL}_2(k)$ with $k$ a finite
field of characteristic $p$ arises from a newform, and was proved in \cite{khare-wintenberger:serre1}.  Wiles in the
introduction of his paper on Fermat's Last Theorem  \cite{wiles:fermat} raises the question of whether one can
prove an  analogue of Serre's conjecture in the reducible case, going beyond
the well-known  conclusion of modularity of an odd reducible $\br$ up to
semisimplification (see Lemma \ref{GL2-easy}). Wiles notes on \cite[page 445]{wiles:fermat}: {\it Even in the reducible case not much is known about the problem in the form which we have described it, and  in that case it should be  observed that one must also choose the lattice carefully as only the semisimplification of $\overline \rho_{f,\lambda}=\rho_{f,\lambda} \pmod{\lambda}$ is independent of the choice of lattice in $K_{f,\lambda}^2$.}

We resolve this problem in Theorem \ref{mainappintro} below, under a mild hypothesis  that rules out  some cases when   a twist of  $\br$ has Serre weight $p-1$, making use of developments of Wiles's method of modularity lifting and developments of Ramakrishna's method of lifting Galois representations. Our work improves on the results    of Hamblen and Ramakrishna in
\cite{ramakrishna-hamblen}. We follow their strategy of proving automorphy (or  synonymously modularity)  of reducible, odd  representations  $\br: \Gal(\Qb/\Q) \to \mr{GL}_2(k)$ by lifting $\br$ to a geometric $p$-adic representation  which is known to arise from a newform by results of Skinner--Wiles in  \cite{skinner-wiles:reducible}.  Our improvements to  \cite{ramakrishna-hamblen}   arise from extending   the technology for  lifting mod $p$ Galois representations,  developed in our earlier paper \cite{fkp:reldef} in the case of irreducible Galois representations, to the case of reducible Galois representations (see Theorem \ref{mainthmintro}). 

To expand on our results, we first make the problem more
explicit. 
Let $\Gamma_{\Q}= \Gal(\Qb/\Q)$. Given an odd representation
$\br : \gal{\Q} \to \mr{GL}_2(k)$, is there a newform
$f \in S_r(\Gamma_1(N))$ of some level $N \geq 1$ and weight
$r \geq 2$, with associated Galois representation
$\rho_{f,\iota}:\gal{\Q} \to \mr{GL}_2(E)$ ($E=E_{f,\iota}$ the
completion of $E_f=\Q(a_n(f))$, the Hecke field of $f$, at an
embedding $\iota: E_f \to \overline{\Q}_p$, with valuation ring
$\cO$), and a lattice in $E^2$ stable under
$\rho_{f,\iota}(\gal{\Q})$ such that the resulting integral model
$\gal{\Q} \to \mr{GL}_2(\cO)$ of the representation $\rho_{f , \iota}$
reduces (via some homomorphism $\mc{O} \to \bar{k}$) to $\br$? We may
summarize this by the following diagram:

\[
\xymatrix{
& \mr{GL}_2(\cO) \ar[d] \\
\gal{\Q} \ar@{-->}[ur]^{\rho_{f,\iota}} \ar[r]_-{\br} & \mr{GL}_2(\bar{k}). 
}
\]

Let us recall some standard facts, due to Ribet and Serre, about lattices in
$E^2$ that are stabilized by $\gal{\Q}$ (see proof of Proposition
\ref{finite} below).  The semisimplification of the residual
representation arising from reducing a
$\rho_{f, \iota}(\gal{\Q})$-stable lattice of $E^2$ is independent of
the lattice. When $\rho_{f,\iota}$ is residually irreducible, then up
to scaling by elements of $E^{\times}$ there is a unique lattice in
$E^2$ that is stabilized by $\gal{\Q}$. In the residually reducible
case, the lattice stabilized by the irreducible representation
$\rho_{f,\iota}$ (which always exists) is never unique up to scaling,
although there are only finitely many such stable lattices up to
scaling since $\rho_{f, \iota}$ is irreducible. We say that $\br$ (in
the irreducible and reducible cases) arises from $f$, or $\br $ arises
from $\rho_{f,\iota}$, if we have the relationship summarized in the
diagram above. In the case when $\rho_{f,\iota}$ is residually
reducible, it gives rise to finitely many residual representations
$\br$, and there are at least two non-isomorphic $\br$ that arise from
$f$.

We note in passing that, unlike in the case of Serre's conjecture for
irreducible representations $\br$, where he makes precise the minimal
weight and level of a newform $f$ that gives rise to $\br$, in the
case of reducible mod $p$ representations $\br$ it is not reasonable
to ask that we can choose a newform $f$ giving rise to $\br$ to be in
$S_r(\Gamma_1(N))$, with $r$, $N$ being the Serre weight $k(\br)$ and
Artin conductor $N(\br)$ of $\br$. We expand on this further in the
introduction below and later in Proposition \ref{finite}.

\subsection{Automorphy of reducible mod $p$ Galois representations}

We recall the result of Hamblen and Ramakrishna that proved modularity
for many odd reducible $\br: \gal{\Q} \to \mr{GL}_2(k)$. We denote by
$\kbar$ the mod $p$ cyclotomic character.

\begin{thmalph}[Theorem 2 of \cite{ramakrishna-hamblen}]\label{hrthm}
Let $p \geq 3$ be a prime, $k$ a finite field of characteristic $p$,  and let $\br \colon \gal{\Q} \to \mr{GL}_2(k)$ 
be a continuous indecomposable representation of the form
\[
\br \sim \begin{pmatrix}
\bar{\chi} & \ast \\
0 & 1
\end{pmatrix}.
\]
Assume that
\begin{enumerate}
\item $\bar{\chi}(c)=-1$, $\bar{\chi} \neq \kbar^{\pm 1}, \bar{\chi}^2 \neq 1$.
\item $\br|_{\gal{\Q_p}}$ is either ramified or unramified but  not of the form \[
\br|_{\gal{\Q_p}} \sim \begin{pmatrix}
1 & \ast \\
0 & 1
\end{pmatrix}.
\]
\item $k$ is spanned as vector space over $\mathbb F_p$   by the values of $\bar \chi$.
\end{enumerate}
Then $\br$  arises from a newform $f$ of some level $N$ and weight $r \geq 2$.
\end{thmalph}

The strategy of \cite{ramakrishna-hamblen} is to lift $\br$ to a
geometric representation $\rho:\gal{\Q} \to \mr{GL}_2(\cO)$ that is
ordinary at $p$ of Hodge--Tate weights $(r-1,0)$ for an integer
$r \geq 2$ and then appeal to results of Skinner--Wiles
\cite{skinner-wiles:reducible} to show that $\rho \cong \rho_{f,\iota}$ for
a newform $f \in S_r(\Gamma_1(N))$. We note that there are infinitely
many odd $\bar{\chi}$ 
such that $\bar{\chi}^2 = 1$. Also, because of
an observation of Berger and Klosin
\cite{berger-klosin:residualmodularity} (see Lemma \ref{BK} below),
condition (3) in Theorem
\ref{hrthm} 
is quite restrictive, ruling out infinitely many non-isomorphic $\br$
that arise from extensions defined over $\overline k$, of $1$ by a
fixed character $\bar{\chi}$, and that factor through the Galois group
$\Gamma$ of a fixed finite extension of $\Q$, in the case that
$\dim_{\bar k} H^1(\Gamma, {\bar k}({\bar \chi}))>1$. This observation
also rules out using the lifting method of
\cite{khare-wintenberger:serre0}, which relies on potential
automorphy, in the residually reducible case. When this method works
it produces minimal lifts, which by Lemma \ref{BK} would be too few to
account for the plethora of non-isomorphic $\br$ parametrized by
$H^1(\Gamma, {\bar k}({\bar \chi}))/\bar k^{\times}$.

We substantially improve the above result of
\cite{ramakrishna-hamblen} by eliminating conditions (2) and (3) and
shrinking the infinite set of (odd) exceptional $\bar{\chi}$ in
condition (1) to a singleton.

\begin{thmalph}[See Theorem \ref{GL2}]\label{mainappintro}
Let $p \geq 3$ be a prime, and let $\br \colon \gal{\Q} \to \mr{GL}_2(k)$ 
be a continuous  odd representation of the form
\[
\br \sim \begin{pmatrix}
\bar{\chi} & \ast \\
0 & 1
\end{pmatrix}.
\]
Assume that  $\bar{\chi} \neq \kbar^{-1}$.
Then $\br$  arises from a newform $f$  of some level $N$ and  weight $r \geq 2$.

\end{thmalph}
We expect the condition $\bar{\chi} \neq \bar{\kappa}^{-1}$ (which is
implied by the condition that no twist of $\br$ has Serre weight
$p-1$, but is strictly weaker than this condition on Serre weight) to
be superfluous, see Conjecture \ref{conjintro}, and the fact that we
have to exclude this case in our theorem is imposed by the limitations
of our main result, Theorem \ref{mainthmintro} below, about lifting
Galois representations.  We anticipate that further improvements to
the technology of lifting Galois representations of \cite{fkp:reldef}
and the present paper will in the future allow one to remove this
condition. 
In any case, the proof of Theorem \ref{mainappintro} proceeds by
specializing Theorem \ref{mainthmintro} to the case $G= \mr{GL}_2$ to
produce an irreducible geometric lift $\rho$ of $\br$, and then
invoking the results of \cite{skinner-wiles:reducible} and
\cite{pan:redFM}. Our main undertaking, then, is a substantial
generalization and improvement of the methods of \cite{fkp:reldef}.

We explain briefly our improvements in Theorem \ref{mainappintro} to
the result of Hamblen and Ramakrishna (Theorem \ref{hrthm}).
Assumptions (1) and (3) of Theorem \ref{hrthm} ensure that the adjoint
representation $\ad(\br)$ of $\gal{\Q}$ is multiplicity-free as an
$\Fp[\gal{\Q}]$-module, and that $H^1(\im(\br),
\ad(\br))=0$. Assumption (2) ensures that the local deformation
problem at $p$ is smooth and that the theorem of Skinner--Wiles
applies to the geometric lift of $\br$ (which is ordinary and
distinguished at at $p$) that Hamblen and Ramakrishna produce.

The improvements of Theorem \ref{mainappintro} over Theorem
\ref{hrthm} come about because our lifting methods, in particular the
``doubling argument'' of \S 3 and \S 4, do not need the assumption that the adjoint
representation $\ad(\br)$ is multiplicity free, and the relative
deformation theory argument of \S 5 allows one to lift
$\br$ even if $H^1(\im(\br), \ad(\br))\neq 0$. Furthermore our methods
do not need the assumption that local deformation rings (at $p$) are
smooth, and the improvements of \cite{pan:redFM} to
\cite{skinner-wiles:reducible} allow one to show that the geometric
lift we produce of $\br$ (which may now be non-distinguished at $p$)
is modular. Our assumption that $\bar{\chi} \neq \kbar^{-1}$ ensures
that both $H^0({\rm Gal}(K/\Q), \ad(\br)^*(1))$ and
$H^1({\rm Gal}(K/\Q), \ad(\br)^*(1))$ are $0$,
with $K= F(\br, \mu_p)$ the minimal Galois extension of $F$ that
trivializes both $\br$ and $\kbar$, conditions which are currently
still needed for our lifting methods.

\smallskip

We also prove similar automorphy results over function fields (see
Theorem \ref{FF}) for a class of reducible representations $\br$ of
$\gal{F}$, with $F$ a function field of characteristic different from
that of $k$, valued in $\mr{GL}_n(k)$ (for arbitrary $n$). The
strategy, as before, is to lift $\br$ to an irreducible representation
$\rho:\gal{F} \to \mr{GL}_n(\cO)$
ramified at finitely many places 
of $F$, and then to invoke the results of
\cite{llafforgue:chtoucas} to show that $\rho$ arises from a cuspidal
automorphic representation of $\mr{GL}_n(\mathbb A_F)$. While our
results are much more general, they for instance apply to the
following cases:

\begin{thmalph}[See Corollary \ref{ffmax}]\label{fnfieldappintro}
Fix an integer $n$, and assume $p \gg_n 0$. Let $F$ be a global function field of characteristic $\ell \neq p$, and let $\br \colon \gal{F} \to \mr{GL}_n(k)$ be a continuous representation satisfying the following:
\begin{itemize}
\item $\br$ factors through a maximal parabolic $P(k) \subset \mr{GL}_n(k)$, and the projection $\br_M \colon \gal{F} \to M(k)$ of $\br$ to the Levi quotient of $P$ is absolutely $M$-irreducible.
\item Let $\br_M= \br_1 \oplus \br_2$, where $M \cong \mr{GL}_{n_1} \times \mr{GL}_{n_2}$, and $\br_i$ is the projection to the $\mr{GL}_{n_i}$ factor, ordered such that $\br$ is an extension of $\br_2$ by $\br_1$. Assume that $\br_M$ further satisfies:
\begin{itemize}
\item $\br_1$ is not isomorphic to $\br_2 \otimes \bar{\psi}$ for
  $\bar{\psi} \in \{1, \kbar^{-1}\}$. (In particular, this condition
  always holds if $n_1 \neq n_2$.)
\item Let $\bar{\chi}$ be the character
  $\det(\br_1)^{n_2/d} \cdot \det(\br_2)^{-n_1/d}$, where
  $d = \mr{gcd}(n_1, n_2)$. Then
  $[F(\zeta_p):F(\zeta_p) \cap F(\bar{\chi})]$ is greater than a
  constant depending only on $n_1$ and $n_2$.
\end{itemize}
\end{itemize}
Then $\br$ is automorphic.
\end{thmalph}
See also Corollary \ref{ffss} for an automorphy result when $\br$ is a
direct sum of an arbitrary number of absolutely irreducible
representations.


Our results provide evidence for the expectation that for a global
field $F$, Galois representations $\br:\gal{F} \rightarrow G(k)$, not
necessarily irreducible and which are odd (see \cite[Definition
1.2]{fkp:reldef}) when $F$ is a number field,
arise as reductions of lattices in irreducible geometric
representations $\rho: \gal{F} \rightarrow G(E)$ associated to
cuspidal automorphic representations on the dual group of $G$. In other words the ``mod $p$ Langlands''  correspondence of 
\cite{serre:conjectures}  between  irreducible  mod $p$ Galois
representations  of  $\gal{\Q}$ and mod $p$ newforms  should extend in
some generality to the reducible case as well,  and Galois
representations arising from cuspidal automorphic representations are
rich enough to account for  all extensions of automorphic mod $p$
Galois representations. To state a precise conjecture, and in all our subsequent work, we will have to track the finite ramification sets of our representations; for a finite set $\mc{S}$ of primes of $F$, we let $\gal{F, \mc{S}}= \Gal(F_{\mc{S}}/F)$, where $F_{\mc{S}}$ is the maximal extension of $F$ inside a separable closure $F^{\mr{sep}}$ such that $F_{\mc{S}}$ is unramified outside $\mc{S}$.
In the following, for a reductive group $G'$ over a global field $F$, we will let ${}^L G$ be a Langlands L-group over $\Z$ of $G'$. More precisely, we consider the finite form $G^\vee \rtimes \Gal(\widetilde{F}/F)$, where $G^\vee$ is the split dual group over $\Z$, and $\widetilde{F}$ is the smallest extension of $F$ such that the homomorphism $\mu_{G'}$ of \cite[\S 1.1]{borel:L} factors through $\Gal(\widetilde{F}/F)$. We conjecture:
\begin{conjalph}\label{conjintro}
  Let $F$ be a global field of characteristic different from $p$.  Let
  $G'$ be a quasi-split group over $F$, and let $G={}^L G$ be its
  Langlands L-group over $\Z$. Let ${\mc{S}}$ be a finite set of
  primes of $F$, assumed to contain all primes ramified in
  $\widetilde{F}/F$ and, when $F$ is a number field, all primes above
  $p$. Let $\br \colon \gal{F, {\mc{S}}} \to G(k)$ be a continuous
  $L$-homomorphism, in the sense that the diagram
  \[
  \xymatrix{
  \gal{F, {\mc{S}}} \ar[rr]^{\br} \ar[rd] & & G(k) \ar[ld] \\
  & \Gal(\widetilde{F}/F) &
  }
  \]
  commutes. If $F$ is a number field, further assume that it is totally real, and that $\br$ is odd. Then:
  \begin{itemize}
  \item There is a lift $\rho \colon \gal{F} \to G(\mc{O})$, unramified outside a finite set of primes, of $\br$ to
  the ring of integers $\mc{O}$ in some finite extension of $\Q_p$
  inside $\ov{\Q}_p$, such that the image of $\rho$ is Zariski-dense in $G$, and, in the
  number field case, $\rho$ has regular Hodge--Tate cocharacters.
  \item There is an $L$-algebraic cuspidal automorphic representation
    $\pi$ of $G'(\mathbb{A}_F)$ such that one of the Galois
    representations $\rho_{\pi, \iota} \colon \gal{F} \to
    G(\ov{\Q}_p)$ conjecturally associated to $(\pi, \iota)$ (see
    \cite[Conjecture 3.2.2]{buzzard-gee:conjecture}) is
    $G^0(\Qpb)$-conjugate to $\rho \colon \gal{F} \to G(\mc{O})
    \subset G(\ov{\Q}_p)$.\footnote{
      We say ``one of" the conjectural Galois representations because
      at present one can only formulate a precise statement using
      local-global compatibility properties. While the $G^0$-conjugacy
      classes of Frobenius elements at almost all primes $v$ will not
      suffice in general to determine the $G^0$-conjugacy class of a
      completely-reducible $\gal{F}$-representation, they do when
      $G=G^0$ and the image is Zariski-dense, by \cite[Proposition
      6.4]{bhkt:fnfieldpotaut}. Finally, note that in the function
      field case, this construction of automorphic Galois
      representations is a theorem of V. Lafforgue for arbitrary split
      connected reductive groups $G'$, proven in \cite{vlafforgue}. In
      the number field case, the relevant $\pi$ is expected to have
      regular infinitesimal character, so the Conjecture only concerns
      the ``most accessible" automorphic Galois representations.
    } 
  \end{itemize}
Moreover, if $\br|_{\gal{\wt{F}}}$ is absolutely irreducible, $\rho$ and $\pi$ can be
taken to be unramified outside ${\mc{S}}$. 
\end{conjalph}
We highlight the following example: if $F$ is a function field, then
for any $G$ we expect the trivial representation to be
automorphic. Our methods do not cover this case even for the trivial
representation of $\Gamma_F$ to $G= \mr{GL}_2$. 
Note too that Theorems \ref{mainappintro}, \ref{fnfieldappintro}, and
\ref{mainthmintro} (below) only provide evidence for this conjecture
when $p \gg_G 0$. In particular, they say nothing about the case $p=2$, in which we note that our definition of oddness, specialized to $G= \mr{GL}_2$, is not as general as that of Serre. A more optimistic formulation of Conjecture \ref{conjintro} to include the case $p=2$ would redefine oddness of $\br$ to mean: if $F$ is a number field, then it is totally real, and for all $v \mid \infty$, with associated complex conjugation $c_v$, $\br(c_v)$ lifts to an order $2$ element $\eta_v \in G(\ov{\Z}_p)$ that satisfies the characteristic zero version of the condition of \cite[Definition 1.2]{fkp:reldef}: $\dim_{\ov{\Q}_p}(\fgder_{\ov{\Q}_p})^{\eta_v=1}= \dim \mr{Flag}_{G^0}$.

\subsection{Lifting  mod $p$  Galois representations}

Although the {\it raison d'\^etre} for this paper is the applications
in Theorems \ref{mainappintro} and \ref{fnfieldappintro}, along the
way we undertake a broad generalization of our paper
\cite{fkp:reldef}. Our earlier paper lifts odd  and absolutely
irreducible mod $p$ representations $\br \colon \gal{F} \to G(k)$, for
$F$ totally real and $G$ any reductive group, using general
enhancements of the ideas of \cite{ramakrishna-hamblen} along with a
novel ``relative deformation theory" method (see \cite[\S
6]{fkp:reldef}). The present paper greatly simplifies and extends the
reach of the methods of \cite{fkp:reldef}, and we produce irreducible
geometric $p$-adic lifts of many reducible
$\br \colon \gal{F} \to G(k)$.  Our methods work for any global field,
and the main result of this paper, Theorem \ref{mainthm}, is a lifting
theorem for fairly general---odd in the case of a number field---mod
$p$ Galois representations over global fields. The following theorem
applies to smooth group schemes $G/\Z_p$ such that $G^0$ is a
connected split reductive group, and $\pi_0(G)$ is finite \'etale of
order prime to $p$. Let $\mc{O}$ be the ring of integers in a finite extension $E$ of $\Q_p$ with residue field $k$ and uniformizer $\vpi$. For any group scheme $H$ over $\mc{O}$ and any complete local $\mc{O}$-algebra $\mc{O}'$ with residue field $k'$, we let $\wh{H}(\mc{O}')$ equal $\ker(H(\mc{O}') \to H(k'))$.
 \begin{thmalph}[See Theorem \ref{mainthm}]\label{mainthmintro}

Let $p \gg_G 0$, and let $\br \colon \gal{F, {\mc{S}}} \to G(k)$ be a
continuous representation satisfying Assumptions \ref{generalhyp},
\ref{modp^Nhyp}, and \ref{finalhyp}.
Additionally, when $F$ is a number field, assume that $F$ is totally real, and $\br$ is odd. Then for some finite set of primes $\widetilde{{\mc{S}}}$ containing ${\mc{S}}$, and ring of integers $\mc{O}'$ in some finite extension $E'/E$, there is a geometric lift
\[
\rho \colon \gal{F, \widetilde{{\mc{S}}}} \to G({\mc{O}'}) 
\]
of $\br$. 

More precisely, we first fix a lift $\mu \colon \gal{F, {\mc{S}}} \to
G/G^{\mr{der}}(\mc{O})$ of the multiplier character $\bar{\mu}$ of
$\br$, requiring $\mu$ to be de Rham in the number field
case. Assumption \ref{finalhyp} provides us with local lifts $\rho_v$
of $\br|_{\gal{F_v}}$ for all $v \in {\mc{S}}$, moreover assumed de
Rham and Hodge-Tate regular for $v \vert p$. We then fix an integer
$t$ and for each $v \in {\mc{S}}$ an irreducible component containing $\rho_v$ of:
\begin{itemize}
\item for $v \in {\mc{S}} \setminus \{v \mid p\}$, the generic fiber
  of the local lifting ring, $R^{\square,
    \mu}_{\br|_{\gal{F_v}}}[1/\vpi]$ (where $R^{\square,
    \mu}_{\br|_{\gal{F_v}}}$ pro-represents
  $\Lift_{\br|_{\gal{F_v}}}$); and 
\item for $v \mid p$, the lifting ring
  $R_{\br|_{\gal{F_v}}}^{\square, \mu, \tau, \mbf{v}}[1/\vpi]$ whose
  $\ov{E}$-points parametrize lifts of $\br|_{\gal{F_v}}$ with
  suitably specified inertial type $\tau$ and Hodge type $\mbf{v}$.
(See after the Theorem statement for an explanation.)
\end{itemize}
Then there exist a finite extension $E'$ of $E$ 
(whose ring of integers and residue field we denote by $\mc{O}'$ and
$k'$), and depending only on the set $\{\rho_v\}_{v \in {\mc{S}}}$; a
finite set of places $\wt{{\mc{S}}}$ containing ${\mc{S}}$; and a
geometric lift
\[
\xymatrix{
& G(\mc{O}') \ar[d] \\
\gal{F, \wt{{\mc{S}}}} \ar[r]_{\br} \ar[ur]^{\rho} & G(k') 
}
\]
of $\br$ such that:
\begin{itemize}
\item $\rho$ has multiplier $\mu$. 
\item $\rho(\gal{F})$ contains $\wh{G^{\mr{der}}}(\mc{O}')$.
\item For all $v \in {\mc{S}}$, $\rho|_{\gal{F_v}}$ is congruent modulo $\vpi^t$ to some $\wh{G}(\mc{O}')$-conjugate of $\rho_v$, and $\rho|_{\gal{F_v}}$ belongs to the specified irreducible component for every $v \in {\mc{S}}$.\footnote{To be clear, the set $\wt{{\mc{S}}}$ may depend on the integer $t$, but the extension $\mc{O}'$ does not depend on $t$.}
\end{itemize}
\end{thmalph}
The condition at $v \mid p$ means the following. The
    given lift $\rho_v \colon \gal{F_v} \to G(\mc{O})$ is de Rham and
    thus by \cite{berger:dr=pst} becomes semistable when restricted to
    some finite extension of $F_v$. One can then by a construction of
    Fontaine associate as in \cite[\S 2.6]{bellovin-gee-G} an inertial
    type $\tau \colon I_{F_v} \to G(\ov{E})$, defined up to
    $G^0(\ov{E})$-conjugacy, and a $p$-adic Hodge type $\mbf{v}$
    (\cite[Definition 2.8.2]{bellovin-gee-G}). The corresponding
    lifting ring $R_{\br|_{\gal{F_v}}}^{\square, \mu, \tau, \mbf{v}}$
    is constructed in \cite[Proposition
    3.0.12]{balaji}. In the statement of Theorem \ref{mainthmintro}, we fix some irreducible component of $\Spec(R_{\br|_{\gal{F_v}}}^{\square, \mu, \tau, \mbf{v}}[1/\vpi])$ containing the given lift $\rho_v$. Note that unconditional existence of such lifts $\rho_v$ for all $v \in \mc{S}$ remains an open problem, but one that has seen dramatic progress (for $G^0=\mr{GL}_n$) in the recent work of Emerton--Gee (\cite{emerton-gee:moduli}); see \cite[Remark 6.16]{fkp:reldef} for further discussion.

\subsection{Comparison to \cite{fkp:reldef}} We end this introduction with a more technical section which highlights  the improvements  we make here to the  arguments in \cite{fkp:reldef} so that they can be applied to lifting reducible mod $p$ Galois representations.

 The  work  of  \cite{fkp:reldef}  has three components: 
\begin{itemize}
\item Analysis of local deformation rings, both the fine integral structure at the auxiliary ``trivial primes" (see Definition \ref{trivialdef}) introduced in the global argument, and qualitative results on the integral structure at all primes of ramification. The latter is deduced from knowledge about the generic fibers of such local deformation rings (Kisin's results in \cite{kisin:pst}, as generalized in \cite{bellovin-gee-G}).
\item A generalization of the ``doubling method" of \cite{klr}, as employed in \cite{ramakrishna-hamblen}, which we used to produce a mod $\vpi^N$ lift of $\br$ with prescribed local properties at primes of ramification. 
\item The ``relative deformation theory" argument that exploits results of Lazard on cohomology of $p$-adic Lie groups to surmount the difficulties in annihilating Selmer and dual Selmer (by the introduction of auxiliary primes of ramification) for $\br$ with ``small" image. Relative deformation theory instead annihilates the ``relative" Selmer and dual Selmer, and shows that this suffices to produce the geometric $p$-adic lift in light of the previous step.
\end{itemize}
The most serious difficulty in the present paper is the generalization of the doubling method, where we now work with $\br \colon \gal{F} \to G(k)$ with quite general image, incorporating both the results of \cite{fkp:reldef} and many reducible cases: we catalogue our assumptions at various stages of the argument in Assumptions \ref{generalhyp}, \ref{modp^Nhyp}, and \ref{finalhyp}, and in Appendix \ref{appendix} we give group-theoretic criteria that imply these assumptions. We now sketch the technical obstacles to this generalization.

Let $K= F(\br, \mu_p)$ be the minimal Galois extension of $F$ that trivializes both $\br$ and $\kbar$. The core problem is one of globally interpolating pre-specified collections of local cohomology classes. These local classes arise from the need to adjust a given mod $\vpi^2$ lift (for some ${\mc{T}} \supset {\mc{S}}$) $\rho_2' \colon \gal{F, {\mc{T}}} \to G(\mc{O}/\vpi^2)$ of $\br$ to have desired local properties (and analogously for constructing mod $\vpi^n$ lifts with controlled local properties). The method first constructs global cohomology classes $h^{(v)}$ for the adjoint representation $\br(\fgder)$ that ramify at the auxiliary prime $v$ along a root space associated to some \v{C}ebotarev class of primes from which $v$ was drawn, and that interpolate given local cohomology classes at other primes (in ${\mc{T}}$) of ramification. Faced with the difficulty that $h^{(v)}|_{\gal{F_v}}$ cannot be sufficiently controlled, the doubling method plays multiple such cocycles against each other, adjusting $\rho_2'$ to a lift
\[
\rho_2= \left(1+\vpi(h^{\mr{old}} -\sum_{n \in N} h^{(v_n)}+2 \sum_{n \in N} h^{(v'_n)}) \right)\rho_2'
\] 
for two $N$-tuples $(v_n)$, $(v_n')$ of auxiliary primes, and for $h^{\mr{old}} \in H^1(\gal{F, {\mc{T}}}, \br(\fgder))$ independent of the auxiliary tuples, with $\rho_2$ well-controlled both at the original primes in ${\mc{T}}$ and at all of the auxiliary primes $v_n, v_n'$. The method's mechanism requires a full understanding of linear disjointness of the fixed fields $K_{h^{(v)}}$ as the primes $v$ vary as well as of the fixed fields $K_{\eta^{(v)}}$ of a set (for each $v$) of auxiliary cohomology classes $\eta^{(v)} \in H^1(\gal{F, {\mc{T}} \cup v}, \br(\fgder)^*)$---chosen to map a generator $\tau_v$ of tame inertia at $v$ to particular basis elements of $\br(\fgder)^*$. Use of these classes $\eta^{(v)}$ allows (very roughly) control of the values $h^{(v)}(\sigma_v)$ via global duality (here and throughout $\sigma_v$ denotes a Frobenius element at $v$). In \cite{fkp:reldef}, technical problems arise from two issues:
\begin{itemize}
\item When one decomposes the (semisimple in \textit{loc.~cit.}) adjoint representation $\br(\fgder)= \oplus_{i \in I} W_i^{\oplus m_i}$ into irreducible (distinct) $\Fp[\gal{F}]$-modules $W_i$ (and likewise for $\br(\fgder)^*$), with endomorphism algebras $k_{W_i}/\Fp$, one can only at a given prime $v$ achieve control over a batch of auxiliary cocycles $\eta^{(v)}$ indexed over $i \in I$ and a $k_{W_i}$-basis of $W_i$; a more na\"{i}ve approach that works with, e.g., a collection of $\eta^{(v)}$ whose values $\eta^{(v)}(\tau_v)$ range over a $k$-basis of $\br(\fgder)^*$, will lose the desired linear disjointness once some $m_i>1$ or some $k_{W_i} \subsetneq k$. This forced us to work with the $k_{W_i} \otimes_{\Fp} \mathbb{F}_{p^{m_i}}$-linear versions of the global duality pairings in \textit{loc.~cit.}. Note that in the reducible case, where $\br(\fgder)$ is typically not semisimple, there is no straightforward generalization of this procedure.
\item The fixed fields $K_{h^{(v)}}$ as $v$ varies are ramified at $v$ but not necessarily totally ramified at $v$; this allows them to interfere both with each other, and, in higher (mod $\vpi^N$) stages of the lifting argument, to interfere with the fixed field $K(\rho_2)$. These two factors forced us to work not over $K$ but over an extension $K'$ (notation as in \textit{loc.~cit.}) that captured all such possible undesired intersections.  
\end{itemize}
Section \ref{doublingsection} of the present paper makes the doubling
method more robust and in many ways more transparent by avoiding these
complications. We repeatedly exploit two techniques: all of the
auxiliary cocycles $\psi$ ($h^{(v)}$ or $\eta^{(v)}$ above) are
constructed to have image $\psi(\gal{K})$ that is a cyclic
$\Fp[\gal{F}]$-module, and possible generators include both the value
$\psi(\tau_v)$ at the auxiliary prime of ramification and certain
$\psi(\tau_b)$ for primes $b$ that are pre-inserted into the initial
ramification set ${\mc{T}}$. Thus, for instance, we at the outset include a
collection of primes $b \in \mc{B}$ indexed over a $k$-basis of
$\br(\fgder)^*$, and we then can construct $\{\eta^{(v)}_b\}_b$ with
$\{\eta^{(v)}_b(\tau_v)\}_b$ a $k$-basis and with the
$K_{\eta^{(v)}_b}$ linearly disjoint over $K$ as $b$ varies
($\eta^{(v)}_b(\tau_{b'})=0$ for $b' \neq b$). Thus we are even more
liberal in the use of auxiliary primes than in \cite{fkp:reldef}, with
the benefit of eliminating some of \textit{loc.~cit.'s} technical
obstacles.

\smallskip

    With this done, we can extend the rest of the method of \cite{fkp:reldef} to deal with the reducible
case: in \S \ref{modp^Nsection} we use the results of \S
\ref{doublingsection} to produce carefully-controlled mod $\vpi^N$
lifts of $\br$, and in \S \ref{reldefsection} we adapt the ``relative
deformation theory argument" of \cite[\S 6]{fkp:reldef}, proving our
main theorem on ``indecomposable" reducible representations in Theorem
\ref{mainthm}.

In \S \ref{latticesection} we extend the theorem to the decomposable case, and having done that we proceed to the applications in \S \ref{GL2section}-\S \ref{numberfieldGLn}: namely, we treat $\mr{GL}_2$ over totally real fields in \S \ref{GL2section}, $\mr{GL}_n$ over function fields in \S \ref{functionfields}, and in \S \ref{numberfieldGLn} we combine our results with those of \cite{blggt:potaut} and \cite{allen-newton-thorne:reducible} to lift certain reducible $n$-dimensional representations over CM fields to compatible systems. In Appendix \ref{appendix} we work out more explicit group-theoretic criteria that imply the Assumptions \ref{generalhyp}, \ref{modp^Nhyp}, and \ref{finalhyp} needed to run the lifting argument. In Appendix \ref{ordappendix} we generalize results of \cite{geraghty:ordinary} on the generic fibers of ordinary deformation rings (when $G= \mr{GL}_n$) to the case of general $G$, which we apply in \S \ref{numberfieldGLn} to the case $G= \mr{GSp}_{2n}$.  

\smallskip

We end the introduction with the following heuristic remark.  In both
the present paper and \cite{fkp:reldef}, our improvements of the
methods of \cite{ramakrishna-hamblen} ultimately come about by
allowing at several turns in the argument more primes to ramify than
would be allowed in \cite{ramakrishna-hamblen} (in getting mod $p^N$
liftings using the doubling method of \cite{klr}, in arguments to kill
relative mod $p$ dual Selmer). Allowing more primes to ramify gives
greater degrees of freedom, and we develop arguments to harness this
greater freedom to lift reducible mod $p$ Galois representations to
geometric, irreducible $p$-adic Galois representations.

\subsection{A guide to reading this paper alongside \cite{fkp:reldef}.}
The reader will need a copy of \cite{fkp:reldef} available while reading this paper. As mentioned earlier in the introduction, there are three basic parts to the arguments of \cite{fkp:reldef}: analysis of local deformation rings, the doubling method, and relative deformation theory. We do not repeat from \cite{fkp:reldef} any of the local theory and instead simply cite the relevant definitions and results from \cite[\S 3-4]{fkp:reldef}, which the reader will need to review. The most significant technical changes in the present paper occur in the doubling method, which is explained in its basic form in \S \ref{doublingsection}. Here we have repeated (with suitable modification) all of the constructions in detail but do still at some points refer in Propositions \ref{5.9-5.11sub} and \ref{doublingprop} to the corresponding arguments of \cite[\S 5]{fkp:reldef} for some verifications. Nevertheless, this section is mostly self-contained. Similar remarks apply to the present \S \ref{modp^Nsection}, which uses the (newly improved) doubling method to find mod $p^N$ lifts of $\br$ with desirable local properties. Our treatment of the relative deformation theory argument in \S \ref{reldefsection} is more abbreviated, since there are fewer technical changes. We expect the reader to have read \cite[\S 6]{fkp:reldef} in advance, and so we only sketch the proof of Theorem \ref{mainthm}, with references to \cite[\S 6]{fkp:reldef}, reducing it to the two statements (the analogues of \cite[Lemma 6.4, Proposition 6.8]{fkp:reldef}) where our present argument must proceed somewhat differently; Lemma \ref{vanishing} and Proposition \ref{auxprimes} then carry out the new details. The remainder of the paper (starting with \S 6) is devoted to applications of our results and does not further rely on \cite{fkp:reldef} except for some references to group-theoretic results proven in \cite[Appendix A]{fkp:reldef}.

\begin{section}{Preliminaries}\label{prelims}
  As in \cite{fkp:reldef}, throughout this paper $G$ will denote a
  smooth group scheme over $\Z_p$ such that $G^0$ is a split reductive
  group, and $\pi_0(G)$ is finite \'{e}tale of order prime to $p$. We always assume that \cite[Assumption 2.1]{fkp:reldef} holds, namely that $p \neq 2$ is a very good prime for the derived group $G^{\mr{der}}$ of $G^0$ and that the canonical central isogeny $G^{\mr{der}} \times Z_{G^0}^0 \to G^0$ has kernel of order prime to $p$. See \textit{loc. cit.} for further discussion of this condition. In particular, any assumption that $p$ is sufficiently large compared to the root datum of $G^0$, which we write $p \gg_G 0$, includes this assumption. We
  let $\mc{O}$ be the ring of integers in a finite extension $E$ of
  $\Q_p$, with uniformizer $\varpi$ and residue field $k$. Let $F$ be
  a global field of characteristic prime to $p$; that is, $F$ is
  either a number field, or it is the field of functions of a smooth
  geometrically connected curve over some finite field $\mathbb{F}$ of
  characteristic $\ell>0$. For a homomorphism
  $\br \colon \gal{F} \to G(k)$, we let
  $\bar{\mu} \colon \gal{F} \to G/G^{\mr{der}}(k)$ denote its image
  after quotienting by the derived group $G^{\mr{der}}$ of $G^0$. We
  will always fix a lift
  $\mu \colon \gal{F} \to G/G^{\mr{der}}(\mc{O})$ of $\bar{\mu}$, and
  in the number field case we will take $\mu$ to be de Rham at places
  above $p$. Aside from allowing $F$ to be any global field, the rest
  of the notation in this paper is the same as in \cite{fkp:reldef}
  (esp. \S 1.4) to which we refer the reader in case clarification is
  needed.

While our main focus in this paper is on the number field case, our
results apply equally well to global function fields. Since references
in the literature are often written only in the number field case, we
collect here the essential arithmetic results on which the methods of
\cite{fkp:reldef} and this paper depend. We begin with the local
input. Tate local duality (\cite[7.2.6 Theorem]{neukirch:cohnum})
holds equally well for $\ell$-adic local fields and for local fields
of characteristic $\ell >0$, requiring in the latter case that the
Galois modules in question have order prime to $\ell$ (as will always
be the case for us). The same remark applies to Tate's local Euler
characteristic formula (\cite[7.3.1 Theorem]{neukirch:cohnum}). The
other local input we will need in the function field case is on the
structure of local deformation rings. Namely, if $F$ is a global
function field, and $v$ is a place of $F$, suppose we are given a
homomorphism $\br \colon \gal{F_v} \to G(k)$. As in \cite[Proposition
4.7]{fkp:reldef}, we consider the lifting ring
$R^{\square, \mu}_{\br}$ and its generic fiber
$R^{\square, \mu}_{\br}[1/\vpi]$, and a choice of irreducible
component of the latter gives rise to its Zariski-closure $R$ in the
former. We require that \cite[Proposition 4.7]{fkp:reldef} continues
to hold for $R$, and that $R[1/\vpi]$ has an open dense regular
subscheme (allowing us to apply \cite[Lemma 4.9]{fkp:reldef}). The
input we need beyond the arguments of \textit{loc.~cit.} is that the
analysis of generic fibers of \cite[Theorem 3.3.3]{bellovin-gee-G} (or
\cite[Theorem 14]{stp-booher:Gdef}) continues to hold, as does
\cite[Lemma 3.4.1]{bellovin-gee-G}. The latter is clear since its proof uses
no arithmetic. The former follows by precisely the original arguments
of \cite{bellovin-gee-G}: one uses the dictionary between Weil--Deligne
representations and (for us) $p$-adic Galois representations of
$\gal{F_v}$ (Grothendieck's ``$\ell$"-adic monodromy theorem), which
holds equally well in the equal characteristic case; then the analysis
of \textit{loc.~cit.} proceeds entirely in terms of Weil--Deligne
representations.

We also use finer information about the local conditions at our auxiliary primes of ramification: this works as in \cite[\S 3]{fkp:reldef}, since the Galois group of the maximal tamely ramified with pro-$p$ ramification index extension of an equicharacteristic local field is isomorphic to the analogous tame Galois group for a mixed characteristic local field having the same residue field $\mathbb{F}$ of characteristic prime to $p$: they are both semi-direct products $\Z_p \rtimes \hat{\Z}$, where a generator $\sigma \in \hat{\Z}$ acts on $\tau \in \Z_p$ by $\sigma \tau \sigma^{-1}= \tau^{\#\mathbb{F}}$. 

Now let $F$ be a global field, let $\gal{F}= \Gal(F^{\mr{sep}}/F)$ for
a separable closure $F^{\mr{sep}}$ of $F$, let $M$ be a finite
$p$-primary discrete $\gal{F}$-module, where $p$ is a prime not equal
to the characteristic of $F$, and let ${\mc{S}}$ be a finite set of places of
$F$ satisfying:
\begin{itemize}
\item ${\mc{S}}$ contains all places of $F$ at which $M$ is ramified.
\item When $F$ is a number field, ${\mc{S}}$ contains all places dividing $\infty$ and all places dividing $p$.
\end{itemize}
Then in both the number field and function field case, we have access
to the following results about the Galois cohomology of
$\gal{F, {\mc{S}}}= \Gal(F_{\mc{S}}/F)$, where $F_{\mc{S}}$ is the maximal extension of $F$
inside $F^{\mr{sep}}$ unramified outside ${\mc{S}}$:
\begin{enumerate}
\item Poitou--Tate duality: we will apply the Poitou--Tate duality theorem (\cite[8.6.7 Theorem]{neukirch:cohnum}), the long-exact Poitou--Tate sequence (\cite[8.6.10]{neukirch:cohnum}), and its variant for Selmer groups, which is an easy consequence of the proof of \cite[8.7.9 Theorem]{neukirch:cohnum}.
\item The global Euler-characteristic formula, particularly in its incarnation as the Greenberg--Wiles formula \cite[8.7.9 Theorem]{neukirch:cohnum}.
\item The \v{C}ebotarev density theorem: the usual statement for number fields carries over to global functions fields (see \cite[\S 6.4]{fried-jarden:field} for a proof).
\end{enumerate}
Finally, we include an elementary lemma of Galois theory that we used implicitly throughout \cite{fkp:reldef} but felt it would be clearer to make explicit:
\begin{lemma}\label{abelian}
Let $L/F$ be any finite Galois extension of fields, and let $M/F$ be an abelian extension. Then $\Gal(L/F)$ acts trivially on $\Gal(LM/L)$ via the canonical action.
\end{lemma}
\begin{proof}
Let $\sigma \in \Gal(L/F)$, and let $h \in \Gal(LM/L)$. The action is given by lifting $\sigma$ to any $\tilde{\sigma} \in \Gal(LM/F)$, and then setting $\sigma \cdot h= \tilde{\sigma}h\tilde{\sigma}^{-1}$. Since the restriction map $\Gal(LM/L) \to \Gal(M/F)$ is injective, and the commutator $\tilde{\sigma}h\tilde{\sigma}^{-1}h^{-1}$ restricts to the identity in the abelian group $\Gal(M/F)$, we must have $\tilde{\sigma}h\tilde{\sigma}^{-1}h^{-1}=1$, i.e. $\sigma \cdot h= h$.
\end{proof}
We will frequently use this lemma without further comment.
\end{section}

\section{Lifting mod $\vpi^2$}\label{doublingsection}
We begin with variants of the arguments of \cite[\S 5]{fkp:reldef}; the reader should begin by reviewing the conventions established in \cite[Notation 5.1]{fkp:reldef}, which we adopt here. In particular, unless otherwise noted, ``dimension" refers to $\Fp$-dimension.
In the present section, we make the following assumptions:
\begin{assumption}\label{generalhyp}
  Assume $p \gg_G 0$, and let $\br \colon \gal{F, {\mc{S}}} \to G(k)$ be a
  continuous representation unramified outside a finite set of finite
  places ${\mc{S}}$; we may and do assume that ${\mc{S}}$ contains all places above
  $p$ if $F$ is a number field. Let $\tF$ denote the
  smallest extension of $F$ such that $\br(\gal{\tF})$ is contained in
  $G^0(k)$, and let $K= F(\br, \mu_p)$.
Assume that $\br$ satisfies:
\begin{itemize} 

 \item $H^1(\Gal(K/F), \br(\fgder)^*)$=0.

\end{itemize}
\end{assumption}
We fix as always a lift $\mu \colon \gal{F, {\mc{S}}} \to G/G^{\mr{der}}(\mc{O})$ of the multiplier character $\bar{\mu}$ and consider in what follows only lifts of $\br$ with multiplier character $\mu$. We set $D$ equal to the greatest integer such that $\mu_{p^D}$ is contained in $K$. Throughout the paper we will let $K_{\infty}=K(\mu_{p^\infty})$ denote the $p$-adic cyclotomic tower over $K$.
\begin{defn}\label{trivialdef}
As in \cite[\S 3]{fkp:reldef}, we will say a prime $w$ of $F$ is \textit{trivial} if $\br|_{\gal{F_w}}=1$ and $N(w) \equiv 1 \pmod p$. 
\end{defn}
By definition of the integer $D$, our trivial primes thus satisfy the stronger condition $N(w) \equiv 1 \pmod{p^D}$, but we will specify this condition explicitly in our arguments.
\begin{rmk}
We fix once and for all decomposition groups $\gal{F_v} \into \gal{F}$ at the places in ${\mc{S}}$, and whenever we introduce auxiliary trivial primes we will specify decomposition groups. See \cite[Notation 5.1]{fkp:reldef} for further remarks on these choices; in particular, they allow us to make sense of elements $\phi(\sigma_w)$, $\phi(\tau_w)$ where $\phi \in H^1(\gal{F, {\mc{S}}'}, \br(\fgder))$ and $w \in {\mc{S}}' \supset {\mc{S}}$ is a trivial prime (and likewise for $\br(\fgder)^*$-valued cohomology classes).
\end{rmk}

As in \cite[\S 5]{fkp:reldef}, we enlarge the set ${\mc{S}}$ to a set ${\mc{T}}$ of trivial primes such that $\Sha^1_{\mc{T}}(\Gamma_{F, {\mc{T}}}, \br(\fgder)^*)=0$, and hence such that $\Sha^2_{\mc{T}}(\Gamma_{F, {\mc{T}}}, \br(\fgder))=0$. This requires our assumption on the vanishing of $H^1(\Gal(K/F), \br(\fgder)^*)$: any non-zero class $\psi \in H^1(\gal{F, {\mc{S}}}, \br(\fgder)^*)$ has non-trivial restriction $\psi|_{\gal{K}}$, and then we can choose a prime $w$ split in $K$ such that $\psi|_{\gal{F_w}}$ is non-zero. We note that the primes $w \in {\mc{T}} \setminus {\mc{S}}$ thus produced must satisfy $N(w) \equiv 1 \pmod{p^D}$ and must be non-trivial in the extension $K_{\psi}/K$ cut out by the relevant class $\psi$. The intersection of $K_{\psi}$ with the $p$-adic cyclotomic tower $K_{\infty}$ will either equal $K$ or $K(\mu_{p^{D+1}})$: it can be no larger since $\Gal(K_{\psi}/K)$ is killed by $p$. In particular, $N(w)$ is congruent to $1 \pmod{p^D}$ but may if desired be chosen non-trivial modulo $p^{D+1}$ (this is compatible with the non-vanishing condition in $K_{\psi}$), and we can if desired prescribe $N(w)$ modulo some higher power of $p$ subject to this constraint modulo $p^{D+1}$. In what follows, as we add auxiliary trivial primes we will keep track of how they split in $K_{\infty}$, but we emphasize that we do not have any particular requirement of the detailed numerics: however they turn out, the later stages of the lifting argument will be able to accommodate them. We further enlarge ${\mc{T}}$ as follows:
\begin{lemma}\label{minimizeSha}
\begin{enumerate}
\item
There is a finite enlargement by trivial primes of ${\mc{T}}$ (which we will continue to denote by ${\mc{T}}$) with the following property: for all cyclic submodules $M_Z:= \Fp[\Gamma_F]\cdot Z \subset \br(\fgder)$ (for $Z \in \fgder$), $\dim_{\Fp} \Sha^1_{\mc{T}}(\Gamma_{F, {\mc{T}}}, M_Z^*)$ is minimal among all such enlargements, i.e. is equal to $\dim_{\Fp} \Sha^1_{{\mc{T}}_{\mr{max}}}(\Gamma_{F, {\mc{T}}}, M_Z^*)$, where ${\mc{T}}_{\mr{max}}$ is the union of ${\mc{S}}$ with the set of \textit{all} trivial primes of $F$.

\item Let ${\mc{T}}$ be the enlargement produced in (1). For any trivial
  prime $w \not \in {\mc{T}}$, let $L_{w,Z} \subset H^1(\Gamma_{F_w}, M_Z)$
  be the subspace of cocycles $\phi$ such that $\phi(\tau_w) \in \Fp Z$. Then there is an exact sequence
\[
0 \to H^1(\Gamma_{F, {\mc{T}}}, M_Z) \to H^1_{L_{w,Z}}(\Gamma_{F, {\mc{T}} \cup w}, M_Z) \xrightarrow{\mr{ev}_{\tau_w}} \Fp Z \to 0,
\]
where $\mr{ev}_{\tau_w}$ is the evaluation map $\phi \mapsto \phi(\tau_w)$.
\item There is a further finite enlargement of the ${\mc{T}}$ produced in (1)
  that satisfies the analogous properties in (1) and (2) with respect
  to cyclic submodules
  $M_{\lambda} =\Fp[\Gamma_F] \cdot \lambda \subset \br(\fgder)^*$,
  for all $\lambda \in \br(\fgder)^*$.
\end{enumerate}
\end{lemma}
\begin{proof}
The first part follows from two observations:
\begin{itemize}
\item For fixed $Z$, $\Sha^1_{\mc{T}}(\Gamma_{F, {\mc{T}}}, M_Z^*)$ is finite-dimensional.
\item As $Z$ varies in $\fgder$, there are only finitely many modules $M_Z^*$ to consider.
\end{itemize}

For the second part, we apply the Greenberg--Wiles Euler characteristic formula twice and obtain
\[
h^1_{L_{w,Z}}(\Gamma_{F, {\mc{T}} \cup w}, M_Z)-h^1(\Gamma_{F, {\mc{T}}}, M_Z)-h^1_{\{0\}_{\mc{T}} \cup L_{w, Z}^{\perp}}(\Gamma_{F, {\mc{T}} \cup w}, M_Z^*)+\dim_{\Fp}\Sha^1_{\mc{T}}(\Gamma_{F, {\mc{T}}}, M_Z^*)=\dim_{\Fp}L_{w,Z}-\dim_{\Fp}L_w^{\mr{unr}}.
\]
By Part (1), the $M_Z^*$ terms cancel, and we conclude
\[
h^1_{L_{w,Z}}(\Gamma_{F, {\mc{T}} \cup w}, M_Z)-h^1(\Gamma_{F, {\mc{T}}}, M_Z)=1,
\]
from which the exactness easily follows.

The argument for $\br(\fgder)^*$ is the same.
\end{proof}
\begin{rmk}
  If we knew $H^1(\Gal(K/F), M_Z^*)=0$ for all cyclic quotients
  $M_Z^*$ of $\br(\fgder)^*$, then we could just annihilate the
  Tate--Shafarevich groups in the lemma by explicit choice of trivial
  primes. The weaker statement of the lemma is all we need. Note that
  in the lifting application, we make at the outset an enlargement of
  the coefficient field to ensure that local lifts at places in ${\mc{S}}$
  exist, but then the coefficient field remains fixed throughout the
  rest of the argument.
\end{rmk}
We thus enlarge ${\mc{T}}$---not changing the notation---as in the lemma. For later use, we will impose the following further enlargements of ${\mc{T}}$ before proceeding:
\begin{itemize}
\item Fix a $k$-basis $\{e_b^*\}_{b \in B}$ of $\br(\fgder)^*$ and,
  for each $b \in B$, include in ${\mc{T}}$ an additional trivial prime
  $t_b$.
\item Include one more trivial prime $t_0$ in ${\mc{T}}$. 
\end{itemize} 
The role of these two enlargements will at this point be unclear: we will use them as technical devices for ensuring the fixed fields of certain auxiliary cocycles are linearly disjoint. This allows us ultimately to \textit{avoid introducing the field} $K'$ of \cite[Definition 5.8]{fkp:reldef}. We then modify \cite[Proposition 5.9, Lemma 5.11]{fkp:reldef} as follows, continuing with the notation of Lemma \ref{minimizeSha}:
\begin{prop}\label{5.9-5.11sub}
Let $r$ be the dimension (over $\Fp$) of the cokernel of the restriction map 
\[
\Psi_{\mc{T}} \colon H^1(\gal{F, {\mc{T}}}, \br(\fgder)) \rightarrow \bigoplus_{v \in {\mc{T}}}  H^1(\gal{F_v}, \br(\fgder)).
\] 
Fix an integer $c \geq D+1$ and a Galois extension $L/F$ containing $K$, unramified outside ${\mc{T}}$, and linearly disjoint over $K$ from the composite of $K_{\infty}$ and the fixed fields $K_{\psi}$ of any collection of classes $\psi \in H^1(\gal{F, {\mc{T}}}, \br(\fgder)^*)$. There is 
\begin{itemize}
\item a collection $\{Y_i\}_{i=1}^r$ of elements of
      $\bigoplus_{v \in \mc{T}} H^1(\gal{F_v}, \br(\fgder))$ with
      image equal to an $\Fp$-basis of $\mr{coker}(\Psi_{\mc{T}})$; 
  and for each $i$
\item a class $q_i \in \ker((\Z/p^c)^\times \to (\Z/p^D)^\times)$ that is non-trivial modulo $p^{D+1}$;
\item a split maximal torus $T_i$, a root $\alpha_i \in \Phi(G^0, T_i)$, and a root vector $X_{\alpha_i}$; and, at this point choosing a tuple $g_{L/K, 1}, \ldots, g_{L/K, r}$ of elements of $\Gal(L/K)$,
\item a \v{C}ebotarev set $\mc{C}_i$ of trivial primes $v \not \in {\mc{T}}$ and a positive upper-density subset $\mf{l}_i \subset \mc{C}_i$; 
\item for each $v \in \mc{C}_i$ a choice of decomposition group at $v$;
\item for each $v \in \mf{l}_i$ a class $h^{(v)} \in H^1(\gal{F, {\mc{T}}\cup v}, \br(\fgder))$; 
\end{itemize}
such that (the choice of decomposition group being implicit in what follows)
\begin{itemize}
\item For all $v \in \mc{C}_i$, $N(v) \equiv q_i \pmod{p^c}$, and the image of $\sigma_v$ in $\Gal(L/K)$ is $g_{L/K, i}$. 
\item For all $v \in \mf{l}_i$:
\begin{itemize}
\item $h^{(v)}|_{\mc{T}}= Y_i$.
\item $h^{(v)}(\tau_v)$ is a non-zero element of the span $\Fp X_{\alpha_i}$. Likewise, $h^{(v)}(\tau_{t_0})$ is a non-zero element of $\Fp X_{\alpha_i}$.
\item The class $h^{(v)}$ lies in the image of $H^1(\Gamma_{F, {\mc{T}} \cup v}, M_{X_{\alpha_i}}) \to H^1(\Gamma_{F, {\mc{T}} \cup v}, \br(\fgder))$.
\end{itemize}
\end{itemize}
Similarly, for any $c \geq D+1$, $L/F$ as above, and non-zero element $Z \in \fgder$, there is a class $q_Z \in \ker((\Z/p^c)^\times \to (\Z/p^D)^\times)$ that is non-trivial modulo $p^{D+1}$, and, for any choice of $g_{L/K, Z} \in \Gal(L/K)$, a \v{C}ebotarev set $\mc{C}_Z$ of trivial primes (and choice of decomposition group at each such prime) containing a positive upper-density subset $\mf{l}_Z \subset \mc{C}_Z$, and for each $v \in \mf{l}_Z$ a class $h^{(v)} \in H^1(\Gamma_{F, {\mc{T}} \cup v}, \br(\fgder))$ such that
\begin{itemize}
\item For all $v \in \mc{C}_Z$, $N(v) \equiv q_Z \pmod{p^{c}}$, and the image of $\sigma_v$ in $\Gal(L/K)$ is $g_{L/K, Z}$.
\item For all $v \in \mf{l}_Z$:
\begin{itemize}
\item The restriction $h^{(v)}|_{\mc{T}}$ is independent of $v \in \mf{l}_Z$.
\item $h^{(v)}(\tau_v)$ spans the line $\Fp Z$. Likewise, $h^{(v)}(\tau_{t_0})$ spans $\Fp Z$.
\item $h^{(v)}$ is in the image of $H^1(\Gamma_{F, {\mc{T}} \cup v}, M_Z) \to H^1(\Gamma_{F, {\mc{T}} \cup v}, \br(\fgder))$.
\end{itemize}
\end{itemize}
\end{prop}
\begin{proof}
  The argument is as in \cite[Proposition 5.9, Lemma 5.11]{fkp:reldef}, and the reader should begin by reviewing the proof of \textit{loc. cit.}: we will refer to the
  argument and notation of those results, taking care to extract
  slightly more precise conclusions. The first part of that
      argument applies the \v{C}ebotarev density theorem $r$ times in
      extensions of the form $LK_{\psi_i}(\mu_{p^c})/F$, for a certain
      $\mathbb{F}_p$-basis $\{\psi_i\}_{i=1}^r$ of
      $H^1(\gal{F, {\mc{T}}}, \br(\fgder)^*)$. Let us explain the first step, referring to \textit{loc. cit.} for the details of the induction. Since $\Sha^2_{\mc{T}}(\gal{F, \mc{T}}, \br(\fgder))=0$, the Poitou--Tate sequence yields an isomorphism $\coker(\Psi_{\mc{T}}) \xrightarrow{\sim} \left(H^1(\gal{F, \mc{T}}, \br(\fgder)^*) \right)^\vee$. When $r=0$ there is nothing to prove in the first part of the Proposition (the analogue of \cite[Proposition 5.9]{fkp:reldef}), so we assume $r>0$. Thus $H^1(\gal{F, \mc{T}}, \br(\fgder)^*)$ contains a non-zero class $\psi_1$. The image $\psi_1(\gal{K})$ is non-zero by Assumption \ref{generalhyp}, and so there is a split maximal torus $T_1$, a root $\alpha_1 \in \Phi(G^0, T_1)$, and a root vector $X_{\alpha_1}$ such that $\psi_1(\gal{K})$ is not contained in $(\Fp X_{\alpha_1})^\perp$: indeed, under the assumption $p \gg_G 0$, we can find a root vector outside of any proper $\Fp$ subspace $\fgder$. We now apply the \v{C}ebotarev density theorem in the Galois extension $LK_{\psi_1}(\mu_{p^c})/F$. Here there is a slight difference from \cite[Proposition 5.9]{fkp:reldef}: whereas in \textit{loc. cit.} $D=1$ and the field $K_{\psi_1}$ is linearly disjoint over $K$ from $K_{\infty}$, here $D$ is general and it is possible for $K_{\psi_1} \cap K_{\infty}$ to equal either $K$ or $K(\mu_{p^{D+1}})$. When $K_{\psi_1} \neq K(\mu_{p^{D+1}})$, in our application of the \v{C}ebotarev theorem we can proceed as in \textit{loc. cit.}, specifying $q_1$ arbitrarily in $\ker((\Z/p^c)^\times \to (\Z/p^D)^\times)$, and prescribing the \v{C}ebotarev condition on the splitting of $v$ in $K_{\psi_1}(\mu_{p^c})/F$ such that  $N(v) \equiv q_1 \pmod{p^c}$ while also ensuring $\psi_1|_{\gal{F_{v}}} \neq 0$. On the other hand, when $K_{\psi_1}=K(\mu_{p^{D+1}})$, our \v{C}ebotarev condition in the extension $K_{\psi_1}(\mu_{p^c})/F$ can only arrange $\psi_1|_{\gal{F_v}} \neq 0$ along with $N(v) \equiv q_1 \in \ker(\Z/p^c)^\times \to (\Z/p^D)^\times)$ reducing to some \textit{non-trivial} class in $(\Z/p^{D+1})^\times$. That said, since $L$ is disjoint over $K$ from $K_{\psi_1}(\mu_{p^c})$ we can now fix one of these admissible congruence classes $q_1 \in (\Z/p^c)^\times$, then fix a class $g_{L/K, 1} \in \Gal(L/K)$, and by \v{C}ebotarev find a positive-density set $\mc{D}_1$ of trivial (i.e., split in $K$) primes $v$ along with a choice of decomposition group at $v$ such that $\sigma_v= g_{L/K, 1}$, $\psi_1(\sigma_v)$ is not in $(\Fp X_{\alpha_1})^\perp$, and $N(v) \equiv q_1 \pmod{p^c}$. For each $v_1 \in \mc{D}_1$, we set $L_{v_1}= \{\phi \in H^1(\gal{F_{v_1}}, \br(\fgder)): \phi(\tau_{v_1}) \in \Fp X_{\alpha_1}\}$; $L_{v_1}$ contains the unramified classes $L_{v_1}^{\mr{un}}$ with codimension 1, and dually $L_{v_1}^\perp \subset (L_{v_1}^{\mr{un}})^\perp$ has codimension 1. In particular, the set 
$\{\psi \in H^1(\gal{F, \mc{T} \cup v_1}, \br(\fgder)^*): \psi|_{\gal{F_{v_1}}} \in L_{v_1}^\perp\}$ is a subspace of $H^1(\gal{F, \mc{T}}, \br(\fgder)^*)$, which as in \cite{fkp:reldef} we denote $H^1_{L_{v_1}^\perp}(\gal{F, \mc{T}}, \br(\fgder)^*)$. The Greenberg--Wiles formula implies that the dimension $h^1_{L_{v_1}^\perp}(\gal{F, \mc{T}}, \br(\fgder)^*)$ of this space is $r-1$. The inductive argument of \textit{loc. cit.} now repeats verbatim, substituting $K$ for every appearance of $K'$ in \textit{loc. cit.} and, just as in the selection of $q_1$, noting the possibly more limited flexibility we have in choosing each $q_i \in (\Z/p^c)^\times$. We do not repeat the details but summarize the output: we obtain inductively the following:
\begin{itemize}
\item the elements $\psi_1, \ldots, \psi_r$ of an $\Fp$-basis of $H^1(\gal{F, \mc{T}}, \br(\fgder)^*)$;
\item a collection $(T_i, \alpha_i, X_{\alpha_i})_{i=1}^r$ of split maximal tori, roots, and root vectors; 
\item congruence classes $q_i \in \ker((\Z/p^c)^\times \to (\Z/p^D)^\times)$, non-trivial modulo $p^{D+1}$, and where again we can choose any such $q_i$ if $K_{\psi_i}\neq K(\mu_{p^{D+1}})$ but can otherwise only fix $q_i$ modulo reducing to a fixed non-trivial class mod $p^{D+1}$; and, after the choice of $q_i$ specifying the class $g_{L/K, i}$, we further obtain
\item trivial primes $\{v_i\}_{i=1}^r$ with $v_i$ in a
  positive-density \v{C}ebotarev set $\mc{D}_i$  depending on $v_1, v_2, \ldots, v_{i-1}$,\footnote{In \cite{fkp:reldef}, $\mc{D}_i$ is denoted $\mc{D}_i(v_{i-1})$ to indicate this dependence.} and decomposition groups at each of these primes.
\end{itemize}
These \v{C}ebotarev sets have the property that for all $v \in \mc{D}_i$, $N(v) \equiv q_i \pmod{p^c}$, $\psi_i(\sigma_v) \not \in (\Fp X_{\alpha_i})^\perp$, and $\sigma_v= g_{L/K, i}$. The $\psi_i$ are inductively obtained by defining $L_{v_i}= \{\phi \in H^1(\gal{F_{v_i}}, \br(\fgder)): \phi(\tau_{v_i}) \in \Fp X_{\alpha_i}\}$, checking (with the Greenberg--Wiles formula) that $H^1_{L_{v_1}^\perp, \ldots, L_{v_{i-1}}^\perp}(\gal{F, \mc{T}}, \br(\fgder)^*)=\{\psi \in H^1(\gal{F, \mc{T}}, \br(\fgder)^*): \psi|_{\gal{F_{v_j}}} \in L_{v_j}^\perp, j=1, \ldots, i-1\}$ has dimension $r-i+1$, and, provided $r > i-1$, taking $\psi_i$ to be a non-zero element of this space. It is shown in \textit{loc. cit.} that by choosing any elements $Y_1, \ldots, Y_r \in \bigoplus_{w \in \mc{T}} H^1(\gal{F_w}, \br(\fgder))$ such that $Y_i$ is in the image (under restriction) of $H^1_{L_{v_i}}(\gal{F, \mc{T} \cup v_i}, \br(\fgder))$ but not in $\im(\Psi_{\mc{T}})$, then $Y_1, \ldots, Y_r$ span $\coker(\Psi_{\mc{T}})$. 

        The
      second part of the argument of \cite[Proposition 5.9]{fkp:reldef} replaces the \v{C}ebotarev sets $\mc{D}_i$ with easier to handle \v{C}ebotarev sets
      $\mc{C}_i$ containing $v_i$ and shows that for all
      $v \in \mc{C}_i$:
  \begin{itemize}
  \item $N(v) \equiv q_i \pmod{p^c}$;
  \item the image of $\sigma_v$ in $\Gal(L/K)$ is $g_{L/K, i}$;
  \item the image of $H^1_{L_v}(\gal{F, \mc{T} \cup v}, \br(\fgder)) \to \bigoplus_{w \in \mc{T}} H^1(\gal{F_w}, \br(\fgder))$ equals the image of $H^1_{L_{v_i}}(\gal{F, \mc{T} \cup v_i}, \br(\fgder))$ (here, again, $L_v$ is the classes such that $\phi(\tau_v) \in \Fp X_{\alpha_i}$);
  \item in particular, 
  there is a class $h^{(v)} \in H^1_{L_v}(\Gamma_{F, {\mc{T}} \cup v}, \br(\fgder))$ such that $h^{(v)}|_{\mc{T}}= Y_i$ and $h^{(v)}(\tau_v) \in \Fp X_{\alpha_i} \setminus 0$. 
  \end{itemize}
 To orient the reader, we recall that $\mc{C}_i$ is a \v{C}ebotarev condition in $LK_{\psi_i} \prod_{k=1}^{r-1} K_{\omega_{i, k}}(\mu_{p^c})/F$, where $\{\omega_{i, k}\}_{k=1}^{r-1}$ is an $\Fp$-basis of $H^1_{L_{v_i}^\perp}(\gal{F, {\mc{T}}}, \br(\fgder)^*)$ (non-emptiness of the condition comes from knowing it is satisfied by the original $v_i$ rather than from having a linear disjointness result for the $K_{\psi_i}$ and $K_{\omega_{i, k}}$). We refer the reader to \textit{loc. cit.} for further details.
 
It remains for us to produce the positive upper-density subset $\mf{l}_i \subset \mc{C}_i$ for which the additional conditions (arising from $M_{X_{\alpha_i}}$ and the value at $\tau_{t_0}$) on $h^{(v)}$ hold. We then consider for any such $v \in \mc{C}_i$ the result of applying Lemma \ref{minimizeSha} to the module $M_{X_{\alpha_i}}$ and find that $H^1_{L_{v,{X_{\alpha_i}}}}(\Gamma_{F, {\mc{T}} \cup v}, M_{X_{\alpha_i}})/H^1(\Gamma_{F, {\mc{T}}}, M_{X_{\alpha_i}})$ is one-dimensional. Let $\phi^{(v)}$ be the image under the map
\[
H^1(\Gamma_{F, {\mc{T}} \cup v}, M_{X_{\alpha_i}}) \to H^1(\Gamma_{F, {\mc{T}} \cup v}, \br(\fgder))
\]
of an element spanning the above quotient. We may further assume that $\phi^{(v)}(\tau_{t_0})=X_{\alpha_i}$. Indeed, we can first apply the Lemma by adding the prime $t_0$ to ${\mc{T}} \setminus t_0$ to find a class in $H^1(\gal{F, {\mc{T}}}, M_{X_{\alpha_i}})$ that maps $\tau_{t_0}$ to $X_{\alpha_i}$; then we can apply the Lemma to $({\mc{T}} \setminus t_0) \cup v$ to find a class in $H^1(\gal{F, ({\mc{T}} \setminus t_0)\cup v}, M_{X_{\alpha_i}})$ that maps $\tau_v$ to $X_{\alpha_i}$; and finally we add these classes to obtain the desired $\phi^{(v)}$. Note that $\phi^{(v)}$ is still non-zero: if it were a coboundary, then for some $m \in \fgder$ we would have $\phi^{(v)}(\tau_v)= \tau_v \cdot m -m=0$, a contradiction. Rescaling $\phi^{(v)}$, we then see that 
\[
\phi^{(v)}|_{\mc{T}}-h^{(v)}|_{\mc{T}} \in \im(\Psi_{\mc{T}}),
\]
so for every $v \in \mc{C}_i$, $\phi^{(v)}|_{\mc{T}}$ is not in
$\im(\Psi_{\mc{T}})$. Since $\bigoplus_{w \in \mc{T}} H^1(\gal{F_w}, \br(\fgder))$ is finite,
there is a positive upper-density subset $\mf{l}_i \subset \mc{C}_i$ where $\phi^{(v)}|_{\mc{T}}$ is independent of $v \in \mf{l}_i$. The first part of the Proposition follows, where we now use these $\phi^{(v)}$ and $Y_i:= \phi^{(v)}|_{\mc{T}}$ for $v \in \mf{l}_i$ in place of the classes $h^{(v)}$ produced above by the argument of \cite[Proposition 5.9]{fkp:reldef}. We observe, as recalled above from \textit{loc. cit.}, that for any $v_i' \in \mc{C}_i$ and any elements $(Y_i')_{i=1}^r$ of $\bigoplus_{v \in \mc{T}} H^1(\gal{F_v}, \br(\fgder))$ chosen with $Y_i' \in \im(H^1_{L_{v_i'}}(\gal{F, \mc{T} \cup v_i'}, \br(\fgder)) \to \bigoplus_{v \in \mc{T}} H^1(\gal{F_v}, \br(\fgder)) \setminus \im(\Psi_{\mc{T}})$, the image of $\{Y_i'\}_{i=1}^r$ spans $\mr{coker}(\Psi_{\mc{T}}$). Thus by construction our $Y_i$ still span $\coker(\Psi_{\mc{T}})$. (In the conclusion of the Proposition, we have used the notation $h^{(v)}$ for these modified classes $\phi^{(v)}$ for ease of comparison with \textit{loc.~cit.})

The analogue of \cite[Lemma 5.11]{fkp:reldef} follows similarly, but one point of that proof was unnecessarily phrased using the (now discarded) semisimplicity assumption on $\br(\fgder)$, so we explain the easy modification. We must show that there is a class $\psi \in H^1(\Gamma_{F, {\mc{T}}}, \br(\fgder)^*)$ such that $\psi(\Gamma_K)$ is not contained in $(\Fp Z)^\perp$. We now use the fact that we previously enlarged ${\mc{T}}$ as in Part (3) of Lemma \ref{minimizeSha}. Fix any vector $\lambda \not \in (\Fp Z)^\perp$, and let $t$ be any of the ``excessive" primes $t_b$ added to ${\mc{T}}$.
Then $H^1_{L_{t, \lambda}}
(\Gamma_{F, ({\mc{T}} \setminus t) \cup t}, M_{\lambda})$ contains an element $\psi$ such that $\psi(\tau_{t})= \lambda$. The argument of the previous paragraph shows that the image of $\psi$ in $H^1(\Gamma_{F, {\mc{T}}}, \br(\fgder)^*)$ is still non-zero, and clearly the image of the resulting cocycle is not contained in $(\Fp Z)^\perp$. From here the argument proceeds as in the above modification to \cite[Proposition 5.9]{fkp:reldef}.
\end{proof}

Recall that $\{e_b^*\}_{b \in B}$ is a fixed $k$-basis of $\br(\fgder)^*$. By Lemma \ref{minimizeSha}, there is a class 
\[
\theta_b \in H^1(\gal{F, {\mc{T}} \setminus (t_0 \cup \{t_{b'}\}_{b' \in B \setminus b})}, \Fp[\gal{F}] \cdot e_b^*)
\] 
such that $\theta_b(\tau_{t_b})=e_b^*$ (we can apply the Lemma with ${\mc{T}} \setminus (t_0 \cup \{t_{b'}\}_{b' \in B \setminus b})$ in place of the Lemma's ${\mc{T}}$, and $t_b$ in place of $v$, since the $t_{b'}$ and $t_0$ were introduced after arranging the hypotheses of the Lemma). By the same result, for any trivial prime $v \not \in {\mc{T}}$, there is a cocycle 
\[
\theta_b^{(v)} \in H^1(\Gamma_{F, ({\mc{T}} \setminus (t_0 \cup \{t_{b'}\}_{b' \in B}) \cup v}, \Fp[\Gamma_F]\cdot e_b^*)
\] 
such that $\theta_b^{(v)}(\tau_v)=e_b^*$. We then set 
\[
\eta_b^{(v)}= \theta_b+ \theta_b^{(v)} \in H^1(\gal{F, {\mc{T}} \cup v \setminus (t_0 \cup \{t_{b'}\}_{b' \neq b}}, \Fp[\gal{F}]\cdot e_b^*),
\] 
so that $\eta_b^{(v)}(\tau_v)=\eta_b^{(v)}(\tau_{t_b})= e_b^*$. 
\begin{lemma}\label{etadisjoint}
As the trivial prime $v$ and the indices $b \in B$ vary, the fixed fields $K_{\eta^{(v)}_b}$ are strongly linearly disjoint over $K$ (see \cite[Notation 5.1]{fkp:reldef} for the terminology). They are moreover strongly linearly disjoint from $K_{\infty}$ over $K$.
\end{lemma}
\begin{proof}
First fix some $b_0 \in B$ and consider any composite $K^{(v)}_{\neq b_0}$ of fields $K_{\eta^{(v)}_b}$ for fixed $v$ but possibly varying $b \neq b_0$. Consider the intersection $L$ of $K^{(v)}_{\neq b_0}$ with $K_{\eta^{(v)}_{b_0}}$; we claim that $L=K$. By induction we can assume that the fields $K_{\eta^{(v)}_b}$ for $b \neq b_0$ are strongly linearly disjoint over $K$, so any non-trivial intersection $L$ will yield for some $b \neq b_0$ a non-zero composite map of $\Fp[\gal{F}]$-modules
\[
\Gal(K_{\eta^{(v)}_b}/K) \to \prod_{b' \neq b_0} \Gal(K_{\eta^{(v)}_{b'}}/K) \xleftarrow{\sim} \Gal(K^{(v)}_{\neq b_0}/K) \onto \Gal(L/K).
\]
Since the image $\Fp[\gal{F}]\cdot e_b^*$ of $\eta^{(v)}_b$ is spanned by $\eta^{(v)}_b(\tau_{t_b})$, the restriction of $\tau_{t_b}$ to $L$ must be non-trivial if $L \neq K$; but $L$ is contained in $K_{\eta^{(v)}_{b_0}}$, and by construction the latter field is unramified above $t_{b}$.

Next we vary $v$ and consider the intersection $L$ of some composite $K^{(v)}=\prod_{b \in B} K_{\eta^{(v)}_b}$ with some composite $K^{(\neq v)}$ of fields $K_{\eta^{(v')}_b}$ where both $v' \neq v$ and $b \in B$ are allowed to vary. By the previous paragraph the $K_{\eta^{(v)}_b}$ are strongly linearly disjoint over $K$, so any non-trivial intersection $L$ leads to some non-trivial composite
\[
\Gal(K_{\eta^{(v)}_b}/K) \to \prod_{b' \in B} \Gal(K_{\eta^{(v)}_{b'}}/K) \xleftarrow{\sim} \Gal(K^{(v)}/K) \onto \Gal(L/K). 
\]
As $\tau_v$ generates $\Gal(K_{\eta^{(v)}_b}/K)$ as $\Fp[\gal{F}]$-module, we as before deduce that $L$ must be ramified above $v$, contradicting the fact that $L$ is a subfield of $K^{(\neq v)}$.

Thus, all the fields $K_{\eta^{(v)}_b}$ as both $v \not \in {\mc{T}}$ and $b \in B$ vary are strongly linearly disjoint over $K$. Finally, their composite is linearly disjoint from $K_{\infty}$ over $K$. Else, letting $L$ be the intersection, there would be some index $v, b$ with a non-zero composite (defined as in the last two paragraphs) $\Gal(K_{\eta^{(v)}_b}/K) \to \Gal(L/K)$, implying $L$ would be ramified above $v$ (and $t_b$); but $K_{\infty}/K$ is unramified away from primes above $p$.

\end{proof}
We now apply these constructions of auxiliary cocycles to construct modulo $\vpi^2$ lifts of $\br$ with prescribed local properties. By the vanishing $\Sha^2_{\mc{T}}(\gal{F, {\mc{T}}}, \br(\fgder))=0$, we produce an initial lift $\rho_2 \colon \Gamma_{F, {\mc{T}}} \to G(\mc{O}/\vpi^2)$ with multiplier $\mu$, and we then fix target local lifts $(\lambda_w)_{w \in {\mc{T}}}$ (of multiplier $\mu$) satisfying the conditions in \cite[Construction 5.6]{fkp:reldef} (in particular, for $w \in {\mc{T}} \setminus {\mc{S}}$, $\lambda_w$ is unramified, and, enlarging ${\mc{T}}$ if necessary, the collection of $\lambda_w(\sigma_w)$ generates $\wh{G^{\mr{der}}}(\mc{O}/\vpi^2)$). These differ from the restrictions $\rho_2|_w$ by a collection of cocycles $z_{\mc{T}}=(z_w)_{w \in {\mc{T}}} \in \bigoplus_{w \in {\mc{T}}} H^1(\Gamma_{F_w}, \br(\fgder))$. 

We can now give the ``doubling argument" analogous to \cite[Proposition 5.12]{fkp:reldef}. In contrast to \textit{loc.~cit.}, we work only with the $k$-linear duality pairings: thus we write $\langle \cdot, \cdot \rangle \colon \fgder \times (\fgder)^* \to k$ for the canonical $k$-linear pairing, and for any prime $x$ of $F$ we write $\langle \cdot, \cdot \rangle_x \colon H^1(\gal{F_x}, \br(\fgder)) \times H^1(\gal{F_x}, \br(\fgder)^*) \to k$ for the $k$-linear Tate local duality pairing. We will use the explicit calculation in \cite[Lemma 3.9]{fkp:reldef} of the local duality pairing at trivial primes.

In Proposition \ref{doublingprop} and Theorem \ref{p^Nlift} we will use sets of auxiliary primes constructed from Proposition \ref{5.9-5.11sub} (in the present section we in fact only use the case $L=K$, $c=D+1$). The primes $v$ of $F$ produced by Proposition \ref{5.9-5.11sub} come with a specification of a decomposition group, which yields in particular a specified place of $K$ (and an extension to $L$).
\begin{prop}\label{doublingprop}
There is a finite indexing set $N$, and there is, for each $n \in N$, a positive upper-density set $\mf{l}_n$ of trivial primes of $F$, with the following properties. Fix any $2|N|$-tuple $(A_n, A_n')_{n \in N}$ of elements of $\wh{G^{\mr{der}}}(\mc{O}/\vpi^2)$. Then there is a finite $2|N|$-tuple of trivial primes $\mc{Q}= (v_n, v_n')_{n \in N}$ disjoint from $\mc{T}$ and having $\{v_n, v_n'\} \subset \mf{l}_n$ for all $n \in N$, and a class $h \in H^1(\gal{F, \mc{T} \cup \mc{Q}}, \br(\fgder))$ such that
\begin{itemize}
\item $h|_{\mc{T}} = z_{\mc{T}}$.
\item For all $n \in N$ there is a pair $(T_n, \alpha_n)$ of a split maximal torus $T_n$ of $G^0$ and a root $\alpha_n \in \Phi(G^0, T_n)$ such that $(1+\vpi h)\rho_2(\tau_{v_n})=u_{\alpha_n}(X_n)$ for some $\alpha_n$-root vector $X_n$, and likewise $(1+\vpi h)\rho_2(\tau_{v'_n})=u_{\alpha_n}(X_n)$; and such that
\[
(1+\vpi h)\rho_2(\sigma_{v_n})= A_n \cdot z_n,
\]
where $z_n$ is in $Z_{G^0}(\mc{O}/\vpi^2) \cap \wh{G}(\mc{O}/\vpi^2)$ (and is determined by $\mu(\sigma_{v_n})$), and similarly 
\[
(1+\vpi h)\rho_2(\sigma_{v'_n})= A'_n \cdot z'_n,
\]
where $z'_n$ is also in $Z_{G^0}(\mc{O}/\vpi^2) \cap \wh{G}(\mc{O}/\vpi^2)$.
\end{itemize} 
\end{prop}
\begin{proof}
We have fixed a $k$-basis $\{e_{b}^*\}_{b \in B}$ of $\br(\fgder)^*$ (the $k$-dual). 
Fix a finite set $N_{\mr{span}}$ indexing root vectors $\{X_{\alpha_n}\}_{n \in N_{\mr{span}}}$ with respect to tori $\{T_n\}_{n \in N_{\mr{span}}}$ such that
\begin{equation}\label{spanningeq}
\sum_{n \in N_{\mr{span}}} \Fp[\Gamma_F] X_{\alpha_n}= \fgder.
\end{equation}
As in the discussion following \cite[Lemma 5.11]{fkp:reldef}, we observe that such a collection exists, since for any proper subspace $U$ of $\fgder$, there is a root vector not in $U$: see the start of the proof of \cite[Proposition 5.9]{fkp:reldef}, where this is reduced, using $p \gg_G 0$, to irreducibility of the simple factors of $\fgder$ as $k[G(k)]$-modules. 
By the second part of Proposition \ref{5.9-5.11sub}, there is for each $n \in N_{\mr{span}}$ a positive upper-density set $\mf{l}_n$ of trivial primes, a non-trivial congruence class $q_n \in (\Z/p^{D+1})^\times$ that is trivial modulo $p^D$, and for each $v \in \mf{l}_{n}$ a class $h^{(v)} \in H^1(\Gamma_{F, {\mc{T}} \cup v}, \br(\fgder))$ such that $h^{(v)}|_{\mc{T}}=Y_n$ is independent of $v \in \mf{l}_{n}$, $h^{(v)}(\tau_v)$ spans $\Fp X_{\alpha_n}$, $h^{(v)}(\tau_{t_0})$ spans $\Fp X_{\alpha_n}$, and $h^{(v)}$ is the image of an $M_{X_{\alpha_n}}= \Fp[\Gamma_F]\cdot X_{\alpha_n}$-valued cocycle. Using the first part of Proposition \ref{5.9-5.11sub}, we also produce a finite set $\{Y_n\}_{n \in N_{\mr{coker}}} \subset \bigoplus_{w \in {\mc{T}}} H^1(\Gamma_{F_w}, \br(\fgder))$ that spans $\mr{coker}(\Psi_{\mc{T}})$ over $\Fp$, and, for each $n \in N_{\mr{coker}}$, a root vector $X_{\alpha_n}$ with respect to a maximal torus $T_n$, a non-trivial congruence class $q_n \in (\Z/p^{D+1})^\times$, trivial modulo $p^D$, a positive upper-density set $\mf{l}_n$ of trivial primes, with all $v \in \mf{l}_n$ satisfying $N(v) \equiv q_n \pmod{p^{D+1}}$, and for each $v \in \mf{l}_n$ a class $h^{(v)} \in H^1(\Gamma_{F, {\mc{T}} \cup v}, \br(\fgder))$ such that $h^{(v)}|_{\mc{T}}=Y_n$, $h^{(v)}(\tau_v)$ and $h^{(v)}(\tau_{t_0})$ both span $\Fp X_{\alpha_n}$, and $h^{(v)}$ is the image of an $M_{X_{\alpha_n}}$-valued cocycle. As in the discussion following \cite[Lemma 5.11]{fkp:reldef}, we obtain a class $h^{\mr{old}} \in H^1(\Gamma_{F, {\mc{T}}}, \br(\fgder))$ and (perhaps rescaling some of the $h^{(v)}$) a subset $N \subset N_{\mr{span}}  \sqcup N_{\mr{coker}}$, containing $N_{\mr{span}}$ but where possibly some unnecessary elements of $N_{\mr{coker}}$ have been discarded, with the relation
\[
z_{\mc{T}}= h^{\mr{old}}|_{\mc{T}} + \sum_{n \in N} h^{(v_n)}|_{\mc{T}}
\]
for all tuples $\un{v}=(v_n)_{n \in N} \in \prod_{n \in N} \mf{l}_n$. For the reader's orientation, we note that as in \textit{loc. cit.}, we guarantee that $N$ contains $N_{\mr{span}}$ by applying Proposition \ref{5.9-5.11sub} not to $z_{\mc{T}}$ itself but rather to the class $z'_{\mc{T}}= z_{\mc{T}}- \sum_{n \in N_{\mr{span}}} Y_n$. We for any pairs $\un{v}, \un{v}' \in \prod_{n \in N} \mf{l}_n$ consider classes
\[
h= h^{\mr{old}}- \sum_{n \in N} h^{(v_n)} +2 \sum_{n \in N} h^{(v_n')} \in H^1(\Gamma_{F, {\mc{T}} \cup \{v_n\} \cup \{v_n'\}}, \br(\fgder));
\]
note that these still satisfy $h|_{\mc{T}}= z_{\mc{T}}$ and the requisite inertial conditions that for all $n \in N$ and any $w \in \mf{l}_n$, $h(\tau_w)$ spans $\Fp X_{\alpha_n}$. We must show there is a pair $\un{v}$, $\un{v}'$ such that $(1+\vpi h)\rho_2$ also has some pre-specified behavior, corresponding to the tuples $(A_n, A'_n)_{n \in N}$ of the theorem statement, at the Frobenius elements $\sigma_w$ for primes $w \in \{v_n\} \cup \{v_n'\}$.

For each $n \in N$ and $v_n \in \mc{C}_n$ (in particular, for $v_n \in \mf{l}_n$), we have specified a unique place $v_{n, K}$ of $K$ above $v_n$. We let $\mf{l}_{n, K}$, $\mc{C}_{n, K}$, etc., denote the set of such places; these are still sets of primes of $K$ of positive (upper) density, since they are all split over $F$. In the limiting argument that follows, it is convenient to work with places and densities in $K$, although for notational simplicity we will not burden each $v_n$ or $\un{v}$ with an additional subscript to indicate the place of $K$. We further restrict to a positive upper-density subset $\mf{l} \subset \prod_{n \in N} \mf{l}_{n, K}$ such that the $N$-tuples $(\sum_{n \in N} h^{(v_n)}(\sigma_{v_m}))_{m \in N}$, $(h^{\mr{old}}(\sigma_{v_m}))_{m \in N}$, $(h^{(v_n)}(\tau_{v_n}))_{n \in N}$, and $(\rho_2(\sigma_{v_n}) \mod{Z_{G^0}})_{n \in N}$ are all independent of $\un{v} \in \mf{l}$; this is possible since the quantities in question take only finitely many values. As in \cite[Proposition 5.12]{fkp:reldef}, this restriction reduces us to showing that for any two fixed $N$-tuples $(C_m)_{m \in N}$ and $(C'_m)_{m \in N}$, there exist $\un{v}, \un{v}' \in \mf{l}$ such that 
\begin{align}\label{doubling}
&\sum_{n \in N} h^{(v'_n)}(\sigma_{v_m})= C_m, \\ \label{doubling2}
&\sum_{n \in N} h^{(v_n)}(\sigma_{v'_m})= C'_m,
\end{align}
for all $m \in N$. 
For fixed $\un{v}$ and each $m \in N$, the equality $\sum_{n \in N} h^{(v_n)}(\sigma_w)= C'_m$ of Equation (\ref{doubling2}) is easily assured by a \v{C}ebotarev condition $\mf{w}_m$ on primes $w$ of $K$ split over $F$: the fixed fields $K_{h^{(v)}}$ are strongly linearly disjoint over $K$ as $v$ varies, as follows from the construction of Proposition \ref{5.9-5.11sub} and the argument of Lemma \ref{etadisjoint}, and then the claim follows from Equation (\ref{spanningeq}). Moreover we remark (by the same ramification argument as before with the fields $K_{\eta^{(v)}_b}$) that the $K_{h^{(v)}}$ are also all disjoint from $K_{\infty}$ over $K$, so we may assume that all $w \in \mf{w}_m$ satisfy $N(w) \equiv q_m \pmod{p^{D+1}}$.\footnote{Note the argument here is somewhat different from that of \cite{fkp:reldef}, where we have a weaker linear disjointness statement.}

Still fixing $\un{v} \in \mf{l}$, we will now show that there is a positive-density \v{C}ebotarev condition $\mf{l}_{\un{v}}$ on $N$-tuples of primes of $K$ (split over $F$) such that for any $\un{v}' \in \mf{l} \cap \mf{l}_{\un{v}}$, Equation (\ref{doubling}) holds. By global duality, we have for each $m, n$, and $b$ an equality
\[
\langle \eta^{(v_m)}_{b}(\tau_{v_m}), h^{(v_n')}(\sigma_{v_m})\rangle=-\sum_{x \in {\mc{T}}} \langle \eta^{(v_m)}_{b}, h^{(v_n')}\rangle_x- \langle \eta^{(v_m)}_{b}(\sigma_{v'_n}), h^{(v_n')}(\tau_{v'_n})\rangle.
\]
By definition of $\mf{l}$ and the fact that the elements $\{\eta^{(v_m)}_{b}(\tau_{v_m})\}_{b \in B}$ constitute a $k$-basis of $\br(\fgder)^*$, it suffices to show that we can prescribe the values
\[
\sum_{n \in N} \langle \eta^{(v_m)}_{b}(\sigma_{v_n'}), X_n \rangle,
\]
where $X_n$ is the common value (in $\Fp^\times X_{\alpha_n}$) of the $h^{(v_n')}(\tau_{v_n'})$ (for fixed $n$ but varying $\un{v}'$). The fields $K_{\infty}$ and $K_{\eta^{(v_m)}_b}$ are strongly linearly disjoint over $K$ as $m$ and $b$ vary, so the values $\eta_b^{(v_m)}(\sigma_{v_n'})$ may be independently (as $m$ and $b$ vary) specified, by a \v{C}ebotarev condition on $v_n'$, to be anything in $\eta^{(v_m)}_b(\Gamma_K)$, and such that $N(v'_n) \equiv q_n \pmod{p^{D+1}}$. It follows (as in \cite{fkp:reldef}) from Equation (\ref{spanningeq}) that this sum of pairings can be made equal to any desired element of $k$.

We claim that the splitting fields $K_{h^{(v_n)}}$ are strongly linearly disjoint from the $K_{\eta^{(v_m)}_{b}}$ (and as noted before from $K_{\infty}$) subquotient)
over $K$, so that the \v{C}ebotarev condition thus produced 
intersects the previously-produced condition $\prod_{m \in N} \mf{w}_m$ in a positive-density \v{C}ebotarev condition. We check this by imitating our earlier arguments. Namely, consider any composites $K_h$ of fields of the form $K_{h^{(v)}}$ and $K_{\eta}$ of fields of the form $K_{\eta^{(v)}_{b}}$ ($b$ can vary, and in both cases $v$ can vary). As the $K_{h^{(v)}}$ are disjoint as $v$ varies, if $L= K_h \cap K_{\eta}$ properly contains $K$, then there is some $v_0$ for which the map
\[
\Gal(K_{h^{(v_0)}}/K) \to \prod_v \Gal(K_{h^{(v)}}/K) \xleftarrow{\sim} \Gal(K_h/K) \onto \Gal(L/K)
\]
of $\Fp[\gal{F}]$-modules is non-zero. Since $h^{(v_0)}(\tau_{t_0})$ generates the image of $h^{(v_0)}$, $t_0$ must be ramified in $L/K$, contradicting the fact that (by construction) none of the $K_{\eta^{(v)}_b}$ is ramified at $t_0$. Having checked this disjointness, we can define the non-trivial \v{C}ebotarev condition $\mf{l}_{\un{v}}$ on $N$-tuples of places of $K$ split over $F$ to be the intersection of the above-constructed \v{C}ebotarev conditions, so that for all $\un{v} \in \mf{l}$, and all $\un{v}' \in \mf{l} \cap \mf{l}_{\un{v}}$, both Equations (\ref{doubling}) and (\ref{doubling2}) hold. Note that $\mf{l}_{\un{v}}$ is a \v{C}ebotarev condition in $L_{\un{v}}(\mu_{p^{D+1}})/K$, where $L_{\un{v}}$ denotes the composite of the various fields $K_{h^{(v_n)}}$ and $K_{\eta^{(v_n)}_b}$ arising for the tuple $\un{v}$. The \v{C}ebotarev condition defining $\mc{C}_K:= \prod_{n} \mc{C}_{n, K}$ occurs (see the proof of Proposition \ref{5.9-5.11sub}) for each $n$ in an extension $M_n(\mu_{p^{D+1}})$ of $K$ for which all the fields $M= \prod_n M_n$ and $L_{\un{v}}$ as $\un{v}$ varies\footnote{The variation we consider is to allow tuples $\un{v}'$ each of whose entries satisfies $v'_n \not \in \{v_m\}_{m \in N}$; and given a previously-constructed list $\un{v}_1, \ldots, \un{v}_{s}$, we allow $\un{v}'$ such that no $v'_n$ is among the entries of the previous tuples $\un{v}_k$.} are strongly linearly disjoint over $K$. 
Moreover, the restriction of the conditions in any $L_{\un{v}}(\mu_{p^{D+1}})$ and in $(M_n(\mu_{p^{D+1}}))_n$ define the same $N$-tuple of conditions in $K(\mu_{p^{D+1}})$ (namely, cutting out the $N$-tuple of congruence classes $(q_n \pmod {p^{D+1}})_{n \in N}$).

Finally, using the above observations we give the limiting argument, as extended in
\cite{fkp:reldef} from \cite{ramakrishna-hamblen} and \cite{klr},
which addresses the possible incompatibility of the conditions
$\mf{l}$ and $\mf{l}_{\un{v}}$. Suppose that for each member of a
finite set $\{\un{v}_1, \ldots, \un{v}_s\}$, the intersection $\mf{l}
\cap \mf{l}_{\un{v}_k}$ is empty, so that $\mf{l} \setminus
\{\un{v}_1, \ldots, \un{v}_s\}$ is contained in $\mf{l} \cap
\bigcap_{k=1}^s \overline{\mf{l}_{\un{v}_k}}$, and in particular is contained in $\mc{C}_K \cap \cap_{k=1}^s \ov{\mf{l}_{\un{v}_k}}$ (we take complements in the set of $N$-tuples of primes of $K$ that are split over $F$). We may and do assume that for all $n \in N$, no two of the tuples (regarded here as multi-sets) $\un{v}_1, \ldots, \un{v}_s$ have any primes in common, since for each $n$ the subset of $\mf{l}$ consisting of elements $\un{v}'$ for which $v'_n$ is in some pre-established finite list has upper-density zero. It follows as in \cite[Proposition 5.12]{fkp:reldef}, by the disjointness of the $L_{\un{v}_k}$ and the compatibility of the conditions in $L_{\un{v}_k}(\mu_{p^{D+1}})$ with the condition $\mc{C}_K$, that there is some constant $\varepsilon > 0$ (independent of $k$) such that
\[
\delta \left( \mc{C}_K \cap  \cap_{k=1}^s \ov{\mf{l}_{\un{v}_k}} \right) \leq (1- \varepsilon)^s.
\]
Thus we either eventually find some pair $\un{v}, \un{v}' \in \mf{l}$ with $\un{v}' \in \mf{l}_{\un{v}}$, or we can let $s$ tend to infinity and thus contradict the positive upper-density of $\delta^+(\mf{l})$. 

\end{proof}
We will specify the desired values $(1+\vpi h)\rho_2(\sigma_w)$ for $w \in {\mc{Q}}$ in the application in Theorem \ref{p^Nlift}; of course, we will do this to ensure that $(1+\vpi h)\rho_2|_{\gal{F_w}}$ belongs to the local lifting condition $\Lift^{\mu, \alpha_w}_{\br|_{\gal{F_w}}}(\mc{O}/\vpi^2)$ of \cite[Definition 3.1, Lemma 3.2]{fkp:reldef}, which allows unipotent ramification along the $\alpha_w$ root space.

\section{Lifting mod $\vpi^N$}\label{modp^Nsection}

We will in this section impose additional hypotheses on the Galois modules $\br(\fgder)$ and $\br(\fgder)^*$:
\begin{assumption}\label{modp^Nhyp}
In addition to Assumption \ref{generalhyp}, further assume that
\begin{itemize}
\item $\br(\fgder)$ does not contain the trivial representation as a submodule.
\item There is no surjection of $\Fp[\gal{F}]$-modules $\br(\fgder) \onto W$ for some $\Fp[\gal{F}]$-module subquotient $W$ of $\br(\fgder)^*$.
\end{itemize}
\end{assumption}

\begin{rmk}
Implicit in this second assumption is of course that the mod $p$ cyclotomic
character $\kbar$ is non-trivial, i.e. $F$ does not contain
$\mu_p$. Note that in the function field case, where the constant
field of $F$ is $\mathbb{F}_q$, this forces $q \not \equiv 1 \pmod
p$.
\end{rmk}
We note that the second condition in Assumption \ref{modp^Nhyp} is
implied by the following:
\begin{itemize}
\item There is no $k[\gal{F}]$-module surjection $\br(\fgder)^{\sigma} \onto V$ for some $\sigma \in \Aut(k)$ and $k[\gal{F}]$-subquotient $V$ of $\br(\fgder)^*$.
\end{itemize}
(Indeed, suppose there were then some $\Fp[\gal{F}]$-quotient $\br(\fgder) \onto W$. It gives rise to a $k[\gal{F}]$-quotient 
\[
\bigoplus_{\sigma \in \Aut(k)} \br(\fgder)^{\sigma} \cong \br(\fgder) \otimes_{\Fp} k \onto W \otimes_{\Fp} k,
\]
and $W \otimes_{\Fp} k$ is a $k[\gal{F}]$-subquotient of $\br(\fgder)^* \otimes_{\Fp} k \cong \bigoplus_{\tau \in \Aut(k)} \br(\fgder)^{*, \tau}$, which would yield a $k[\gal{F}]$-quotient $\br(\fgder)^{\sigma \tau^{-1}} \onto V$ for some $k[\gal{F}]$-subquotient $V$ of $\br(\fgder)^*$.) We will apply this version of the criterion in Lemma \ref{GL2lemma} below.

We will inductively produce mod $\vpi^n$ lifts of $\br$. For the step in which we pass from a mod $\vpi^{n-1}$ lift $\rho_{n-1}$ to a mod $\vpi^n$ lift, we will use Proposition \ref{5.9-5.11sub} in the case $L= K(\rho_{n-1}(\fgder))$ and $c= \max \{D+1, \lceil \frac{n}{e} \rceil \}$. To that end, we will need the following linear disjointness result:
\begin{lemma}\label{disjointmodp^N}
Let $\br$ satisfy Assumption \ref{modp^Nhyp}, and suppose that we have inductively constructed a lift 
\[
\rho_{n-1} \colon \Gamma_{F, {\mc{T}}_{n-1}} \to G(\mc{O}/\vpi^{n-1})
\] 
of $\br$ for some finite set of primes ${\mc{T}}_{n-1}$ containing ${\mc{T}}$.
Assume moreover that the image $\im(\rho_{n-1})$ contains $\widehat{G}^{\mr{der}}(\mc{O}/\vpi^{n-1})$. 
Then the field $L=K(\rho_{n-1}(\fgder))$ is linearly disjoint over $K$ from the composite of $K_{\infty}$ with any composite of fields $K_{\psi}$, $\psi \in H^1(\gal{F, {\mc{T}}_{n-1}}, \br(\fgder)^*)$.

\end{lemma}
\begin{proof}
By assumption, $\Gal(K(\rho_{n-1}(\fgder))/K)$ is isomorphic to $\wh{G^{\mr{der}}}(\mc{O}/\vpi^{n-1})$, so any of its abelian quotients is a quotient of the abelianization $\wh{G^{\mr{der}}}(\mc{O}/\vpi^2) \cong \br(\fgder)$ of $\wh{G^{\mr{der}}}(\mc{O}/\vpi^{n-1})$ (see the proof of \cite[Theorem 5.14]{fkp:reldef}). A non-trivial (properly containing $K$) intersection of $L$ with any composite of $K_{\infty}$ with a composite of fields $K_{\psi}$, $\psi \in H^1(\gal{F, {\mc{T}}_{n-1}}, \br(\fgder)^*)$, thus yields a surjection of $\Fp[\gal{F}]$-modules from $\br(\fgder)$ to some subquotient of $\br(\fgder)^*$ or (by Lemma \ref{abelian}) $\Z/p$. Assumption \ref{modp^Nhyp} excludes both of these possibilities. (To exclude the latter, note that $\br(\fgder)$ has no trivial \textit{quotient}, since by assumption $H^0(\Gamma_F, \br(\fgder))=0$, and $\br(\fgder)$ is self-dual via the Killing form under our assumption $p \gg_G 0$, for which see \S 2 and \cite[Assumption 2.1]{fkp:reldef}.)

\end{proof}
We will make reference in the following theorem to the spaces of local lifts described in \cite[\S 3]{fkp:reldef}; see especially \cite[Definitions 3.1, 3.4]{fkp:reldef} for the notation.
\begin{thm}\label{p^Nlift}
Let $p \gg_G 0$. Assume that $\br \colon \gal{F, {\mc{S}}} \to G(k)$ satisfies Assumption \ref{generalhyp} and Assumption \ref{modp^Nhyp}. Fix a lift $\mu$ of the multiplier character $\bar{\mu}= \br \pmod{G^{\mr{der}}}$. Moreover assume that for all $v \in {\mc{S}}$ there are lifts $\rho_v \colon \gal{F_v} \to G(\mc{O})$ with multiplier $\mu$. Let ${\mc{T}} \supset {\mc{S}}$ be the set constructed in the discussion preceding Proposition \ref{5.9-5.11sub}.

Then there exists a sequence of finite sets of primes of $F$, ${\mc{T}} \subset {\mc{T}}_2 \subset {\mc{T}}_3 \subset \cdots {\mc{T}}_n \subset \cdots$, and for each $n \geq 2$ a lift $\rho_n \colon \gal{F, {\mc{T}}_n} \to G(\mc{O}/\vpi^n)$ of $\br$ with multiplier $\mu$, such that $\rho_{n} = \rho_{n+1} \pmod{\vpi^n}$ for all $n$. This system of lifts $(\rho_n)_{n \geq 1}$ satisfies the following properties:
\begin{enumerate}
\item If $w \in {\mc{T}}_n \setminus {\mc{S}}$ is ramified in $\rho_n$, then there is a split maximal torus and root $(T_w, \alpha_w)$ such that $\rho_n(\sigma_w) \in T_w(\mc{O}/\vpi^n)$, $\alpha_w(\rho_n(\sigma_w)) \equiv N(w) \pmod{\vpi^n}$, and $\rho_n|_{\gal{F_w}} \in \Lift_{\br}^{\mu, \alpha_w}(\mc{O}/\vpi^n)$; in addition, one of the following two properties holds:
\begin{enumerate}
\item For some $s \leq eD$, $\rho_s(\tau_w)$ is a non-trivial element of $U_{\alpha_w}(\mc{O}/\vpi^s)$ (in particular, $s \leq n$), and for all $n' \geq s$, $\rho_{n'}|_{\gal{F_w}}$ is $\wh{G}(\mc{O})$-conjugate to the reduction modulo $\vpi^{n'}$ of a fixed lift $\rho_w \colon \gal{F_w} \to G(\mc{O})$ of $\rho_s|_{\gal{F_w}}$. We may choose this $\rho_w$ to be constructed as in \cite[Lemma 3.7]{fkp:reldef} to be a formally smooth point of the generic fiber of the local lifting ring of $\br|_{\gal{F_w}}$.
\item 
For $s=eD$, $\rho_s|_{\gal{F_w}}$ is trivial mod center (in particular, $s<n$), while $\alpha_w(\rho_{s+1}(\sigma_w)) \equiv N(w) \not \equiv 1 \pmod{\vpi^{s+1}}$, and $\beta(\rho_{s+1}(\sigma_w)) \not \equiv 1 \pmod{\vpi^{s+1}}$ for all roots $\beta \in \Phi(G^0, T_w)$.
\end{enumerate}
\item For all $v \in {\mc{S}}$, $\rho_n|_{\gal{F_v}}$ is strictly equivalent to $\rho_v \pmod{\vpi^n}$.
\item The image $\rho_n(\gal{F})$ contains $\wh{G^{\mr{der}}}(\mc{O}/\vpi^n)$.
\end{enumerate}

\end{thm}
\begin{proof}
In light of Lemma \ref{disjointmodp^N} and Proposition \ref{5.9-5.11sub}, the argument of Proposition \ref{doublingprop} allows us to argue as in \cite[Theorem 5.14]{fkp:reldef}. Since our assumptions differ from those of \textit{loc.~cit.}---and in fact the technical improvements of \S \ref{doublingsection} allow us to simplify the argument somewhat---we will repeat the proof. We will inductively lift $\br$ to a $\rho_n \colon \gal{F, {\mc{T}}_n} \to G(\mc{O}/\vpi^n)$ satisfying the conclusions of the theorem, at each stage enlarging the ramification set ${\mc{T}}_n \supset {\mc{T}}$. Thus suppose for some $n \geq 2$ we have already constructed a lift $\rho_{n-1} \colon \gal{F, {\mc{T}}_{n-1}} \to G(\mc{O}/\vpi^{n-1})$ as in the theorem. For each $w \in {\mc{T}}_{n-1} \setminus {\mc{S}}$ at which $\rho_{n-1}$ is ramified, we are given a torus and root $(T_w, \alpha_w)$ as in the theorem statement. In what follows, we will tacitly allow ourselves to change this pair to a $\wh{G}(\mc{O})$-conjugate without changing the notation; see \cite[Remark 5.15]{fkp:reldef}. There are no local obstructions to lifting $\rho_{n-1}$, and we fix local lifts $\lambda_w \colon \gal{F_w} \to G(\mc{O}/\vpi^n)$ of $\rho_{n-1}|_{\gal{F_w}}$ as follows: 
\begin{itemize}
\item If $w \in {\mc{S}}$, by assumption $\rho_{n-1}|_{\gal{F_w}}$ is $\wh{G}(\mc{O})$-conjugate to the given lift $\rho_w \pmod{\vpi^{n-1}}$, so we can take $\lambda_w$ to be $\wh{G}(\mc{O})$-conjugate to $\rho_w \pmod{\vpi^n}$.
\item If $w \in {\mc{T}}_{n-1} \setminus {\mc{S}}$, and $\rho_{n-1}|_{\gal{F_w}}$ is unramified, then let $\lambda_w$ be any (multiplier $\mu$) unramified lift.
\item If $w \in {\mc{T}}_{n-1} \setminus {\mc{S}}$, and $\rho_{n-1}|_{\gal{F_w}}$ is ramified, then by the inductive hypothesis we either have:
\begin{itemize}
\item As in Case (a) of the Theorem statement, for some $s \leq eD$, $\rho_s(\tau_w) \in U_{\alpha_w}(\mc{O}/\vpi^s)$ is non-trivial, and we are given a lift $\wh{G}(\mc{O})$-conjugate to the $\rho_w$ that was introduced at the $s^{th}$ stage in the induction (the one in which the prime $w$ was introduced; we will explain how to choose this $\rho_w$ below when we carry out the heart of the induction step). We may and do replace $\rho_w$ by this suitable $\wh{G}(\mc{O})$-conjugate, and we then take $\lambda_w= \rho_w \pmod{\vpi^n}$.
\item As in Case (b) of the Theorem, $n-1 \geq eD+1$, $\rho_{eD}|_{\gal{F_w}}$ is trivial mod center, and $\rho_{eD+1}|_{\gal{F_w}}$ is in general position as described in Case (b). Then we take $\lambda_w$ to be any lift still satisfying all the conditions in Case(b) (and the general requirements  of Case (1)). This is easily seen to be possible by arguing as in the proof of \cite[Lemma 5.13]{fkp:reldef}, again using the $p \gg_G 0$ assumption.
\end{itemize}
\end{itemize}
We also enlarge ${\mc{T}}_{n-1}$ by a finite set of primes split in $K(\rho_{n-1}(\fgder))$ and introduce at these $w$ unramified, multiplier $\mu$, lifts $\lambda_w$ such that the elements $\lambda_w(\sigma_w)$ generate $\ker(G^{\mr{der}}(\mc{O}/\vpi^n) \to G^{\mr{der}}(\mc{O}/\vpi^{n-1}))$. (This is a device for ensuring each lift $\rho_n$ has maximal image; by \cite[Lemma 6.15]{fkp:reldef}, this step only needs to be carried out for finitely many $n$.) We denote this enlarged set by ${\mc{T}}'_{n-1}$.

Since $\Sha^1_{\mc{T}}(\gal{F, {\mc{T}}}, \br(\fgder)^*)=0$, \textit{a fortiori} we see that $\Sha^1_{{\mc{T}}'_{n-1}}(\gal{F, {\mc{T}}'_{n-1}}, \br(\fgder)^*)=0$, so by global duality and the local unobstructedness there is some lift $\rho'_n \colon \gal{F, {\mc{T}}'_{n-1}} \to G(\mc{O}/\vpi^n)$ of $\rho_{n-1}$. We then as before define a class $z_{{\mc{T}}'_{n-1}}=(z_w)_{w \in {\mc{T}}'_{n-1}} \in \bigoplus_{w \in {\mc{T}}'_{n-1}} H^1(\gal{F_w}, \br(\fgder))$ such that $(1+\vpi^{n-1} z_w)\rho'_{n}|_{\gal{F_w}}= \lambda_w$ for all $w \in {\mc{T}}'_{n-1}$. If $z_{{\mc{T}}'_{n-1}}$ lies in the image of some $h \in H^1(\gal{F, {\mc{T}}'_{n-1}}, \br(\fgder))$, we replace $\rho'_{n}$ by $(1+\vpi^{n-1}h)\rho'_{n}= \rho_n$, a lift of $\rho_{n-1}$ such that $\rho_n|_{\gal{F_w}}$ is $\ker(G^{\mr{der}}(\mc{O}/\vpi^n) \to G^{\mr{der}}(\mc{O}/\vpi^{n-1})$-conjugate to the fixed $\lambda_w$, and we are done.

If $z_{{\mc{T}}'_{n-1}}$ does not lie in the image of $H^1(\gal{F, {\mc{T}}'_{n-1}}, \br(\fgder))$, we apply Proposition \ref{5.9-5.11sub} with ${\mc{T}}'_{n-1}$ in place of ${\mc{T}}$ (in the notation of the Proposition), $c= \max \{ D+1, \lceil \frac{n}{e} \rceil \}$, and $L= K(\rho_{n-1}(\fgder))$. Lemma \ref{disjointmodp^N} shows the linear disjointness hypothesis of the Proposition is satisfied. As in the statement of Proposition \ref{5.9-5.11sub}, after the classes $q_i \in \ker((\Z/p^c)^\times \to (\Z/p^D)^\times)$ and the tori and roots $(T_i, \alpha_i)$ are produced, we are allowed to choose the elements $g_{L/K, i} \in \Gal(L/K)$. We do this as follows:
\begin{itemize}
\item If $n-1 \leq eD$, take all $g_{L/K, i}$ to be trivial.  
\item If $n-1 \geq eD+1$, via the isomorphism $\Gal(L/K) \xrightarrow{\sim} \wh{G^{\mr{der}}}(\mc{O}/\vpi^{n-1})$ take $g_{L/K, i}$ to be an element  $t_i \in T_i(\mc{O}/\vpi^{n-1})$ satisfying $\alpha_i(t_i) \equiv q_i \pmod{\vpi^{n-1}}$ that is trivial modulo $\vpi^{eD}$ but in general position modulo $\vpi^{eD+1}$: for all $\beta \in \Phi(G^0, T_i)$, $\beta(t_i) \not \equiv 1 \pmod{\vpi^{eD+1}}$ (note that by construction $q_i$ is non-trivial modulo $\vpi^{eD+1}$, so it is possible to choose such a $t_i$).
\end{itemize}
Proposition \ref{5.9-5.11sub} then produces \v{C}ebotarev sets $\mc{C}_i$ with positive-density subsets $\mf{l}_i$ and classes $h^{(v)}$ for all $v \in \mf{l}_i$; as in Proposition \ref{doublingprop}, we apply this both to produce $\mc{C}_i$ for $i \in N_{\mr{coker}}$ such that for any tuple $(v_i)_{i \in N_{\mr{coker}}}$, $\{h^{(v_i)}|_{{\mc{T}}'_{n-1}})\}$ spans $\mr{coker}(\Psi_{{\mc{T}}'_{n-1}})$ and to produce $\mc{C}_i$  for $i \in N_{\mr{span}}$ such that for any $(v_i)_{i \in N_{\mr{span}}}$, $h^{(v_i)}(\tau_{v_i})$ is a root vector $X_{\alpha_i}$ with $\sum_{i \in N_{\mr{span}}} \Fp[\gal{F}] X_{\alpha_i}= \fgder$. We set $N= N_{\mr{span}} \sqcup N_{\mr{coker}}$. We likewise produce $\eta^{(v)}_b$ for all trivial primes $v \not \in {\mc{T}}'_{n-1}$ and $b \in B$ an indexing set for a $k$-basis of $(\fgder)^*$ as in the discussion around Lemma \ref{etadisjoint}. The argument now proceeds as in \cite[Theorem 5.14]{fkp:reldef}, with the simplification that the fields $L= K(\rho_{n-1}(\fgder))$, $K(\mu_{p^c})$, $K_{\eta^{(v)}_b}$ (as $v$ and $b$ vary), and $K_{h^{(v)}}$ (as $v$ varies) are all strongly linearly disjoint over $K$ ($L$ is disjoint from the $K_{h^{(v)}}$ by the same observation we have used before, that the latter is totally ramified over $K$ at places above $v$). We sketch it. For any pair $\un{v}, \un{v}' \in \prod_{i \in N} \mf{l}_n$, we form as in \textit{loc.~cit.} the classes $h= h^{\mr{old}}- \sum_{i \in N} h^{(v_i)}+ 2 \sum_{i \in N} h^{(v'_i)}$. Writing $\mc{C}_K= \mc{C}_{i, K}$ and $\mf{l}_K = \prod_{i \in N} \mf{l}_{i, K}$ for the corresponding sets of $N$-tuples of places of $K$ specified by the fixed decomposition groups, there is a positive upper-density subset $\mf{l} \subset \mf{l}_K$ such that for $(v_i)_{i \in N} \in \mf{l}$ (we in the notation still write $v_i$ for the specified place of $K$ above $v_i$), the quantities $(h^{\mr{old}}(\sigma_{v_i}))_{i \in N}$, $(h^{(v_i)}(\tau_{v_i}))_{i \in N}$, $(\sum_{i \in N} h^{(v_i)}(\sigma_{v_j}))_{j \in N}$ and $(\rho'_n(\sigma_{v_i}))_{i \in N}$ are independent of $(v_i)_{i \in N} \in \mf{l}$. With these values fixed, we choose tuples $(C_i)_{i \in N} \in (\fgder)^N$ and $(C'_i)_{i \in N} \in (\fgder)^N$ such that \textit{if} $\un{v}, \un{v}' \in \mf{l}$ can be chosen to satisfy the analogues of Equations (\ref{doubling}) and (\ref{doubling2}) in the proof of Proposition \ref{doublingprop}, then $\rho_n= (1+\vpi^{n-1}h)\rho'_n$ will, for any $i \in N$, when restricted to each $w \in \{v_i, v'_i\}$ satisfy (note that in all cases the construction forces $\rho_n(\tau_{w})$ to be a non-trivial element of $U_{\alpha_i}(\mc{O}/\vpi^n)$, so we just specify the image of $\sigma_w$): 
\begin{itemize}
\item If $n \leq eD$, $\rho_n(\sigma_{w})$ is trivial modulo center (note that then $\alpha_i(\rho_n(\sigma_{w})) \equiv 1 \equiv q_i \pmod{\vpi^n}$). We then fix a lift $\rho_w \in \Lift_{\rho_n|_{\gal{F_w}}}^{\mu, \alpha_i}(\mc{O})$ as in \cite[Lemma 3.7]{fkp:reldef} (defining a formally smooth point of the generic fiber of the trivial inertial type lifting ring for $\br|_{\gal{F_w}}$ and having $\rho_w(\sigma_w) \in T_i(\mc{O})$). This $\rho_w$ feeds into the continuation of the induction as described in the first bulleted list of the present proof.
\item If $n= eD+1$, then by construction $\rho_{n-1}(\sigma_w)$ is trivial modulo $Z_{G^0}$ (since $g_{L/K, i}$ is trivial), and $q_i$ is non-trivial modulo $\vpi^n$ but trivial modulo $\vpi^{n-1}$, so we take $\rho_n(\sigma_w)$ to be any element $t_{n, w}$ of $T_i(\mc{O}/\vpi^n)$ such that $t_{n, w} \pmod{\vpi^{n-1}}$ is trivial modulo $Z_{G^0}$, $\alpha_i(t_{n, w}) \equiv q_i \pmod{\vpi^{n}}$, and $\beta(t_{n, w}) \not \equiv 1 \pmod{\vpi^n}$ for all $\beta \in \Phi(G^0, T_i)$. (Existence of such an element when $p \gg_G 0$ is shown as in \cite[Lemma 5.13]{fkp:reldef}.)
\item If $n >eD+1$, then by construction $\rho_{n-1}(\sigma_w)$ is (modulo $Z_{G^0}$) an element $t_{w, n-1} \in T_i(\mc{O}/\vpi^{n-1})$ that is trivial modulo $\vpi^{eD}$, in general position modulo $\vpi^{eD+1}$, and satisfies $\alpha_i(t_{w, n-1}) \equiv q_i \pmod{\vpi^{n-1}}$. We require only that $\rho_n(\sigma_w)$ continue to satisfy these properties modulo $\vpi^n$, which is easily arranged.   
\end{itemize}
With these desired $C_i$ and $C'_i$ specified, the argument now
defines a \v{C}ebotarev condition $\mf{l}_{\un{v}}$ for any
$\un{v} \in \mf{l}$ such that any pair $\un{v}, \un{v}' \in \mf{l}$
with $\un{v}' \in \mf{l}_{\un{v}}$ successfully arranged Equations
(\ref{doubling}) and (\ref{doubling2}). One then runs the limiting
argument as in the proof of Proposition \ref{doublingprop}, using the
above disjointness observations for the fields $L$, $K(\mu_{p^c})$,
$K_{h^{(v_i)}}$, and $K_{\eta^{(v_i)}_b}$, to show that such
$\un{v}, \un{v}'$ exist.
\end{proof}

\section{Relative deformation theory}\label{reldefsection}
In this section we explain how to modify the relative deformation theory argument of \cite[\S 6]{fkp:reldef} for our reducible setting, and we deduce our main theorem. We will require the following additional hypotheses:
\begin{assumption}\label{finalhyp}
Let $p \gg_G 0$. Assume that $\br \colon \gal{F, {\mc{S}}} \to G(k)$
satisfies Assumptions \ref{generalhyp} and \ref{modp^Nhyp}. Fix a
lift $\mu \colon \gal{F, {\mc{S}}} \to G/G^{\mr{der}}(\mc{O})$ of $\bar{\mu}=
\br \pmod{G^{\mr{der}}}$, which we assume is geometric if $F$ is a
number field. Additionally assume the following:
\begin{itemize}
\item $H^0(\gal{F}, \br(\fgder)^*)=0$.
\item For all $v \in {\mc{S}}$, there is some lift (which may require an initial enlargement of $\mc{O}$) $\rho_v \colon \gal{F_v} \to G(\mc{O})$, of type $\mu$, of $\br|_{\gal{F_v}}$; and if (when $F$ is a number field) $v \vert p$, there is such a $\rho_v$ that is de Rham and Hodge--Tate regular. We fix throughout the section such a choice of lifts $(\rho_v)_{v \in \mc{S}}$.
\end{itemize}
\end{assumption}

\begin{thm}\label{mainthm}
Let $p \gg_G 0$, and let $\br \colon \gal{F, {\mc{S}}} \to G(k)$ be a continuous representation satisfying Assumptions \ref{generalhyp}, \ref{modp^Nhyp}, and \ref{finalhyp}. Additionally, when $F$ is a number field, assume that $F$ is totally real, and $\br$ is odd. Then for some finite set of primes $\widetilde{{\mc{S}}}$ containing ${\mc{S}}$, which we may assume disjoint from a fixed finite set $\mc{S}_0$ of primes disjoint from $\mc{S}$, there is a geometric lift
\[
\rho \colon \gal{F, \widetilde{{\mc{S}}}} \to G(\ov{\Z}_p)
\]
of $\br$. 

More precisely, we fix an integer $t$ and for each $v \in {\mc{S}}$ an irreducible component containing $\rho_v$ of:
\begin{itemize}
\item for $v \in {\mc{S}} \setminus \{v \mid p\}$, the generic fiber of the local lifting ring, $R^{\square, \mu}_{\br|_{\gal{F_v}}}[1/\vpi]$ (where $R^{\square, \mu}_{\br|_{\gal{F_v}}}$ pro-represents $\Lift_{\br}|_{\gal{F_v}}$); and
\item for $v \mid p$, the lifting ring $R_{\br|_{\gal{F_v}}}^{\square, \mu, \tau, \mbf{v}}[1/\vpi]$ whose $\ov{E}$-points parametrize lifts of $\br|_{\gal{F_v}}$ with suitably specified inertial type $\tau$ and $p$-adic Hodge type $\mbf{v}$. (See the discussion following Theorem \ref{mainthmintro}.)
\end{itemize}
Then there exist a finite extension $E'$ of $E=\Frac(\mc{O})$ (whose ring of integers and residue field we denote by $\mc{O}'$ and $k'$), which depends only on the set $\{\rho_v\}_{v \in {\mc{S}}}$; a finite set of places $\wt{{\mc{S}}}$ containing ${\mc{S}}$; and a geometric lift 
\[
\xymatrix{
& G(\mc{O}') \ar[d] \\
\gal{F, \wt{{\mc{S}}}} \ar[r]_{\br} \ar[ur]^{\rho} & G(k') 
}
\]
of $\br$ such that:
\begin{itemize}
\item $\rho$ has multiplier $\mu$.
\item $\rho(\gal{F})$ contains $\wh{G^{\mr{der}}}(\mc{O}')$.
\item For all $v \in {\mc{S}}$, $\rho|_{\gal{F_v}}$ is congruent modulo $\vpi^t$ to some $\wh{G}(\mc{O}')$-conjugate of $\rho_v$, and $\rho|_{\gal{F_v}}$ belongs to the specified irreducible component for every $v \in {\mc{S}}$.\footnote{To be clear, the set $\wt{{\mc{S}}}$ may depend on the integer $t$, but the extension $\mc{O}'$ does not depend on $t$.}
\end{itemize}
\end{thm}
\begin{proof}
As in \cite[Claim 6.13]{fkp:reldef}, we reduce to the case in which $G^0$ is adjoint, and $\fg=\fgder$ is equal to a single $\pi_0(G)$-orbit of simple factors; it is easy to see that the representations denoted $\br_s$ in \textit{loc.~cit.} attached to each $\pi_0(G)$-orbit $s$ still satisfy our assumptions. We thus assume $\br$ satisfies Assumption \ref{finalhyp}, and that $G^0$ is adjoint and is a single $\pi_0(G)$-orbit of simple factors. We have in the theorem statement fixed an integer $t$ that is our desired precision of approximation of the given local lifts $\rho_v$ ($v \in {\mc{S}}$) and irreducible components $\ov{R}_v[1/\vpi]$ containing $\rho_v$ of the corresponding lifting rings. Then as in the proof of \cite[Theorem 6.11]{fkp:reldef}, we apply \cite[Theorem 3.3.3]{bellovin-gee-G} (and see the remarks in \S \ref{prelims} on the function field case) and \cite[Lemma 4.9]{fkp:reldef} to produce a finite extension $\mc{O}'$ of $\mc{O}$ (independent of $t$) and for all $v \in {\mc{S}}$ lifts $\rho'_v$ of $\br|_{\gal{F_v}}$ that are defined over $\mc{O}'$, are congruent modulo $\vpi^t$ to $\rho_v$, and correspond to formally smooth points of $\ov{R}_v[1/\vpi]$. We replace $\mc{O}$ by $\mc{O}'$ and the $\rho_v$ by the $\rho'_v$ but retain the notation $\rho_v$, $\mc{O}$ for these replacements.\footnote{Note that this is our only enlargement of $\mc{O}$; it depends only on the local data at primes in ${\mc{S}}$ and occurs before we apply any of our global lifting results. In particular, in all applications of the arguments of \S \ref{doublingsection}, the residue field $k$ remains fixed once and for all.}

Applying \cite[Corollary B.2]{fkp:reldef}, we let $M_1$ be any integer large enough that the map
\[
H^1(\Gamma, \fgder \otimes_{\mc{O}} \mc{O}/\vpi^M) \to H^1(\Gamma, \br(\fgder) \otimes_{\mc{O}} k)
\]
is identically zero for any $M \geq M_1$, where $\Gamma$ denotes the preimage in $G^{\mr{ad}}(\mc{O})$ of $(\Ad \circ \br)(\gal{\wt{F}}) \subset G^{\mr{ad}}(k)$. We fix an $M \geq \max\{M_1, eD\}$, which for technical reasons as in \cite[\S 6]{fkp:reldef} we assume to be divisible by $e$. 

To find $\wt{\mc{S}}$ disjoint from a fixed finite set of primes $\mc{S}_0$, itself disjoint from $\mc{S}$, we replace $\mc{S}$ by $\mc{S} \cup \mc{S}_0$ and specify unramified local lifts at the places in $\mc{S}_0$. We thus may and do assume $\mc{S}_0$ is empty. Let $\mc{T} \supset \mc{S}$ be the finite enlargement taken in the statement of Theorem \ref{p^Nlift}. We then run the first $eD$ steps of the inductive argument of Theorem \ref{p^Nlift}, producing a lift $\rho_{eD} \colon \gal{F, {\mc{T}}_{eD}} \to G(\mc{O}/\vpi^{eD})$ satisfying the conclusions of \textit{loc.~cit.}. In the course of this argument, for any prime $w \in {\mc{T}}_{eD} \setminus {\mc{S}}$ at which $\rho_{eD}$ is ramified, we have fixed a lift $\rho_w \colon \gal{F_w} \to G(\mc{O})$ as in the theorem. We now for notational convenience enlarge ${\mc{S}}$ to include this finite set of primes. It is still the case that for all $v \in {\mc{S}}$ we have the fixed irreducible components $\ov{R}_v[1/\vpi]$ with their formally smooth $E$-points corresponding to the lifts $\rho_v$. We apply \cite[Proposition 4.7]{fkp:reldef} with $r_0=M$ to each $(\ov{R}_v[1/\vpi], \rho_v)$ for $v \in {\mc{S}}$ and let $N_0$ be the maximum of all the integers $n_0=n_0(v)$ produced by that result. We fix  
an integer $N$, which we may assume divisible by $e$, as in Equation (12) of \cite[Theorem 6.11]{fkp:reldef}---in brief, large enough relative to $t$, $M$, the singularities of the lifting rings $\ov{R}_v$, and such that if a mod $\vpi^N$ lift has maximal image, so does any further lift. Finally, we continue the induction of Theorem \ref{p^Nlift}, lifting $\rho_{eD}$ to a $\rho_N \colon \gal{F, {\mc{T}}_N} \to G(\mc{O}/\vpi^N)$ satisfying that theorem's conclusions.

We note that we have the Selmer conditions $L_{r, v}$ for $v \in {\mc{T}}_N$ and $1 \leq r \leq M$ needed in the proof of \cite[Theorem 6.11]{fkp:reldef}: for $v \in {\mc{S}}$, \cite[Proposition 4.7]{fkp:reldef} provides these, for $v \in {\mc{T}}_N \setminus {\mc{S}}$ at which $\rho_N$ is unramified we can take the unramified classes, and for $v \in {\mc{T}}_N \setminus {\mc{S}}$ at which $\rho_N$ ramifies, we can apply \cite[Lemma 3.5]{fkp:reldef} with $s= eD$: indeed, since we have absorbed the earlier set $\mc{T}_{eD}$ into $\mc{S}$, any $v \in \mc{T}_N \setminus \mc{S}$ at which $\rho_N$ ramifies has the property that $\rho_N|_{\gal{F_v}}$ is unramified modulo $\vpi^{eD}$, so falls under Case 1b of Theorem \ref{p^Nlift}. We can then define the associated relative Selmer group $\ov{H^1_{\mc{L}_M}(\gal{F, \mc{T}_N}, \rho_M(\fgder))}$ and relative dual Selmer group $\ov{H^1_{\mc{L}_M^\perp}(\gal{F, \mc{T}_N}, \rho_M(\fgder)^*)}$ (see \cite[Definition 6.2]{fkp:reldef}). To simplify the notation and for consistency with the notation of \cite[Theorem 6.11]{fkp:reldef} we now write $\mc{S}'$ for the set $\mc{T}_N$. The proof of the theorem finishes as in \cite[Theorem 6.11]{fkp:reldef} (see especially the proof of Claim 6.14 of \textit{loc. cit.}), provided we can generalize the argument that adds auxiliary trivial primes, drawn from the set $\mc{Q}_N$ of \cite[Definition 6.6]{fkp:reldef}, to annihilate the relative Selmer and dual Selmer groups. An inspection of \cite[\S 6]{fkp:reldef} shows that the crucial \cite[Theorem 6.9]{fkp:reldef} follows formally once we have established the analogue of Proposition 6.8, which in turn relies on Lemma 6.4, of \textit{loc.~cit.}; in the function field case, not discussed explicitly in \textit{loc. cit.}, we simply note that the Selmer conditions $\mc{L}_M$ of Proposition 6.8 satisfy the (``balanced") numerics of \cite[Proposition 4.7(3)]{fkp:reldef}, so the Greenberg--Wiles formula implies that Selmer and dual Selmer groups are ``balanced" in the sense that 
\[
|H^1_{\mc{L}_M}(\gal{F, \mc{T}_N}, \rho_M(\fgder))|= | H^1_{\mc{L}_M^\perp}(\gal{F, {\mc{T}_N}}, \rho_M(\fgder)^*)|.
\] 
(Note that \cite[Lemma 6.3]{fkp:reldef} then shows that the relative Selmer and dual Selmer groups are balanced.) We fill in the arguments generalizing Lemma 6.4 and Proposition 6.8 of \textit{loc.~cit.} in Lemma \ref{vanishing} and Proposition \ref{auxprimes} below; together these complete the proof of Theorem \ref{mainthm}. 
\end{proof}
We continue with the integers $M$ and $N$ produced in the first steps of the proof of Theorem \ref{mainthm}. For any integer $1 \leq r \leq N$, let $F_r$ be the fixed field $F(\br, \rho_r(\fgder))$. Set $F_M^*= F_M(\mu_{p^{M/e}})$, and set $F_N^*= F_N(\mu_{p^{N/e}})$ (we follow the notation of \cite[\S 6]{fkp:reldef}).
\begin{lemma}\label{vanishing}
With notation as above, we have:
\begin{itemize}
\item The map $H^1(\Gal(F_N^*/F), \rho_M(\fgder)) \to H^1(\Gal(F_N^*/F), \br(\fgder))$ is zero.
\item $H^1(\Gal(F_N^*/F), \rho_M(\fgder)^*)=0$.
\end{itemize}
\end{lemma}
\begin{proof}
$\Gal(F_N/F)$ acts trivially on the maximal $p$-power quotient of $\Gal(F_N^*/F_N)$, by Lemma \ref{abelian}. The argument of \cite[Lemma 6.4]{fkp:reldef}, which depends on the way we have arranged $M$ as at the start of Theorem \ref{mainthm}, now applies to show the first vanishing claim. Indeed, by the remark in the first sentence of the proof, $H^1(\Gal(F_N^*/F_N), \rho_M(\fgder))^{\Gal(F_N/F)}= 0$
since $\rho_M(\fgder)^{\gal{F}}= 0$, and so by inflation-restriction we are reduced to the (already arranged) vanishing of the reduction map 
\[
H^1(\Gal(F_N/F), \rho_M(\fgder)) \to H^1(\Gal(F_N/F), \br(\fgder)).
\]

The second assertion of the Lemma reduces as in \cite[Lemma 6.4]{fkp:reldef} to showing that 
\[
\Hom_{\Gal(K/F)}(\Gal(F_N^*/K), \br(\fgder)^*)=0.
\] 
By Lemma \ref{disjointmodp^N}, we reduce to showing that $\Hom_{\Fp[\gal{F}]}(\br(\fgder) \oplus \Z/p, \br(\fgder)^*)=0$. By Assumption \ref{finalhyp}, $\br(\fgder)^*$ contains no copy of the trivial $\Fp[\gal{F}]$-module, and by (a weakening of) Assumption \ref{modp^Nhyp}, $\Hom_{\gal{F}}(\br(\fgder), \br(\fgder)^*)=0$.
\end{proof}
We define the \v{C}ebotarev set ${\mc{Q}}_N$ of auxiliary primes as in \cite[Definition 6.6]{fkp:reldef}, and for $v \in {\mc{Q}}_N$ (which comes with a maximal torus $T$ and root $\alpha \in \Phi(G^0, T)$) we define the spaces of cocycles $L^{\alpha}_{r, v}$ for $1 \leq r \leq M$ as in \cite[Lemma 6.7]{fkp:reldef}.
\begin{prop}\label{auxprimes}
Assume that $\fgder$ consists of a single $\pi_0(G)$-orbit of simple factors. Let ${\mc{Q}}$ be any finite subset of ${\mc{Q}}_N$, and let $\phi \in H^1_{\mc{L}_{M}}(\Gamma_{F, {\mc{S}}' \cup {\mc{Q}}}, \rho_M(\fgder))$ and $\psi \in H^1_{\mc{L}_{M}^{\perp}}(\Gamma_{F, {\mc{S}}' \cup {\mc{Q}}}, \rho_M(\fgder)^*)$ be such that $0 \neq \ov{\phi} \in \ov{H^1_{\mc{L}_{M}}(\Gamma_{F, {\mc{S}}' \cup {\mc{Q}}}, \rho_M(\fgder))}$ and $0 \neq \ov{\psi} \in \ov{H^1_{\mc{L}_{M}^{\perp}}(\Gamma_{F, {\mc{S}}' \cup {\mc{Q}}}, \rho_M(\fgder)^*)}$.\footnote{These are the \textit{relative} Selmer and dual Selmer groups of \cite[Definition 6.2]{fkp:reldef}.} Then there exists a prime $v \in {\mc{Q}}_N$, with associated torus and root $(T, \alpha)$, such that
  \begin{itemize}
  \item $\ov{\psi}|_{\Gamma_{F_v}} \notin L_{1,v}^{\alpha, \perp}$; and
  \item $\phi|_{\gal{F_v}} \not \in L^{\alpha}_{M, v}$.
  \end{itemize}
\end{prop}
\begin{proof}
By Lemma \ref{vanishing}, we see as in \cite[Proposition 6.8]{fkp:reldef} that $\phi|_{\gal{F_N^*}}$ and $\bar{\psi}|_{\gal{F_N^*}}$ are non-zero. Fix a $\gamma_2 \in \gal{F_N^*}$ such that $\phi(\gamma_2)$ and $\bar{\psi}(\gamma_2)$ are both non-zero (note that this requires no linear disjointness over $F_N^*$ of $F_N^*(\phi)$ and $F_N^*(\bar{\psi})$---it only needs that these extensions are non-trivial). In the proof in \textit{loc.~cit.}, we replace $Y_{\bar{\psi}}$ by
\[
Y_{\bar{\psi}(\gamma_2)}=\{g \in G: \langle \bar{\psi}(\gamma_2), \fg_{\alpha_g}\rangle =0\},
\]
a proper closed subscheme of $G_k$, and we similarly replace $U_{\bar{\psi}}$ by $U_{\bar{\psi}(\gamma_2)}= G \setminus Y_{\bar{\psi}(\gamma_2)}$ and $U_{\bar{\psi}, M}$ by $U_{\bar{\psi}(\gamma_2), M}$, the set of elements $g \in G(\mc{O}/\vpi^M)$ that reduce modulo $\vpi$ to $U_{\bar{\psi}(\gamma_2)}(k)$. The purely Lie-theoretic part of the argument of \textit{loc.~cit.} then as before yields (for $p \gg_G 0$) a pair $(T, \alpha)$ consisting of a split maximal torus $T$ and a root $\alpha \in \Phi(G^0, T)$ such that $\phi(\gamma_2)$ is not contained in $\ker(\alpha|_{\mr{Lie}(T)}) \oplus \bigoplus_{\beta} \mathfrak{g}_{\beta}$, and such that $\bar{\psi}(\gamma_2)$ is not contained in the annihilator of $\mathfrak{g}_{\alpha}$ under local duality.

Now as in \textit{loc.~cit.}, we choose an element $\gamma_1 \in \gal{K}$ such that $\rho_N(\gamma_1)$ is an element of $T(\mc{O}/\vpi^N)$    satisfying
\begin{itemize}
\item the image of $\rho_M(\gamma_1)$ in $T(\mc{O}/\vpi^M)$ is trivial mod center (in fact, the reduction at the start of Theorem \ref{mainthm} ensures $Z_{G^0}$ is trivial);
\item $\kappa(\gamma_1) \equiv 1 \pmod {\vpi^M}$ but $\kappa(\gamma_1) \not \equiv 1 \pmod{\vpi^{M+1}}$;
\item for all roots $\beta \in \Phi(G^0, T)$, $\beta(\rho_{M+1}(\gamma_1) \not \equiv 1 \pmod {\vpi^{M+1}}$;
\item $\alpha(\rho_N(\gamma_1))=\kappa(\gamma_1)$.
\end{itemize}
Such a $\gamma_1$ exists because $\rho_N(\gal{F})$ contains $\wh{G^{\mr{der}}}(\mc{O}/\vpi^N)$, and $F_N(\mu_{p^{M/e}})$ is linearly disjoint from $F_M(\mu_{p^{N/e}})$ over $F_M(\mu_{p^{M/e}})$: thus we can take $\gamma_1$ trivial on $F_M(\mu_{p^{M/e}})$ but with $\rho_N(\gamma_1)$ any element that is trivial mod $\vpi^M$, and $\kappa(\gamma_1)$ any element of $\ker((\Z/p^{N/e})^\times \to (\Z/p^{M/e})^\times)$. Noting that $M \geq eD$, the linear disjointness follows as in Lemma \ref{disjointmodp^N}, which shows that $K(\rho_N(\fgder))$ is disjoint from $K_{\infty}$ over $K$. The properties desired of $\gamma_1$ are moreover determined by its restriction to $F_N^*$.

We now consider expressions $\phi(\gamma_2^r \gamma_1)= r\phi(\gamma_2)+\phi(\gamma_1)$ and $\bar{\psi}(\gamma_2^r \gamma_1)= r\bar{\psi}(\gamma_2)+\bar{\psi}(\gamma_1)$ for $r \in \{0,1,2\}$. By the construction of $\gamma_2$, we see that for some $r$, $\phi(\gamma_2^r \gamma_1)$ and $\bar{\psi}(\gamma_2^r \gamma_1)$ do not belong, respectively, to $\ker(\alpha|_{\mf{t}}) \oplus \bigoplus_{\beta} \fg_{\beta}$ and $\fg_{\alpha}^\perp$.\footnote{If the desired condition fails for both $\phi(\gamma_1)$ and $\bar{\psi}(\gamma_1)$, we take $r=1$ and are done. If it succeeds for both, we take $r=0$. If it fails for one (say $\phi$) but not the other, we first take $r=1$; we win for $\phi$ then but we might have created a problem for $\bar{\psi}$; but then if we take $r=2$ instead we'll win for both.} We conclude by defining the desired \v{C}ebotarev class of primes of $F$ to be those $v$ such that $\sigma_v$ lies in the conjugacy class of $\gamma_2^r \gamma_1 \in \Gal(F_N^*(\phi, \bar{\psi})/F)$.
\end{proof}
In \S \ref{GL2section}, we will use the following refinement:
\begin{cor}\label{gl2bound}
When $\fgder= \mf{sl}_2$, the requisite bound on $p$ in Theorem \ref{mainthm} is simply $p \geq 3$.
\end{cor}
\begin{proof}
We check the various points in the argument where the assumption $p \gg_G 0$ enters, making it explicit for $\fgder= \mf{sl}_2$. Following \cite[Remark 6.17]{fkp:reldef}, we observe the following:
\begin{itemize}
\item The proof of Proposition \ref{5.9-5.11sub} (building on \cite[Proposition 5.9, Lemma 5.11]{fkp:reldef}) requires us to check that the $k$-span of all root vectors in $\fgder$ is equal to all of $\fgder$. This is satisfied for all $p \geq 3$.
\item The selection of suitably ``general position" lifts at trivial primes needed in Theorem \ref{p^Nlift} comes from \cite[Lemma 5.13]{fkp:reldef}: it needs only a mod $\vpi$ element of the Lie algebra $\mf{t}$ of a maximal torus in $G^{\mr{der}}$ on which the (positive) root is non-zero, i.e. it needs only $p \neq 2$.
\item To produce the pair $(T, \alpha)$ in Proposition \ref{auxprimes} requires, in light of the proof of \cite[Proposition 6.8]{fkp:reldef}, to show that for $p \gg_G 0$, and any non-zero pair of elements $(\ov{A}, \ov{B}) \in \fgder \times (\fgder)^*$, there is an element $g \in G(k)$ such that $\Ad(g) \ov{A} \not \in \ker(\alpha|_{\mf{t}}) \oplus \bigoplus_{\beta} \fg_\beta$ and $\Ad(g) \ov{B} \not \in \fg_{\alpha}^\perp$. (Here $(T, \alpha)$ is a fixed initial choice of maximal torus and root, which the argument then modifies by conjugating by such a $g$.) We may assume $T$ is the diagonal maximal torus, $\alpha$ the positive root for the upper-triangular Borel, and, after identifying $\fgder \sim (\fgder)^*$ via the standard trace form (a perfect duality for $p \geq 3$), we must simultaneously arrange that $\Ad(g) (\ov{A}) \not \in \begin{pmatrix} 0 & *\\ * & 0 \\ \end{pmatrix}$, and $\Ad(g) (\ov{B}) \not \in \begin{pmatrix} * & * \\ 0 & * \\ \end{pmatrix}$. 
If these conditions hold with $g=1$, we are done; otherwise, there are three cases, depending on whether both conditions fail, or one of the two conditions fails. It is straightforward in each case to check that for $p \geq 3$ we can simultaneously conjugate these matrices by an element of $\mr{SL}_2(\Fp)$ to arrange that $\ov{A}$ has non-zero diagonal component, and $\ov{B}$ has non-zero $\fg_{-\alpha}$-component.
\end{itemize}
The other uses of $p \gg_G 0$ in the main theorem of \cite{fkp:reldef} were all to arrange that an irreducibility assumption on $\br$ implied the more technical hypotheses of the lifting method; we have instead checked these hypotheses directly, in Lemma \ref{GL2lemma}.
\end{proof}

When $F$ is a number field, but $\br$ is not necessarily totally odd, we can by the same methods prove a theorem producing not necessarily geometric, but finitely ramified, lifts:
\begin{thm}\label{nongeometriclifts}
Let $\br \colon \gal{F, {\mc{S}}} \to G(k)$ satisfy all of these hypotheses of Theorem \ref{mainthm} except:
\begin{itemize}
\item $F$ is now any number field, and in particular we do not assume $\br$ is totally odd.
\item For $v \vert p$, we assume that $\br|_{\gal{F_v}}$ has a lift $\rho_v \colon \gal{F_v} \to G(\mc{O})$ of multiplier $\mu$ such that $\rho_v$ corresponds to a formally smooth point on an irreducible component of the generic fiber $R^{\square, \mu}_{\br|_{\gal{F_v}}}[1/ \vpi]$ that has dimension $(1+[F_v: \Q_p]) \dim(G^{\mr{der}})$.
\end{itemize}
Then for some finite set of primes $\wt{{\mc{S}}} \supset {\mc{S}}$ and finite extension $\mc{O}'$ of $\mc{O}$, $\br$ admits a lift $\rho \colon \gal{F, \wt{{\mc{S}}}} \to G(\mc{O}')$ with image containing $\wh{G^{\mr{der}}}(\mc{O}')$, and $\rho$ may be arranged such that for all $v \in {\mc{S}}$, $\rho|_{\gal{F_v}}$ is congruent modulo $\vpi^t$ to some $\wh{G}(\mc{O}')$-conjugate of $\rho_v$.
\end{thm}
\begin{proof}
The technique of Theorem \ref{mainthm} still applies; see \cite[Theorem 6.21]{fkp:reldef} for how to carry out this minor variant.
\end{proof}
\section{Deducing the residually split case}\label{latticesection}

Throughout this section we assume that $G$ is connected.

\smallskip

We will show how to construct irreducible lifts of certain
representations $\br \colon \gal{F, {\mc{S}}} \to G(k)$ which do not satisfy
the assumption that $H^0(\gal{F, {\mc{S}}}, \br(\fgder)) = 0$ (which is
needed for Theorem \ref{mainthm}). Our main interest is when $\br$ is
$G$-completely reducible ($G$-cr) but not $G$-irreducible, in which
case the vanishing never holds, but the method can be applied more
generally. The main idea is that if $\br(\gal{F,{\mc{S}}})$ is contained in
$M(k)$, where $M$ is a Levi subgroup of a parabolic subgroup $P$ of
$G$, then we can often find a finite set ${\mc{T}}$ of trivial primes and a
representation $\br': \gal{F,{\mc{S}} \cup {\mc{T}}} \to G(k)$ with image contained
in $P(k)$, such that the projection of $\br'$ to $M(k)$ is $\br$ and
$H^0(\gal{F, {\mc{S}} \cup {\mc{T}}}, \br'(\fgder)) = 0$.  The following lemma shows
that if we can lift $\br'$, then we can also lift $\br$.
\begin{lemma}\label{Glattices}
  Let $\rho \colon \Gamma \to G(\mc{O})$ be a continuous
  representation of a profinite group $\Gamma$.  For any parabolic
  subgroup $P$ such that $\br(\Gamma)$ is contained in $P(k)$, we fix
  a maximal torus and Borel subgroup $T \subset B \subset P$; these
  determine a Levi subgroup $M$ of $P$ and a projection $P \to
  M$. Then there is a finite totally ramified extension $K'/K$ with
  ring of integers $\mc{O}'/\mc{O}$ and a $g \in G(K')$ such that
  $g \rho g^{-1} \colon \Gamma \to G(\mc{O}')$ and
  $\ov{g \rho g^{-1}}$ factors through $M(k)$ and is equal to the
  projection of $\br$ to $M(k)$.  In particular, if $P$ is minimal
  with respect to the property that $\br(\Gamma)$ is contained in
  $P(k)$, then $\ov{g \rho g^{-1}}$ belongs to the unique
  $G(k)$-conjugacy class of homomorphisms $\Gamma \to G(k)$
  representing ``the'' semisimplification of $\br$.
\end{lemma}
  
\begin{proof}
  Let $P$ be a parabolic subgroup of $G$ containing a Borel $B$ and
  maximal torus $T$ such that $\br$ factors through $P(k)$. The
  parabolic $P$ is associated to a subset $\Theta$ of the set of
  $B$-simple roots $\Delta$, or to a cocharacter $\eta$ of the adjoint
  torus: we can alternatively describe the root spaces appearing in
  $P$ as either the positive roots union the root system generated by
  $\Theta$, or those $\alpha$ such that
  $\langle \eta, \alpha \rangle \geq 0$ (the relation between the
  descriptions being that for simple roots $\alpha$,
  $\langle \eta, \alpha \rangle$ is 0 or 1 according to whether
  $\alpha$ belongs to $\Theta$ or $\Delta \setminus \Theta$). The
  associated Levi subgroup $M$ of $P$ is the subgroup generated by $T$
  and the root subgroups $u_{\alpha}$ for $\alpha$ belonging to the
  root system generated by $\Theta$. The projection $P \to M$ is
  determined by noting that the composite $M \subset P \to P/U$ (where
  $U$ is the unipotent radical of $P$) is an isomorphism, so composing
  with its inverse we obtain $P \to P/U \to M$.

The lift $\rho$ factors through the parahoric subgroup $\mbf{P} \subset G(\mc{O})$ equal to the preimage of $P(k)$ in $G(\mc{O})$; alternatively, $\mbf{P}$ is the stabilizer of the barycenter $x$ of the facet $\{\alpha=0: \alpha \in \Theta\}$ in the apartment $X_{\bullet}(T) \otimes_{\Z} \RR$. (This point $x$ equals $\frac{\eta}{d}$ for some positive integer $d$, and is specified by the conditions $\alpha(x)=0$ for $\alpha \in \Theta$, and $\psi(x)=\frac{1}{d}$ for the other simple affine roots $\psi$). It is generated by $T(\mc{O})$ and the root subgroups $u_{\alpha}(\vpi^n \mc{O})$ for all affine roots $\alpha+n$ such that $\alpha(x)+n \geq 0$. Explicitly, it is generated by $T(\mc{O})$ and:
\begin{itemize}
\item $u_{\alpha}(\mc{O})$ for all roots $\alpha$ such that $\langle \eta, \alpha \rangle \geq 0$; and
\item $u_{\alpha}(\vpi \mc{O})$ for all $\alpha$ such that $\langle \eta, \alpha \rangle <0$. 
\end{itemize}
Let $\mc{O}'/\mc{O}$ be the ring of integers in a finite totally
ramified extension $K'/K$ with ramification index $e$, and with uniformizer $\vpi' \in \mc{O}'$. We conjugate $\rho$ by $g= \eta(\vpi') \in G(K')$, and find that $g \rho g^{-1}$ is contained in the subgroup $g \mbf{P} g^{-1}$ of $G(K')$ generated by $T(\mc{O})$ and
\begin{itemize}
\item $u_{\alpha}(\vpi'^{\langle \eta, \alpha \rangle} \mc{O}')$ if
  $\langle \eta, \alpha \rangle \geq 0$; and
\item $u_{\alpha}(\vpi'^{\langle \eta, \alpha \rangle +e} \mc{O}')$ if
  $\langle \eta, \alpha \rangle <0$.
\end{itemize}
Choosing $e$ large enough that $\langle \eta, \alpha \rangle+e$ is positive for all $\alpha$ such that $\langle \eta, \alpha \rangle <0$, we find that $g \rho g^{-1}$ takes values in $G(\mc{O}') \subset G(K')$, and when reduced modulo $\vpi'$ takes values in the Levi subgroup $M(k) \subset P(k)$ having root basis $\Theta$. More precisely, for all $\gamma \in \Gamma$, $\ov{g \rho(\gamma) g^{-1}}$ is equal to the projection of $\rho(\gamma)$ to $M(k)$: this is clear from the above bulleted formulae and the fact that $\langle \eta, \beta \rangle=0$ for all $\beta$ in the root system of $M$.

By definition, a semisimplification of $\br$ is its projection to a
Levi factor of a minimal parabolic subgroup containing $\br(\Gamma)$;
as it is unique up to $G(k)$-conjugacy (\cite[Proposition
3.3]{serre:CR}), the last claim of the lemma is clear.
\end{proof}

\begin{lemma}\label{gw}
  Let $W \neq 0$ be an $\Fp[\gal{F,{\mc{S}}}]$-module that is finite
  dimensional over $\Fp$ and let $W'$ be any other finite dimensional
  $\Fp[\gal{F,{\mc{S}}}]$-module. Let $W_0$ be any nonzero $\Fp$-subspace of
  $W$. Then given any positive integer $n$, there is finite set ${\mc{T}}$ of
  $W'$-trivial primes of $F$, i.e., places $v$ of $F$ such that
  $N(v) \equiv 1 \pmod p$ and $W'$ is a trivial
  $\Fp[\gal{F_v}]$-module, satisfying
  \[
    \dim_{\Fp}( H^1_{\mc{L}}(\gal{F,{\mc{S}} \cup {\mc{T}}}, W)) > n.
  \]
  Here $\mc{L}$ is the Selmer condition given by $\mc{L}_v =
  \{0\}$ for $v \in {\mc{S}}$, and $\mc{L}_v$ is the
  subspace of  $H^1(\gal{F_v}, W) = \Hom(\gal{F_v}, W)$
  consisting of all homomorphisms such that the image of the inertia
  subgroup $I_v$ is contained in $W_0$ for $v \in {\mc{T}}$.
\end{lemma}
 
\begin{proof}
  The condition for being a trivial prime is a splitting condition, so
  it is a nonempty Chebotarev condition. Once we know such primes
  exist, the statement follows from the Greenberg--Wiles formula, and
  the proof is very similar to Lemma \ref{minimizeSha}: the key point
  is that for a trivial prime $v$,
  $\dim_{\Fp}(\mc{L}_v) > \dim_{\Fp}H^1_{\mr{unr}}(\gal{F_v}, W)$.
\end{proof}

\begin{lemma} \label{redlift} %
  Let $G$ be connected reductive group, $P$ a proper parabolic subgroup of
  $G$ and $M$ a Levi subgroup of $P$ containing a maximal torus $T$.
  Let $\br \colon \gal{F, {\mc{S}}} \to G(k)$ be a continuous representation
  with image contained in $M(k)$.  Let $A \subset P$ be an abelian
  unipotent subgroup which is normalised by $M$ and is a product of
  root subgroups, and let $P' = A \rtimes M \subset P$.
  Then there exists a finite set ${\mc{T}}$ of $\br$-trivial primes of $F$
  and a continuous homomorphism
  $\br' \colon \gal{F, {\mc{S}} \cup {\mc{T}}} \to G(k)$ with image contained in
  $P'(k)$ such that:
  \begin{itemize}
  \item the image of $\br'$ contains $A(k)$;
  \item the projection of $\br'$ to $M(k)$ is isomorphic to $\br$;
  \item $\br'|_{\gal{F_v}} \simeq \br|_{\gal{F_v}}$ for all $v \in {\mc{S}}$;
  \item $\br'(\gal{F_v}) \subset A(k)$ and $\br'(I_v) \
      \subset u_{\alpha}(k)$ for all $v \in {\mc{T}}$, for some root subgroup
      $u_{\alpha} \colon \mathbb{G}_a \to A$ (depending on $v$).
  \end{itemize}
\end{lemma}

\begin{proof}
  Consider the (split) exact sequence of groups
  \[
    1 \to A(k) \to P'(k) \to M(k) \to 1 .
  \]
  Let $\Gamma$ be any profinite group and $h: \Gamma \to M(k)$ a
  continuous homomorphism. Since the sequence splits, $h$ lifts to a
  homomorphism $h_1: \Gamma \to P'(k)$. If $h_2$ is any other lift of
  $h$ then one easily sees that the map $c:\Gamma \to A(k)$ defined by
  $\gamma \mapsto h_1(\gamma) \cdot h_2(\gamma^{-1})$ is a $1$-cocyle,
  where the $\Gamma$-module structure of $A(k)$ is induced by
  conjugation via $h$.  Conversely, if the lift $h_1$ is fixed and $c$
  is a $1$-cocycle then the map $h_2:\Gamma \to P'(k)$ defined by
  $h_2(\gamma) = h_1(\gamma) c(\gamma^{-1})$ is another
  lift of $h$.  Moreover, the cocycle $c$ is a boundary iff $h_2$ is
  conjugate to $h_1$ by an element of $A(k)$.

  Let $h: \gal{F,{\mc{S}}} \to M(k)$ be given by composing $\br$ with the
  projection to $M(k)$, so $A(k)$ becomes an $\Fp[\gal{F,{\mc{S}}}]$-module
  by conjugation. Let $u_{\alpha}: \mathbb{G}_a \to A$ be any root
  subgroup.  By applying Lemma \ref{gw} with $W = A(k)$, $W_0$ a one-dimensional $\Fp$-subspace $V_i$ of $u_{\alpha}(k)$, and $W'$ the
  $\gal{F,{\mc{S}}}$-module obtained by composing $\br$ with any finite
  dimensional faithful representation of $G(k)$ over $k$, we get a
  finite set of $\br$-trivial primes ${\mc{T}}_i$ and a $1$-cocyle $c_i$
  giving a class in $H^1_{\mc{L}}(\gal{F,{\mc{S}} \cup {\mc{T}}_i}, W))$ which is
  ramified at some prime in ${\mc{T}}_i$, and with the image of $I_{v_i}$
  contained in $V_i$ for all $v_i \in {\mc{T}}_i$. Let $c = \sum_i c_i$,
  where $i$ runs over all possible subspaces $V_i$ of $u_{\alpha}$
  (for all $\alpha$) and set ${\mc{T}} = \cup_i {\mc{T}}_i$; we may and do assume
  that all the ${\mc{T}}_i$ and ${\mc{S}}$ are mutually disjoint. The discussion of
  the previous paragraph then shows that by modifying $\br$ by $c$ we
  obtain a representation $\br': \gal{F,{\mc{S}} \cup {\mc{T}}} \to P'(k)$
  satisfying all the bulleted properties: the first holds because $A$ is
  generated by root subgroups, the second is clear, the third holds because
  the Selmer condition in Lemma \ref{gw} at places in $\mc{S}$ is trivial, and the last holds because of the choice of $W_0$.
\end{proof}
In our application, we will apply the following lemma to the primes in the set ${\mc{T}}$ produced in Lemma \ref{redlift}.
\begin{lemma} \label{rootlift}%
  Let $G$ be a connected reductive group, $T$ a maximal torus of $G$,
  $A$ a commutative unipotent subgroup of $G$ and
  $u_{\alpha} \colon \mathbb{G}_a \to A$ a root subgroup. Let $F_v$ be
  a local field with $N(v) \equiv 1 \mod p$ and
  $\br:\gal{F_v} \to G(k)$ a continuous representation such that
  $\br(\gal{F_v}) \subset A(k)$ and $\br(I_v) \subset
  u_{\alpha}(k)$. If $p \neq 2$ then $\br$ lifts to a representation
  $\rho: \gal{F_v} \to G(\mc{O})$.
\end{lemma}

\begin{proof}
  Since $A(k)$ is a $p$-group, the representation $\br$ is tamely
  ramified. Let $\sigma$ in $\gal{F_v}$ be a lift of Frobenius and
  $\tau \in \gal{F_v}$ a lift of a generator of tame inertia. Choose
  any lift $u_1 \in A(\mc{O})$ of $\br(\sigma)$ and any lift
  $u_2 \in u_{\alpha}(\mc{O})$ of $\br(\tau)$; since $A$ is
  commutative, $u_1$ and $u_2$ commute. Let $x \in \mc{O}$ be the
  square root of $N(v)$ which is congruent to $1$ modulo $p$ and let
  $t:= \alpha^{\vee}(x) \in T$. Define $\rho: \gal{F_v} \to G(\mc{O})$
  by $\rho(\sigma) = tu_1$ and $\rho(\tau) = u_2$. By the choices of
  $u_1$, $u_2$ and $t$, it follows that
  $\rho(\sigma)\rho(\tau)\rho(\sigma)^{-1} = \rho(\tau)^{N(v)}$, so
  (by the structure of tame inertia) $\rho$, defined on generators
  above, does indeed give rise to a continuous homomorphism
  $\rho: \gal{F_v} \to G(\mc{O})$.  Finally, since $x$ reduces to $1$
  modulo $p$ it follows from the choice of $u_1$ and $u_2$ that $\rho$
  is indeed a lift of $\br$.

\end{proof}

\begin{lemma} \label{levi} Let $G$ be a connected reductive group, $P$
  a parabolic subgroup of $G$ and $M$ a Levi subgroup of $P$ and
  assume $p \gg_G 0$. Let $\br: \gal{F,{\mc{S}}} \to G(k)$ be a continuous
  representation with $\br(\gal{F,{\mc{S}}}) \subset M(k)$ satisfying the
  local parts of Assumption \ref{finalhyp} and with
  $H^0(\gal{F,{\mc{S}}}, \br(\fgder)) = \mf{z}(M) \cap \fgder$, where
  $\mf{z}(M)$ is the Lie algebra of the centre of $M$. Then there
  exist a finite set of $\br$-trivial primes ${\mc{T}}$ and a continuous
  representation $\br': \gal{F,{\mc{S}} \cup {\mc{T}}} \to G(k)$ such that
  $\br'(\gal{F,{\mc{S}} \cup {\mc{T}}}) \subset P(k)$, the projection of $\br'$ to
  $M(k)$ is isomorphic to $\br$, $\br'$ satisfies the local hypotheses
  of Assumption \ref{finalhyp}, and
  $H^0(\gal{F,{\mc{S}} \cup {\mc{T}}}, \br'(\fgder)) = 0$.
\end{lemma}

\begin{proof}
  We apply Lemma \ref{l:abelian} (to $G^{\mr{der}}$) to get an abelian
  subgroup $A$ of $U$ (the unipotent radical of $P$) which is a
  product of root subgroups, is normalised by $M$, and such that
  $\mf{z}(M) \cap \fgder$ acts faithfully on $\mf{a}$. By applying
  Lemma \ref{redlift} with this $A$ we get $\br'$ which satisfies all
  the required properties except that we still need to show that
  $H^0(\gal{F,{\mc{S}} \cup {\mc{T}}}, \br'(\fgder)) = 0$.

  Since $\br'$ projects to $\br$, and $\br(\gal{F,{\mc{S}}}) \subset M(k)$,
  the assumption on $\br$ implies that
  $H^0(\gal{F,{\mc{S}} \cup {\mc{T}}}, \br'(\fgder)) \subset \mf{z}(M) \cap \fgder$.
  An element $a \in \mf{a}$ satisfies $[a,g] = 0$ for $g \in \fgder$
  iff $\exp(a)$ (see Lemma \ref{l:exp}) acts trivially on $g$; here we
  use that $p \gg_G 0$. From Lemma \ref{l:abelian} we have that
  $\mf{z}(M) \cap \fgder$ acts faithfully on $\mf{a}$, i.e., for any
  $0 \neq z \in \mf{z}(M) \cap \fgder$ there exists $a \in \mf{a}$
  such that $[z,a] \neq 0$. Since
  $A(k) \subset \br'(\gal{F,{\mc{S}} \cup {\mc{T}}})$ (by Lemma \ref{redlift}) we
  conclude that $H^0(\gal{F,{\mc{S}} \cup {\mc{T}}}, \br'(\fgder)) = 0$.
\end{proof}

\begin{cor} \label{cr-lift} Let $G$ be a connected reductive group,
  $P$ a parabolic subgroup of $G$ and $M$ a Levi subgroup of $P$.  Let
  $\br: \gal{F,{\mc{S}}} \to G(k)$ be a continuous representation with
  image contained in $M(k)$. Assume that $ p \gg_G 0$, $\br$ satisfies
  the local parts of Assumption \ref{finalhyp},
  $H^0(\gal{F,{\mc{S}}}, \br(\fgder)) = \mf{z}(M) \cap \fgder$,
  $H^0(\gal{F,{\mc{S}}}, \br(\fgder)^*) = 0$, and either the $p^{th}$
  roots of unity are not contained in $F(\br(\fgder))$ or (restricting
  to $P$ a maximal parabolic) $\br$ satisfies condition (2) of Lemma
  \ref{lem:h^1}.  Choose $\mu$ as in Assumption \ref{finalhyp}. Then
  there exist a finite extension $E'$ of $E$ (whose ring of integers
  and residue field we denote by $\mc{O}'$ and $k'$), an enlargement
  $\wt{{\mc{S}}}$ of ${\mc{S}}$ (by adding trivial primes) and a
  geometric lift
\[
\xymatrix{
& G(\mc{O}') \ar[d] \\
\gal{F, \wt{{\mc{S}}}} \ar[r]_{\br} \ar[ur]^{\rho} & G(k') 
}
\]
of $\br$ such that:
\begin{itemize}
\item $\rho$ has multiplier $\mu$.
\item $\rho(\gal{F}) \cap G^{\mr{der}}(\mc{O}')$ is Zariski-dense in
  $G^{\mr{der}}$.
\end{itemize}
\end{cor}

\begin{proof}
  Let $\br': \gal{F,{\mc{S}} \cup {\mc{T}}} \to G(k)$ be a
  representation obtained by applying Lemma \ref{levi} to $\br$. By
  the conclusion of that lemma and Lemma \ref{rootlift}, $\br'$
  satisfies the local parts of Assumption \ref{finalhyp} and
  $H^0(\gal{F,{\mc{S}}\cup {\mc{T}}}, \br'(\fgder)) = 0$. Since the
  projection of $\br'$ to $M(k)$ is equivalent to $\br$ and
  $H^0(\gal{F,{\mc{S}}}, \br(\fgder)^*) = 0$, we also have
  $H^0(\gal{F,{\mc{S}} \cup {\mc{T}}}, \br'(\fgder)^*) = 0$ (since the
  semisimplifications of $\br(\fgder)^*$ and
  $\br'(\fgder)^*$ are isomorphic). The assumptions of Lemmas
  \ref{l:allP} and \ref{lem:h^1} only depend on the semisimplification
  of a representation, so they hold for $\br'$ if they hold for
  $\br$. Since we have assumed that at least one of them holds for
  $\br$, it follows from these lemmas that Assumption \ref{finalhyp}
  holds for $\br'$ in its entirety.

  Thus, we may apply Theorem \ref{mainthm} to $\br'$ to get a lift
  $\rho': \gal{F,\wt{{\mc{S}}}} \to G(\mc{O}')$ of $\br'$ satisfying the
  conclusions of that theorem with respect to $\br'$. We now apply
  Lemma \ref{Glattices} to $\rho'$ to see that after replacing
  $\mc{O}'$ with a finite extension, we get a lift
  $\rho: \gal{F,\wt{{\mc{S}}}} \to G(\mc{O}')$ of $\br$ satisfying all the
  claimed properties.
\end{proof}

\begin{rmk}
  \begin{itemize}
  \item Lemma \ref{lem:h^0} gives some examples in which $P$ is a
    maximal parabolic where the assumptions on $H^0$ can be easily
    checked.
  \item If $M$ is a maximal torus, and the image of $\br$ is contained in no proper subtorus, then the condition on
    $H^0(\gal{F,{\mc{S}}}, \br(\fgder))$ always holds. We have
    $H^0(\gal{F,{\mc{S}}}, \br(\fgder)^*) = 0$ if $\kbar^{-1}$ does
    not occur as a subrepresentation of the $\gal{F}$ action on
    $\fgder$ and (3) of Lemma \ref{l:borel} gives a condition for
    Assumption \ref{modp^Nhyp} to hold. Thus, Corollary \ref{cr-lift}
    gives a fairly simple conditions under which representations into
    a torus have irreducible (geometric) lifts. For a more general
    criterion for lifting representations into a Borel subgroup see \S
    \ref{s:borel}.
  \end{itemize}
    
\end{rmk}

\begin{eg} \label{gsp4} %
  Let $F = \Q$, $G = \mr{GSp_4}$ and let $P$ be the Siegel parabolic,
  so the Levi quotient $M$ of $P$ is $\mr{GL}_2 \times \mbb{G}_m$. Let
  $\br_f:\Gamma_{\Q} \to \mr{GL}_2(k)$ be an absolutely irreducible
  representation associated to a newform $f$ on $\Gamma_0(N)$ of
  weight $r \geq 2$ and let $\delta: \Gamma_{\Q} \to \mbb{G}_m(k)$ be
  any odd character. Then by choosing a Levi subgroup of $P$,
  $\br = (\br_f,\delta)$ can be viewed as a representation 
  $\gal{\Q,N'} \to G(k)$, for some $N'$ (depending on $\delta$) with
  $Np \mid N'$, where we abuse notation by using $N'$ to also mean the
  set of primes dividing $N'$. By construction, $\br_f$
  has a global lift to $\mr{GL}_2(\mc{O}')$ (for some finite extension
  $\mc{O}'$ of $\mc{O}$) which is regular, so by lifting $\delta$ to a
  de Rham character of nonzero weight we see that the local lifting
  hypotheses of Corollary \ref{cr-lift} are satisfied for all
  $v \in N'$. Since $\br_f$ is absolutely irreducible, one easily sees
  from case (2) of Lemma \ref{lem:h^0} that both properties (a) and
  (b) of the lemma hold for $\br$ except possibly when the projective
  image of $\br_f$ is a dihedral group.  The character
  $\mr{det} \circ \br_f$ is equal to $\kbar^{r-1}$, so in the dihedral
  case we see that if $\kbar^{r-1}$ has order greater than $2$ then
  $\br_f$ does not preserve a symmetric bilinear form. Similarly, if
  $\kbar^{r}$ and $\kbar^{r-2}$ are of order greater than $2$ then
  $\br_f$ does not preserve a symmetric bilinear form with multiplier
  $\kbar^{\pm 1}$. Thus whenever these conditions on the order of
  $\kbar$ are satisfied, Properties (a) and (b) of Case (2) of Lemma
  \ref{lem:h^0} hold. It also follows from Example \ref{exgsp4a} that
  properties (1) and (2) of Lemma \ref{lem:h^1} will hold for $\br$ as
  long as the characters $\kbar^{a(r-1) + 1}$ have order at least $5$
  for $|a| \leq 2$. (In particular, both of the conditions only depend
  on $r$ and $p$, and for any fixed $r>2$ this holds for all
  $p \gg 0$.)  We now give conditions so that
  $H^0(\gal{\Q}, \br(\fgder))=\mf{z}_M \cap \fgder$. First, the
  adjoint action on $\br(\mf{m} \cap \fgder)$ itself is isomorphic to
  the adjoint representation of $\br_f$; by irreduciblity, the
  invariants are just $\mf{z}_M \cap \fgder$. On the unipotent radical
  $\mf{u}^+$ of $\Lie(P)$, the adjoint action of $\gal{\Q}$ is
  isomorphic to $\Sym^2(\br_f) \otimes \delta^{-1}$. When $\br_f$ is
  non-dihedral, this representation is absolutely irreducible, so has
  no $\gal{\Q}$-invariants.  In the dihedral case it clearly follows
  that there will be no $\gal{\Q}$-invariants for all but at most
  three choices of $\delta$ (which can be determined explicitly, but
  we leave this for the interested reader). Similar statements hold
  for the action on the unipotent radical $\mf{u}^-$ of the Lie
  algebra of the opposite parabolic.
     We conclude by Corollary \ref{cr-lift} that if
      these conditions hold then
      $\br$ has a regular geometric lift to $\rho:\Gamma_{\Q} \to
      \mr{GSp}_4(\mc{O}')$ with $\rho(\Gamma_{\Q}) \cap
      \mr{Sp}_4(\mc{O'})$ Zariski-dense in $\mr{Sp}_4$.
\end{eg}

\section{The case $F$ totally real and $G=\mr{GL}_2$}\label{GL2section}
In this section we specialize our main result to the case $G= \mr{GL}_2$, and in doing so we improve considerably the results of \cite{ramakrishna-hamblen} that were our starting-point in \cite{fkp:reldef} and the present paper. In this section, we assume that $F$ is a totally real number field, since this is the modularity application of greatest interest. In \S \ref{functionfields} we will discuss modularity applications in the function field setting. We will analyze reducible representations conjugate to (denoted $\sim$) extensions of the form $\begin{pmatrix} \bar{\chi} & * \\ 0 & 1\end{pmatrix}$; since any reducible representation has a twist of this form, we lose no generality in this assumption but do simplify the notation somewhat.
\begin{lemma}\label{GL2lemma}
Let $p \geq 3$ be a prime. Let $\br \colon \gal{F, {\mc{S}}} \to \mr{GL}_2(k)$ be a continuous representation of the form
\[
\br \sim \begin{pmatrix}
\bar{\chi} & \ast \\
0 & 1
\end{pmatrix},
\]
and assume that
\begin{itemize}
\item $\bar{\chi}(c_v)=-1$ for all $v \vert \infty$.
\item $\bar{\chi} \neq \kbar^{-1}$.
\item The extension class $\ast$ is non-trivial.
\end{itemize}
Then Assumptions \ref{generalhyp}, \ref{modp^Nhyp}, and the first item in Assumption \ref{finalhyp}, hold for $\br$.
\end{lemma}
\begin{proof}
Fix a geometric lift $\chi \colon \gal{F, {\mc{S}}} \to \mc{O}^\times$ of $\bar{\chi}$. 
We verify in turn Assumptions \ref{generalhyp}, \ref{modp^Nhyp}, and the first part of \ref{finalhyp}. We first check that the trivial representation is neither a submodule of $\br(\fgder)$ nor of $\br(\fgder)^*$, 
and note that for this it suffices to consider the trivial $k[\gal{F}]$ (rather than $\Fp[\gal{F}]$) representation. Non-splitness of the extension implies that the only one-dimensional subrepresentation of $\br(\fgder)$ is $k(\bar{\chi})$, which is non-trivial. Likewise the only one-dimensional subrepresentation of $\br(\fgder)^* \cong \br(\fgder)(1)$ is isomorphic to $k(\bar{\chi}\kbar)$, so the assumption that $\bar{\chi} \neq \kappa^{-1}$ implies our claim. 

Next we compute the cohomology group $H^1(\Gal(K/F), \br(\fgder)^*)$; we first treat the case in which $\bar{\chi}$ is also not equal to $\bar{\kappa}$.  Let $P$ be the (unique) $p$-Sylow subgroup of $\Gal(K/F)$, and set $K_0=K^P$; explicitly, $K_0=F(\bar{\chi}, \mu_p)$. Since $[K_0:F]$ is coprime to $p$, it suffices to check the vanishing of $H^1(P, \br(\fgder)^*)^{\Gal(K_0/F)}$. Consider the obvious filtration of $\br(\fgder)^*$ by $k[\gal{F}]$-submodules with successive subquotients isomorphic to $k(\bar{\chi} \kbar)$, $k(\kbar)$, and $k(\bar{\chi}^{-1}\kbar)$. Since taking invariants by the group $\Gal(K_0/F)$ is exact in characteristic $p$, it suffices from the long exact sequence in cohomology to check that $H^1(P, k(\psi))^{\Gal(K_0/F)}=0$ for $\psi \in \{\bar{\chi}^{-1}\kbar, \kbar, \bar{\chi} \kbar\}$. The group $P$ acts trivially on these subquotients, so we are left to compute $\Hom_{\Gal(K_0/F)}(P, k(\psi))$. As $\Fp[\Gal(K_0/F)]$-module, $P$ is isomorphic to a direct sum of copies of $\Fp(\bar{\chi})$ (i.e., the field extension of $\Fp$ generated by the values of $\bar{\chi}$, with Galois action given by $\bar{\chi}$). We therefore reduce the desired vanishing to checking that $\bar{\chi}^{\sigma}$ is not equal to $\bar{\chi}^{-1}\kbar$, $\kbar$, or $\bar{\chi}\kbar$ for any $\sigma \in \Aut(k)$: but such an equality with $\bar{\chi}\kbar^{-1}$ or $\bar{\chi}\kbar$ would contradict the behavior of complex conjugation (as $p \neq 2$), and with $\kbar$ would contradict the assumption $\bar{\chi} \neq \kbar$.

We must treat separately the case $\bar{\chi}= \bar{\kappa}$ (but not equal to $\bar{\kappa}^{-1}$). We have as before the $k[\gal{F}]$-stable filtration 
\[
0 \subset F_2 \subset F_1 \subset F_0= \br(\fgder)^*
\]
with $F_0/F_1 \cong k$, $F_1/F_2 \cong k(\bar{\kappa})$, and $F_2 \cong k(\bar{\kappa}^2)$. We abbreviate $\Gamma= \Gal(K/F)$, so we have $P \subset \Gamma$ as above. Taking (exact) $\Gamma/P$-invariants on the long exact sequence in cohomology associated to $0 \to F_1 \to F_0 \to F_0/F_1 \to 0$, we obtain an exact sequence
\begin{equation}\label{cohses}
0 \to k \to H^1(P, F_1)^{\Gamma/P} \to H^1(P, \br(\fgder)^*)^{\Gamma/P} \to 0,
\end{equation}
where the first term is zero since $(\br(\fgder)^*)^\Gamma$ vanishes, and the last term is zero since 
\[
H^1(P, \br(\fgder)^*/F_1)^{\Gamma/P}=\Hom_{\Gamma}(P, k)=0
\]
as $P$ is isomorphic to a direct sum of copies of $\Fp(\bar{\kappa})$. Similarly, we have an exact sequence
\[
0 \to H^1(P, F_1)^{\Gamma/P} \to H^1(P, F_1/F_2)^{\Gamma/P} \to H^2(P, F_2)^{\Gamma/P},
\]
where the first map is injective because $H^1(P, F_2)^{\Gamma/P}=\Hom_{\Gamma}(P, k(\bar{\kappa}^2))=0$. We will show that $H^1(P, F_1)^{\Gamma/P}$ is one-dimensional over $k$. It will follow from sequence (\ref{cohses}) that $H^1(P, \br(\fgder)^*)^{\Gamma/P}$ vanishes, as needed.

The group $P$ is identified by $\br$ to an $\Fp$-subspace of $k \cong \begin{pmatrix} 1 & k \\ 0 & 1 \end{pmatrix}$, and we choose a basis $P= \bigoplus_{i=1}^r \Fp e_i$. Any cocycle $\phi \in Z^1(P, F_1)$ is determined by its values $\{\phi(e_i)\}_{i=1}^r$ on this basis. We identify $\br(\fgder)^* \cong \br(\fgder)(1)$, and then we write $\phi(e_i)= \begin{pmatrix} a_i & b_i \\ 0 & -a_i \end{pmatrix}$, for some $a_i, b_i \in k$. Now for any pair $i, j$, consider the relation 
\[
\phi(e_i)+e_i \cdot \phi(e_j)= \phi(e_i+e_j)=\phi(e_j+e_i)= \phi(e_j)+e_j \cdot \phi(e_i), 
\]
which translates to an identity
\[
\begin{pmatrix}
a_i+a_j & b_i+b_j-2a_je_i \\
0 & -a_i-a_j
\end{pmatrix}
=
\begin{pmatrix}
a_i+a_j & b_i+b_j-2a_ie_j \\
0 & -a_i-a_j
\end{pmatrix}.
\]
We conclude that all $a_i$ are determined by the value $a_1$, via $a_i= a_1 \cdot \frac{e_i}{e_1}$. Thus we realize $Z^1(P, F_1)$ as a subgroup
\begin{align*}
Z^1(P, &F_1) \subset k \oplus k^{\oplus r} \\
&\phi \mapsto (a_1, (b_1, \ldots, b_r)).
\end{align*}
The cocycles $\phi$ corresponding to tuples with $a_1=0$ (hence all $a_i=0$) all obviously map to zero in $Z^1(P, F_1/F_2)$, so the image of $H^1(P, F_1) \to H^1(P, F_1/F_2)$ is at most one-dimensional. Since the map on $\Gamma/P$-invariants $H^1(P, F_1)^{\Gamma/P} \to H^1(P, F_1/F_2)^{\Gamma/P}$ is injective, we conclude that $\dim_k H^1(P, F_1)^{\Gamma/P}$ is at most one. As noted above, this suffices to conclude that $H^1(\Gamma, \br(\fgder)^*)=0$.

Next we verify Assumption \ref{modp^Nhyp}, checking that there is no $k[\gal{F}]$-module surjection $\br(\fgder)^{\sigma} \onto V$ for some $\sigma \in \Aut(k)$ and some $k[\gal{F}]$-subquotient $V$ of $\br(\fgder)^*$. We immediately reduce to ruling out the cases $V=k(\psi)$ for $\psi \in \{\bar{\chi}^{-1}\kbar, \kbar, \bar{\chi} \kbar\}$. By non-splitness, the only one-dimensional quotient of $\br(\fgder)^{\sigma}$ is $(\bar{\chi}^{-1})^{\sigma}$. Arguing as before, the behavior of complex conjugation rules out all cases but $(\bar{\chi}^{-1})^{\sigma}=\kbar$, which is in turn ruled out by the assumption $\bar{\chi} \neq \kbar^{-1}$.
\end{proof}
We now briefly discuss the local lifting hypothesis of Assumption \ref{finalhyp}. For primes $v \vert p$, we will say $\br|_{\gal{F_v}}$ is ordinary if for some finite extension $F_v'/F_v$ it is $F_v'$-ordinary in the sense of Definition \ref{orddef}.
\begin{lemma}\label{locallifting}
Assume $p \geq 3$. Let $\br \colon \gal{F, {\mc{S}}} \to \mr{GL}_2(k)$ be a continuous representation of the form 
\[
\br \sim \begin{pmatrix}
\bar{\chi} & \ast \\
0 & 1
\end{pmatrix}.
\]
Fix a geometric lift $\mu \colon \gal{F, \mc{S}} \to \mc{O}^\times$ of $\bar{\mu}= \det(\br)= \bar{\chi}$ of the form $\kappa^{r-1} \chi_0$ for some integer $r \geq 2$ and finite-order character $\chi_0$. Then, allowing ourselves to replace $\mc{O}$ by the ring of integers $\mc{O}'$ in some finite extension of $\Frac(\mc{O})$, we have:
\begin{itemize}
\item For $v \in {\mc{S}}$ not dividing $p$, and any such choice of $\mu$, there is a lift $\rho_v \colon \gal{F_v} \to \mr{GL}_2(\mc{O}')$ of $\br|_{\gal{F_v}}$ with determinant $\mu$.
\item For $v \vert p$ and any such choice of $\mu$, there is an ordinary potentially crystalline lift $\rho_v \colon \gal{F_v} \to \mr{GL}_2(\mc{O}')$ of $\br|_{\gal{F_v}}$ with Hodge--Tate weights $\{0, r-1\}$ and determinant $\mu$. Moreover, we can always find an ordinary potentially crystalline lift $\rho_v \colon \gal{F_v} \to \mr{GL}_2(\mc{O}')$ of $\br|_{\gal{F_v}}$ with Hodge--Tate weights $\{0, r-1\}$ and having a non-trivial unramified quotient, possibly at the expense of not guaranteeing $\det(\rho_v)= \mu$.
\item When $F= \Q$ (or is split at all places above $p$), and $r$ is equal to the weight $k(\br|_{\gal{\Q_p}})$ associated to $\br$ by Serre (\cite[\S 2.3]{serre:conjectures}), there is a crystalline lift $\rho_p \colon \gal{\Q_p} \to \mr{GL}_2(\mc{O}')$ of $\br|_{\gal{\Q_p}}$ with Hodge--Tate weights $\{0, r-1\}$.
\end{itemize} 
\end{lemma}
\begin{proof}
First consider the case where $v$ does not divide $p$. If $\br(I_{F_v})$ has order coprime to $p$, then as is well-known there is in fact a non-empty formally smooth irreducible component of the lifting ring $R_{\br|_{\gal{F_v}}}^{\square, \mu}$ corresponding to lifts whose projectivization factors through $\ker(\br|_{I_{F_v}})$. If on the other hand $\br(I_{F_v})$ has order divisible by $p$, then (for $p \neq 2$) $\br|_{\gal{F_v}}$ has the form $\begin{pmatrix} \kbar & * \\ 0 & 1\end{pmatrix}$: indeed, from \cite[\S 2]{diamond:extension}, we see that $\br|_{\gal{F_v}}$ must be a non-split extension of $\bar{\psi}$ by $\kappa \bar{\psi}$ for some character $\bar{\psi}$, and since we have assumed that $\br$ is an extension of $1$ by $\bar{\chi}$ we must have $\bar{\psi}=1$, $\bar{\chi}|_{\gal{F_v}}= \kbar$. By the discussion of \cite[\S 1 E3]{taylor:icos2}, there is for any lift $\psi$ of the trivial character mod $\vpi$ a formally smooth irreducible component of the lifting ring $R_{\br|_{\gal{F_v}}}^{\square, \kappa \psi^2}$. Any character $\gal{F_v} \to 1+\vpi \mc{O}$ admits a unique square-root ($p \neq 2$), so we can always choose the twist $\psi$ such that $\kappa \psi^2= \mu$.

Now consider $v \vert p$. When $\bar{\chi}|_{\gal{F_v}} \neq
\bar{\kappa}$, a straightforward duality argument shows that for any
$r \geq 2$ and finite-order character $\chi_0$ such that $\kappa^{r-1}
\chi_0$ lifts $\bar{\chi}$, $\br|_{\gal{F_v}}$ has a potentially
crystalline lift of the form $\begin{pmatrix} \kappa^{r-1} \chi_0 & *
  \\ 0 & 1\end{pmatrix}$. When $\bar{\chi}|_{\gal{F_v}}=
\bar{\kappa}$, the argument of \cite[Lemma 6.1.6]{blgg:ordhmf} shows
that we can find potentially unramified characters $\psi_1$ and
$\psi_2$ whose product is $\chi_0$ such that $\br|_{\gal{F_v}}$ has a
potentially crystalline lift $\rho_v$ of the form $\begin{pmatrix}
  \kappa^{r-1} \psi_1 & * \\ 0 & \psi_2 \end{pmatrix}$. The argument
of \textit{loc.~cit.~}is only written for $r=2$, so we sketch the
proof. Let $L \subset H^1(\gal{F_v}, k(\bar{\kappa}))$ be the line
corresponding to $\br|_{\gal{F_v}}$, and let $H \subset H^1(\gal{F_v},
k)$ be its annihilator under the local duality pairing. As in
\textit{loc.~cit.}, we must find potentially unramified characters
$\psi_1$ and $\psi_2$ such that $\chi_0=\psi_1 \psi_2$, $\psi_2$ lifts
the trivial character, and $\kappa^{r-1} \psi_1$ lifts $\bar{\chi}$,
and such that, writing $\kappa^{2-r} \frac{\psi_2}{\psi_1}= 1+ \vpi^n
\alpha$ for some continuous function $\alpha \colon \gal{F_v} \to
\mc{O}$ whose reduction modulo $\vpi$ is a non-trivial element of
$H^1(\gal{F_v}, k)$, we have $\bar{\alpha} \in H$. We make an initial
choice with $\psi_2=1$, $\psi_1= \chi_0$, and we suppose now that the
associated $\bar{\alpha}$ does not lie in $H$. When $\br|_{\gal{F_v}}$
is peu ramifi\'e, $H$ contains the unramified line, and by replacing
$\mc{O}$ by a ramified extension (guaranteeing $n \geq 2$) and
twisting $\psi_1$ and $\psi_2$ by unramified characters sending
Frobenius to $(1+\vpi)$ and $(1+\vpi)^{-1}$, the new $\bar{\alpha}$
will be unramified and thus in $H$. Note that in this case our lift
satisfies $\psi_2|_{I_{F_v}}=1$. When $\br|_{\gal{F_v}}$ is tr\`{e}s ramifi\'{e}, we first make a ramified extension of $\mc{O}$ by adjoining some $\zeta_{p^M}$: we can then twist the initial $\psi_1$ and $\psi_2$ by potentially unramified characters with trivial reduction such that the associated $\bar{\alpha}$ is ramified. A further unramified twist of $\psi_1$ and $\psi_2$ then allows us to adjust $\bar{\alpha}$ to lie in $H$ (using that the initial $\bar{\alpha}$ is now ramified, and $H$ is a hyperplane not containing the unramified line). Note that in this case, $\psi_2|_{I_{F_v}}$ is not necessarily trivial, but it is possible to keep the same $r$ and modify the initial choice of $\chi_0$ to produce an ordinary lift with Hodge--Tate weights $\{0, r-1\}$ and $\psi_2|_{I_{F_v}}=1$.


The last point is implicit in the discussion of \cite[\S 3.2, Proof of Theorem 3.17]{buzzard-diamond-jarvis:conjecture} and can also easily be seen to follow from the argument of the previous paragraph.

\end{proof}

Recall that to a newform $f \in S_{r}(\Gamma_1(N))$ for any integers $r \geq 2$ and $N \geq 1$, and any isomorphism $\iota \colon \CC \xrightarrow{\sim} \ov{\Q}_p$, there is an associated geometric Galois representation $\rho_{f, \iota} \colon \gal{\Q, \wt{{\mc{S}}}} \to \mr{GL}_2(\ov{\Q}_p)$, where $\wt{{\mc{S}}}$ is the set of primes dividing $Np$. The representation $\rho_{f, \iota}$ takes values in $\mr{GL}_2(E)$ for some finite extension $E$ of  $\Q_p$. For any $\gal{\Q}$-stable $\mc{O}_E$-lattice we obtain an associated residual representation $\br_{f, \iota} \colon \gal{\Q, {\mc{S}}} \to \mr{GL}_2(\mc{O}_E/\mf{m}_E)$; we do not semisimplify, so the isomorphism class of $\br_{f, \iota}$ may depend on the choice of lattice. We say that a representation $\br \colon \gal{\Q, {\mc{S}}} \to \mr{GL}_2(k)$ is modular of weight $r$ and level $N$ if there exist an $f \in S_{r}(\Gamma_1(N))$, an embedding $\iota$ as above, and an $\mc{O}_E$-lattice inside $\rho_{f, \iota}$ such that  $\bar{\rho}_{f, \iota}$ is isomorphic to $\br$ for some embedding $\mc{O}_E/ \mf{m}_E \to \bar{k}$.

Before proceeding, we state the following well-known result (not used in what follows) to contrast it with  \cite[Theorem 2]{ramakrishna-hamblen} and  our improvement in Theorem \ref{GL2} below. 

\begin{lemma}\label{GL2-easy}
Let  let $\br \colon \gal{\Q, {\mc{S}}} \to \mr{GL}_2(k)$ be a continuous odd  representation.
Then the semisimplification of $\br$  is isomorphic to the semisimplification 
of a residual representation that arises from $\rho_{f, \iota}$ for a  newform $f \in S_r(\Gamma_1(N))$
for integers $N \geq 1$ and $r \geq 2$.
\end{lemma}

\begin{proof}
 In the case that $\br$ is irreducible, this is the original conjecture of Serre proved in \cite{khare-wintenberger:serre2}.
 In the case when $\br$ is reducible, this is a folklore  result, see \cite[\S 3.4]{skinner-wiles:reducible}. One first shows that
 the semisimplification of $\br$ arises from an Eisenstein series $E_r \in M_r(\Gamma_1(N))$ (for a positive integer $N$ which we may choose so that $N$ is divisible by a prime $q$ that is 1 mod $p$)  whose constant terms vanish modulo   a prime $\wp$  above $p$ of
 the field of Fourier coefficients of $E_r$. Then one shows that  $E_r$ is congruent modulo a prime above $p$  to a newform $f \in S_r(\Gamma_1(N))$.  An alternative approach is  to show that the semisimplification of $\br$ arises from the reduction $\overline E$ of an Eisenstein series $E\in M_r(\Gamma_1(N))$  (for a positive integer $N$;  we do not need  in this approach to ensure that $E$ has vanishing constant terms mod $p$) and then use that $\Theta^{p-1}(\overline E)$, with $\Theta$ the Ramanujan operator given by $q\frac{d}{dq}$,  is a mod $p$  cuspidal eigenform which lifts to a newform $f$ of weight $r \geq 2$ and level $N$ such that the semisimplification of $\br$  is isomorphic to the semisimplification 
of a residual representation that arises from $\rho_{f, \iota}$.
\end{proof}

Following the strategy of the proof of  \cite[Theorem 2]{ramakrishna-hamblen}, we now deduce from  Theorem \ref{mainthm}, specialised to the case $G=\mr{GL}_2$ and $F=\Q$,  and  the work of Skinner--Wiles (\cite{skinner-wiles:reducible})  and Lue Pan \cite{pan:redFM}, a modularity theorem for a reducible mod $p$ representation $\br$ without passing to its semisimplification. 

\begin{thm}\label{GL2}
Let $p \geq 3$ be a prime, and let $\br \colon \gal{\Q, {\mc{S}}} \to \mr{GL}_2(k)$ be a continuous representation of the form
\[
\br \sim \begin{pmatrix}
\bar{\chi} & \ast \\
0 & 1
\end{pmatrix}.
\]
Assume that
\begin{itemize}
\item $\bar{\chi}(c)=-1$.
\item $\bar{\chi} \neq \kbar^{- 1}$.
\end{itemize}
Then $\br$ is modular. More precisely: 
\begin{itemize}
\item For any integer $r \geq 2$, $\br$ is modular of weight $r$ and some level.
\item For $r=k(\br|_{\gal{\Q_p}})$ the weight associated to $\br|_{\gal{\Q_p}}$ in Serre's modularity conjecture, $\br$ is modular of weight $k(\br|_{\gal{\Q_p}})$ and level prime to $p$. 
\end{itemize}
\end{thm}
\begin{proof}
We first assume that the extension class $*$ is non-zero. Lemma \ref{GL2lemma}, Lemma \ref{locallifting}, and Corollary \ref{gl2bound} allow us to apply Theorem \ref{mainthm} to produce for any $r \geq 2$ a lift $\rho \colon \gal{\Q, \wt{{\mc{S}}}} \to \mr{GL}_2(\mc{O}')$ of $\br$ such that $\rho|_{\gal{\Q_p}}$ is potentially crystalline with Hodge--Tate weights $\{0, r-1\}$, satisfies $\rho|_{I_{\Q_p}} \sim \begin{pmatrix} * & * \\ 0 & 1 \end{pmatrix}$, and for $r=k(\br|_{\gal{\Q_p}})$ is moreover crystalline. Namely, we first fix a determinant for our lifts as follows. With the desired weight $r \geq 2$ fixed, we let $\mu= \kappa^{r-1}\cdot [\det(\br) \bar{\kappa}^{1-r}]$. Possibly replacing $\mc{O}$ by a ramified extension $\mc{O}'$ (with uniformizer $\vpi'$), there is then from the proof of Lemma \ref{locallifting} a potentially unramified character $\omega \colon \gal{\Q_p} \to (\mc{O}')^\times$ with trivial mod $\vpi'$ reduction such that $\br|_{\gal{\Q_p}}$ admits a potentially crystalline lift $\rho_p$ with determinant $\mu|_{\gal{\Q_p}} \omega$ and satisfying $\rho_p|_{I_{\Q_p}} \sim \begin{pmatrix} * & * \\ 0 & 1 \end{pmatrix}$. By the local and global Kronecker--Weber theorems, we can inflate $\omega$ to a character of $\gal{\Q}$ (also trivial mod $\vpi'$ and now unramified away from $p$). We then consider for the remainder of the argument local and global lifts with fixed determinant $\mu \omega$. By the previous observation and the $v \neq p$ part of Lemma \ref{locallifting}, we can now for all $v \in \mc{S}$ specify local lifts $\rho_v \colon \gal{\Q_v} \to \mr{GL}_2(\mc{O}')$ with determinant $\mu \omega$. For $v \neq p$, we choose any irreducible component of $R_{\br|_{\gal{\Q_v}}}^{\square, \mu \omega}[1/\vpi']$ containing $\rho_v$. For $v=p$, we fix the $p$-adic Hodge type $\mbf{v}$ equal to that of $\rho_p$, and we choose an irreducible component of $R^{\square, \mu \omega, \mbf{v}}_{\br|_{\gal{\Q_p}}}[1/\vpi']$ that is ordinary (cf Lemma \ref{ordinary}) of the same inertial type as $\rho_p$ and contains $\rho_p$. Lemma \ref{GL2lemma} now shows that all of the hypotheses of Theorem \ref{mainthm} are satisfied, so the latter produces (perhaps further enlarging $\mc{O}'$, and introducing an auxiliary set of primes $\widetilde{\mc{S}} \supset \mc{S}$) a lift $\rho \colon \gal{\Q, \widetilde{S}} \to \mr{GL}_2(\mc{O}')$ of $\br$ with all local restrictions to primes in $\mc{S}$ on the specified irreducible components of the local lifting rings. 

By the main theorem of \cite{skinner-wiles:reducible} (assuming $\bar{\chi}|_{\gal{\Q_p}} \neq 1$) and \cite{pan:redFM} (allowing $\bar{\chi}|_{\gal{\Q_p}}=1$), $\rho$ is modular, and the claim about the weight follows from local-global compatibility at $p$.

Finally, we treat the case when $\br$ is split. We first use Lemma \ref{redlift} to produce a
non-split extension of $1$ by $\bar \chi$. As in the previous paragraph we can apply Theorem \ref{mainthm} to
lift this to an irreducible and modular representation $\rho_{f, \iota}$, and then Lemma
\ref{Glattices} shows that some conjugate (over a ramified extension)
of $\rho_{f, \iota}$ reduces to $\br$. Alternatively, as pointed out by the referee, we can apply Lemma \ref{GL2-easy} and then conclude by Lemma \ref{Glattices}.
\end{proof}

\begin{rmk} {}
\begin{itemize}
\item Since the work of Skinner--Wiles also includes modularity of $p$-adic representations in the residually reducible case over certain totally real fields (\cite[Theorem A, Theorem B]{skinner-wiles:reducible}), we also obtain a mod $p$ modularity application in such settings.


\item The  method  of lifting reducible mod $p$ Galois representations to irreducible geometric $p$-adic representations  introduced in \cite{ramakrishna-hamblen}, and developed further here,  perhaps gives at the moment the only technique to access lifting and modularity of residual reducible representations. Lifting Galois representations has thus played a key role both in the proof of Serre's original conjecture \cite{khare-wintenberger:serre1} and \cite{khare-wintenberger:serre2}  for irreducible odd $\br:\Gamma_{\Q} \to \mr{GL}_2(k)$ and in its analogue for reducible representations (\cite{ramakrishna-hamblen}, Theorem \ref{GL2}).
\end{itemize} 

\end{rmk}

We now justify our statement in the introduction that there is no
possibility in general of producing minimal lifts of reducible $\br$, and 
that the lifting method of \cite{khare-wintenberger:serre0} will not
work in the residually reducible case. The following lemma is implicit
in \cite{berger-klosin:residualmodularity}.

\begin{lemma}\label{BK}
  Let $\Gamma$ be a finite group, $L$ an infinite field, and
  $\phi: \Gamma \rightarrow L^{\times}$ a homomorphism. Assume that either 
  $\dim_L H^1(\Gamma,L(\phi))>1$ or  $\dim_L H^1(\Gamma,L(\phi^{-1}))>1$. Then there are infinitely many
  representations $\rho: \Gamma \rightarrow \mr{GL}_2(L)$ with
  semisimplification ${\rm id} \oplus \phi$ that are not isomorphic,
  i.e., not conjugate under $\mr{GL}_2(L)$.
\end{lemma}

\begin{proof}
  The isomorphism classes of $\rho$ of the theorem are in bijection
  with the orbits of $H^1(\Gamma,L(\phi))$ and
  $H^1(\Gamma,L(\phi^{-1}))$ under the action of $L^{\times}$,
  namely with $\{0\} \cup \mathbb P^{n-1}(L) \cup \mathbb P^{m-1}(L) $
  for $n=\dim_L H^1(\Gamma,L(\phi))$ and
  $m=\dim_L H^1(\Gamma,L(\phi^{-1}))$ (with the convention that
  $\mathbb P^{-1}(L)=0$ and that we identify the two projective spaces
  when $\phi$ is trivial) . The number of orbits is finite if and only
  if $m,n \leq 1$.
\end{proof}

\begin{prop}\label{finite}

Let $\Gamma$ be a finite group that is a quotient of $\gal{\Q}$, and let $\bar \chi: \Gamma \rightarrow k^{\times}$ be a non-trivial odd character. Assume that $\dim_k H^1(\Gamma,k(\bar \chi))>1$. Then for any fixed integers  $N \geq 1$ and $r \geq 2$, only finitely many of the infinitely many  non-isomorphic  representations $\rho: \Gamma \rightarrow \mr{GL}_2(\overline{k})$ with semisimplification ${\rm id} \oplus \overline \chi$ (see Lemma \ref{BK}) can arise from reductions of  Galois-stable lattices in  $\rho_{f,\iota}:\gal{\Q} \to \mr{GL}_2(E)$ for  newforms $f \in S_{r}(\Gamma_1(N))$.
\end{prop}

\begin{proof}
Let $E$ be a finite extension of $\Q_p$ with residue field $k'$ that contains  $k$.
Consider an irreducible representation  $\rho:\gal{\Q} \to \mr{GL}_2(E)$ such that the semisimplification of the   residual representation arising  from (any)  $\gal{\Q}$-lattice in $E^2$ is $1 \oplus \bar \chi$. Note that $1 \neq \bar \chi$ as $\bar \chi$ is odd and  $p>2$. 

By an argument of Serre, which uses $ 1 \neq \bar \chi$, the fixed
points under $\rho(\gal{\Q})$ of the Bruhat-Tits tree for
$\mr{PGL}_2(E)$ form a segment. For a proof of this the reader may
consult \cite{bc-arbres}, especially Proposition 32, and for the
statements below see \cite[Proposition 11]{bc-arbres} (see also \cite[\S 1]{bellaiche:hawaiiribet} for detailed proofs).  The residual
representations arising from stable lattices are:
 \begin{enumerate}
 \item In the case the segment is an edge there are precisely two such representations, which correspond to the lattices that form the end points of the edge, whose reductions give a non-split extension of $1$ by $\bar \chi$ and vice-versa.
 \item In the case the segment is of length $>1$, there is the additional split representation, which arises by reduction of a lattice corresponding to a vertex that is an interior point of the segment.
 \end{enumerate}

 Denote by $X_E$ the classes of indecomposable representations of
 $\gal{\Q}$ with values in $\mr{GL_2}(k’)$ up to conjugacy by
 $\mr{GL}_2(k’)$, arising from reductions of $\gal{\Q}$-stable
 lattices in $E^2$. Then for any finite extension $E'/E$ there is a
 bijection between $X_E$ and $X_{E'}$.  Thus as
 $\dim_\mathbb C S_{r}(\Gamma_1(N))$ is finite, the number of
 isomorphism classes of representations of the form $\br_{f,\iota}$
 arising from $f \in S_{r}(\Gamma_1(N))$ is finite. Thus only
 finitely many of the infinitely many isomorphism classes of
 representations $\rho: \Gamma \rightarrow \mr{GL}_2(\overline{k})$
 with semisimplification ${\rm id} \oplus \overline \chi$ arise from
 $f \in S_{r}(\Gamma_1(N))$ for fixed integers $N \geq 1$, $r \geq 2$.
 \end{proof}

 \begin{rmk}
 
Note that given a character $\bar \chi: \Gamma_\Q  \rightarrow k^\times$, using Lemma \ref{gw}, there is finite quotient $\Gamma$ of $\Gamma_{\Q,{\mc{S}}}$, for ${\mc{S}}$ a sufficiently large finite set of places of $\Q$ containing $p$ and the places at which $\bar \chi$ is ramified, through which $\bar \chi$ factors, such that   $\dim_k H^1(\Gamma,k(\bar \chi))>1$.  The method of \cite{khare-wintenberger:serre0} produces liftings by producing upper and lower bounds on deformation rings $R$. The upper bound takes the form of showing that $R/pR$ is finite, while the lower bound says that $R$ is of positive dimension. The upper bounds rely on some version  of Lemma 3.15 of  \cite{dejong:conjecture}, which may  not be true for residually reducible representations. This  indicates that the method of \cite{khare-wintenberger:serre0} may not apply in the residually reducible case. Further the method of \cite{khare-wintenberger:serre0} when it works produces minimal liftings. Proposition \ref{finite} shows that these may  not exist in the residually reducible case. 

\end{rmk}

\section{The case of function fields}\label{functionfields}

In this section we briefly elaborate on the case of function fields,
where thanks to the work of L.~Lafforgue (\cite{llafforgue:chtoucas})
stronger automorphy results are possible for $G= \mr{GL}_n$. Let $F$
be a global function field of characteristic $\ell \neq p$, and let
$\br \colon \gal{F} \to G(k)$ be a representation satisfying the
hypotheses of Theorem \ref{mainthm}. Then $\br$ has, as in Theorem
\ref{mainthm}, a finitely-ramified lift
$\rho \colon \gal{F, \widetilde{{\mc{S}}}} \to G(\mc{O}')$. We now take
$G=\mr{GL}_n$ and deduce a stronger conclusion. We first recall the
relevant notion of automorphy in this setting. Fix an isomorphism
$\iota \colon \CC \xrightarrow{\sim} \ov{\Q}_p$. For each cuspidal
automorphic representation $\pi$ of $\mr{GL}_n(\mathbb{A}_F)$ with
central character of finite order (this hypothesis is not essential),
L. Lafforgue has in the main theorem of \cite{llafforgue:chtoucas}
constructed the Galois representation
$\rho_{\pi, \iota} \colon \gal{F} \to \mr{GL}_n(\ov{\Q}_p)$ associated
to $\pi$ (and $\iota$: note that $\iota$ is only implicit in the
notation of \textit{loc.~cit.}, and it is fixed after
\cite[Th\'{e}or\`{e}me VI.1.1]{llafforgue:chtoucas}). Any such
$\rho_{\pi, \iota}$ stabilizes a lattice in $E^n$ for some finite
extension $E$ of $\Q_p$ inside $\ov{\Q}_p$, and as in \S
\ref{GL2section} we say that $\br$ is automorphic if it is isomorphic,
after suitable extension of scalars, to the representation
$\bar{\rho}_{\pi, \iota} \colon \gal{F} \to
\mr{GL}_n(\mc{O}_E/\mf{m}_E)$ obtained by reducing some
$\rho_{\pi, \iota}$-stable lattice modulo $\mf{m}_E$.
\begin{thm}\label{FF}
Let $p \gg_n 0$, and let $\br \colon \gal{F, {\mc{S}}} \to \mr{GL}_n(k)$ be a continuous representation satisfying Assumptions \ref{generalhyp}, \ref{modp^Nhyp}, and the first part of Assumption \ref{finalhyp}. Then $\br$ is automorphic.
\end{thm}
\begin{proof}
Once we observe that the second part of Assumption \ref{finalhyp} automatically holds for $\br$, the theorem follows immediately from combining Theorem \ref{mainthm} with \cite[Th\'eor\`eme, pg. 2]{llafforgue:chtoucas}. That is, for $v \in {\mc{S}}$ we need to know that $\br|_{\gal{F_v}}$ has some $p$-adic lift; to arrange that the multiplier character of the local lift can be chosen to match the restriction of a fixed global lift $\mu$, we then use the fact that any character $\gal{F_v} \to (1+\vpi \mc{O})$ has an $n^{th}$ root, since $p$ does not divide $n$. The existence of some local lift of $\br|_{\gal{F_v}}$  follows as in \cite[Corollary 2.4.21]{clozel-harris-taylor}. In the context of that paper, $F_v$ is a finite extension of $\Q_{\ell}$, but the results of \cite[\S 2.4.4]{clozel-harris-taylor} only depend on
\begin{itemize}
\item the fact that the kernel of any surjection $I_{F_v} \to \Z_p$ has pro-order prime to $p$; and
\item the structure of the Galois group of the maximal tamely ramified, with $p$-power ramification index, extension of $F_v$.
\end{itemize}
These properties continue to hold in the equal characteristic setting; see the remarks in \S \ref{prelims}.
\end{proof}
\begin{rmk}
When $\br$ is absolutely irreducible, the hypotheses of this
  theorem are satisfied for all $p \gg_n 0$, which we could make
  explicit (as in \S \ref{GL2section}) if desired. This case of the
  theorem is already known: for $p \not \mid n$, de Jong's conjecture
  (\cite[Conjecture 2.3]{dejong:conjecture}) implies, by \cite[Theorem
  3.5]{dejong:conjecture} (and \cite[Theorem
  5.13]{bhkt:fnfieldpotaut}, which adapts de Jong's argument to allow
  $\br|_{\gal{F \ov{\mathbb{F}}}}$ reducible), that $\br$ has a
  finitely-ramified lift, and in fact one obtained without adding new
  primes of ramification. Gaitsgory proved de Jong's conjecture in
  \cite{gaitsgory:dejongconjecture}, modulo some results that are
  widely-held to be straightforward adaptations of theorems in the
  literature (namely, basic results in the theory of \'{e}tale
  $k[[t]]$-sheaves and, more importantly, a complete proof of the result
  announced in \cite[Theorem 14.1]{mirkovic-vilonen:geometricsatake}
  on the geometric Satake equivalence with general coefficients,
  including $k[[t]]$ and $k((t))$, and over any separably closed
  field).
\end{rmk}

When the semisimplification of $\br$ has absolutely irreducible
  image in a maximal proper Levi subgroup of $\mr{GL}_n$, Lemmas
  \ref{lem:h^0} and \ref{lem:h^1} give very concrete conditions under which Assumptions
  \ref{generalhyp}, \ref{modp^Nhyp}, and \ref{finalhyp} hold. We thus obtain Theorem \ref{fnfieldappintro} of the Introduction:
\begin{cor}\label{ffmax}
Let $p \gg_n 0$, and let $\br \colon \gal{F, {\mc{S}}} \to \mr{GL}_n(k)$ be a continuous representation that factors through a maximal parabolic $P$ with Levi quotient $M \cong \mr{GL}_{n_1} \times \mr{GL}_{n_2}$. Let $\br_M$ denote the projection $\br_M \colon \gal{F, {\mc{S}}} \to M(k)$ of $\br$ to $M$, and set $\br_M= \br_1 \oplus \br_2$, where $\br_i$ is the projection to the $\mr{GL}_{n_i}$ factor; we order the $\br_i$ such that $\br$ is isomorphic to an extension of $\br_2$ by $\br_1$. Moreover assume that $\br_M$ satisfies:
\begin{itemize}
\item $\br_M$ is absolutely irreducible. 
\item $\br_1$ is not isomorphic to $\br_2 \otimes \bar{\psi}$ for $\bar{\psi} \in \{1, \kbar^{-1}\}$. (In particular, this condition always holds if $n_1 \neq n_2$.) 
\item Let $\bar{\chi}$ be the character $\det(\br_1)^{n_2/d} \cdot \det(\br_2)^{-n_1/d}$, where $d = \mr{gcd}(n_1, n_2)$. Then $[F(\zeta_p):F(\zeta_p) \cap F(\bar{\chi})]$ is greater than a constant $b$ depending only on $n_1$ and $n_2$.
\end{itemize}
Then $\br$ is automorphic.
\end{cor}
\begin{proof}
In the notation of Lemma \ref{lem:h^1}, the character $M \to \ov{M} \to \ov{M}/\ov{M}^{\mr{der}} \xrightarrow{\sim} \mathbb{G}_m$ sends a pair $(g_1, g_2) \in \mr{GL}_{n_1} \times \mr{GL}_{n_2}$ to $\det(g_1)^{n_2/d} \cdot \det(g_2)^{-n_1/d}$ and thus agrees with the $\bar{\chi}$ of the present Corollary. Thus we may apply Lemma \ref{lem:h^1} (Case 1) to find that Assumption \ref{generalhyp} and the second part of Assumption \ref{modp^Nhyp} hold (and in fact that $\mu_p$ is not contained in $F(\br(\fgder))$). 

Now assume that we are in the case where $\br$ is not split, so the projection $\br(\gal{F}) \to M(k)$ is not injective. Then our assumptions allow us to apply Lemma \ref{lem:h^0} (Case (1)): indeed, the conditions (a) and (b) of Case (1) hold by assumption, and the condition that $[F(\zeta_p):F] \gg_{\mr{GL}_n} 0$ is implicit in the hypothesis on $[F(\zeta_p):F(\zeta_p) \cap F(\bar{\chi})]$. By Lemma \ref{lem:h^0}, then, the first part of Assumption \ref{modp^Nhyp} holds, as does the first part of Assumption \ref{finalhyp}. We can therefore apply Theorem \ref{FF} to deduce the Corollary in this case.

When on the other hand $\br \cong \br_1 \oplus \br_2$, we assume only that $\br_1 \not \cong \br_2$. If both $\br_1 \cong \br_2 \otimes \bar{\kappa}^{-1}$ and $\br_2 \cong \br_1 \otimes \bar{\kappa}^{-1}$, then $\br_1 \cong \br_1 \otimes \bar{\kappa}^{-2}$, and in particular $\bar{\kappa}^{2n_1}=1$; we may rule this possibility out using the condition, implicit in the third bulleted assumption, that $[F(\zeta_p):F] \gg_{\mr{GL}_n} 0$. Then possibly reordering the representations we may assume $\br_1 \not \cong \br_2 \otimes \bar{\kappa}^{-1}$. As in the proof of Corollary \ref{cr-lift}, we then apply Lemma \ref{levi} to $\br_1 \oplus \br_2$ to produce a non-split representation $\br'$, an extension of $\br_2$ by $\br_1$ ramified at an enlarged set $\mc{S} \cup \mc{T}$ of primes; and now since $\br_1 \not \cong \br_2 \otimes \kappa^{-1}$ we find that $H^0(\gal{F, \mc{S} \cup \mc{T}}, \br'(\fgder)^*)=0$ by condition (b) of Case (1) of Lemma \ref{lem:h^0}. Likewise (by condition (a) of Case (1) of Lemma \ref{lem:h^0}) our assumption $\br_1 \not \cong \br_2$ implies $H^0(\gal{F, \mc{S} \cup \mc{T}}, \br'(\fgder))=0$. We complete the proof in this case by combining Theorem \ref{FF}, all of whose hypotheses are now satisfied, with Lemma \ref{Glattices}.
\end{proof}
Note that while the asymmetry of the statement implies that when $\br
\cong \br_1 \oplus \br_2$ we only exclude from the second bullet-point
the case $\br_1 \cong \br_2$, the third bullet point imposes a strong enough hypothesis ($[F(\zeta_p):F(\bar{\kappa}^{n_1})] \gg_{n_1} 0$)
when $\br_1 \cong \br_2 \otimes \bar{\kappa}$ to make this result useless in this case. Nevertheless, we have left the proof in its present form because in some situations it is possible to verify Assumption \ref{generalhyp} and the second part of Assumption \ref{modp^Nhyp} without using the third bulleted assumption (i.e., without relying on Lemma \ref{lem:h^1}): for instance, the $\bar{\chi}=\bar{\kappa}$ case of Lemma \ref{GL2lemma} provides such an example when $n=2$.
When $\br$ itself is semisimple, we further have the following application:
\begin{cor}\label{ffss}
Let $p \gg_n 0$, and let $\br \colon \gal{F, {\mc{S}}} \to \mr{GL}_n(k)$ be a continuous representation equal to, for some $r>1$, a direct sum $\br= \bigoplus_{i=1}^r \br_i$ of representations $\br_i \colon \gal{F, {\mc{S}}} \to \mr{GL}_{n_i}(k)$. Assume that $\br$ satisfies:
  \begin{enumerate}
  \item $\br_{i}$ is absolutely irreducible for $i=1,2, \ldots,r$.
  \item For all $i \neq j \in \{1,2,\dots,r\}$, $\br_i$ is not isomorphic to $\br_j \otimes \bar{\psi}$ for any
    $\bar{\psi} \in \{1, \kbar, \kbar^{-1}\}$  (In particular, this condition
    always holds if $n_i \neq n_j$.)
  \item Let $\bar{\chi}_{i,j}$ be the character
    $\det(\br_i)^{n_j} \cdot \det(\br_j)^{-n_i}$.
    Then $[F(\zeta_p):F(\zeta_p) \cap F(\bar{\chi}_{i,j})]$ is greater
    than a constant $b$ depending only on $n$, for all
    $i \neq j \in \{1,2,\dots,r\}$.
\end{enumerate}
Then $\br$ is automorphic.
\end{cor}
\begin{proof}
Lemma \ref{l:gln} (and the remarks in the proof of Theorem \ref{FF}) shows that $\br$ satisfies the hypotheses of Corollary \ref{cr-lift}, so we apply the latter to produce an irreducible lift $\rho$ with finite-order determinant, which again by \cite{llafforgue:chtoucas} is automorphic.
\end{proof}

\section{Remarks on the higher-rank case over number fields}\label{numberfieldGLn}
For $G$ of semisimple rank greater than 1, the results of this paper
do not at present imply modularity of residually reducible odd
representations $\br \colon \gal{F} \to G(k)$ when $F$ is a totally
real field, since we no longer can invoke the work of
Skinner--Wiles. Of course, the Fontaine--Mazur and Langlands conjectures
in combination do predict the automorphy of the geometric lifts we
construct, and we can hope for future generalizations of
\cite{skinner-wiles:reducible}. While the force of our application of
\cite{skinner-wiles:reducible} is in studying non-split $\br$, the
split case for $\mr{GL}_2$ being well-known, for higher-rank $G$ the
lifting results we prove in the split ($G$-completely reducible) case
are novel, and we find it worthwhile to compare them with the recent
important works \cite{thorne:reducible} and \cite{allen-newton-thorne:reducible}, which prove automorphy lifting theorems for residually reducible representations $\rho \colon \gal{L} \to \mr{GL}_n(\Qpb)$ where $L$ is a CM field, and $\rho$ is a suitable (we pass over the precise hypotheses) essentially conjugate self-dual representation.

For simplicity we focus on the case $G= \mr{GSp}_{2n}$. We distinguish
between a representation $\br \colon \gal{F} \to \mr{GSp}_{2n}(k)$
being $\mr{GSp}_{2n}$-irreducible and being
$\mr{GL}_{2n}$-irreducible. In Proposition \ref{CS} we will combine
our lifting theorems (in fact, the main theorem of \cite{fkp:reldef}
is sufficient) with the potential automorphy theorem of
\cite{allen-newton-thorne:reducible} to lift many
$\mr{GSp}_{2n}$-irreducible but $\mr{GL}_{2n}$-reducible
representations of $\gal{F}$ to compatible systems of $\ell$-adic
representations. (Note that the results of
\cite{allen-newton-thorne:reducible} were not available when
\cite{fkp:reldef} and \cite{fkp:SOeg} were written.) 

We will first produce using our lifting results a geometric $p$-adic ordinary lift $\rho \colon \gal{F} \to \mr{GSp}_{2n}(\ov{\Q_p})$, assuming the existence of ordinary lifts of $\br|_{\gal{F_v}}$ for $v \vert p$; we moreover require a generalization to other reductive groups (here just $\mr{GSp}_{2n}$) of Geraghty's results (\cite{geraghty:ordinary}) on ordinary local deformation rings, which we work out in Appendix \ref{ordappendix}. In order to apply the results of \cite{allen-newton-thorne:reducible}, we then need to verify the residual cuspidal ordinary automorphy of $\br|_{\gal{L'}}$ for some CM extension $L'/F$.
We do this using the potential automorphy
theorems of \cite{blggt:potaut}, and then we can strengthen our result producing a $p$-adic
lift $\rho$ of $\br$ by showing that $\rho$ belongs to a compatible
systems of $\gal{F}$-representations. Here without striving for
maximal generality we give a sample result; note that our restriction
to the symplectic case greatly simplifies the deduction of potential
automorphy, but something can be said in other cases as well.

We have to introduce some notation. We follow \cite[\S 2]{thorne:reducible} and refer the reader there for a more detailed discussion. If $L$ is a CM field, and $\pi$ is a regular algebraic conjugate self-dual cuspidal (RACSDC) automorphic representation of $\mr{GL}_n(\mathbb{A}_{L})$, then for any isomorphism $\iota \colon \CC \xrightarrow{\sim} \Qpb$ there is an associated $p$-adic Galois representation $r_{\iota}(\pi) \colon \gal{L} \to \mr{GL}_n(\Qpb)$, satisfying local-global compatibility at all places of $L$, whose construction represents the culmination of work of many people: we refer here only to \cite{chenevier-harris:autgal2} for the construction and to \cite{caraiani:loc-glob1}, \cite{caraiani:loc-glob2} for the completion of the proof of local-global compatibility. We write $\bar{r}_{\iota}(\pi)$ for the semisimplification of the reduction of any $\gal{L}$-stable lattice; this is the only canonical mod $p$ representation associated to $\pi$ and $\iota$. We say that $\pi$ is $\iota$-ordinary if it satisfies the conditions in \cite[Lemma 2.3]{thorne:reducible}. If $K$ is a finite extension of $\Q_p$, and we are given for each embedding $\tau \colon K \into \Qpb$ an element $\xi_\tau \in \Z^n_+=\{(\xi_{1}, \ldots, \xi_{n}) \in \Z^n: \xi_{i} \geq \xi_{i+1} \forall i\}$, we say that a $p$-adic representation $\rho \colon \gal{K} \to \mr{GL}_n(\Qpb)$ is $\xi= (\xi_\tau)_{\tau}$-ordinary if it satisfies the conditions of \cite[Definition 2.5]{thorne:reducible}. In particular, we note that $\rho$ is de Rham. To see this, normalize Hodge--Tate weights as in \cite[Notation]{thorne:reducible}, so that the cyclotomic character has Hodge--Tate weight -1. Then for all $\tau \colon K \to \Qpb$, $\rho$ is an iterated extension of characters 
\[
\rho|_{\gal{K}} \sim
\begin{pmatrix}
\psi_1 & * & * & * \\
0 & \psi_2 & * & * \\
\vdots & \ddots & \ddots & * \\
0 & \cdots & 0 & \psi_n 
\end{pmatrix},
\]
where the $\tau$-labeled Hodge--Tate weight of $\psi_i$ is $-(\xi_{\tau, n-i+1}+i-1)$, which is strictly less than the corresponding value $-\xi_{\tau, n-i}-i$ for $\psi_{i+1}$. By work of Bloch--Kato, such an extension is automatically de Rham: see for instance \cite[Proposition 1.28]{nekovar:p-adicheight} for a precise reference.
\begin{prop}\label{CS}
Let $F$ be a totally real field, $n \geq 1$ an integer, and $p \gg_n 0$ a prime
such that $[F(\zeta_p):F]$ is greater than the integer $a_{\mr{GSp}_{2n}}$ of \cite[Lemma A.6]{fkp:reldef}. Let $\br \colon \gal{F, {\mc{S}}} \to \mr{GSp}_{2n}(k)$ be a homomorphism with absolutely $\mr{GSp}_{2n}$-irreducible restriction to $\gal{F(\zeta_p)}$ and with similitude character $\bar{\kappa}^{1-2n}$. Assume the following:
\begin{itemize}
\item For all $v \vert p$, $\br|_{\gal{F_v}}$ admits a lift $\rho_v \colon \gal{F_v} \to \mr{GSp}_{2n}(\mc{O})$ with similitude character $\kappa^{1-2n}$ such that the composite $\rho_v \colon \gal{F_v} \to \mr{GSp}_{2n}(\Qpb) \to \mr{GL}_{2n}(\Qpb)$ is ordinary of some weight $\xi_v= (\xi_{\tau})_{\tau} \in (\Z^n_+)^{\Hom_{\Q_p}(F_v, \Qpb)}$. 
\item $\br$ (under the composite $\mr{GSp}_{2n} \to \mr{GL}_{2n}$) is not induced from any proper subgroup of $\gal{F}$.
\end{itemize}
Then there exists a finite set of primes $\widetilde{{\mc{S}}} \supset {\mc{S}}$, a finite extension of $\mr{Frac}(\mc{O})$ with ring of integers $\mc{O}' \supset \mc{O}$, and
a geometric lift $\rho \colon \gal{F, \widetilde{{\mc{S}}}} \to
\mr{GSp}_{2n}(\mc{O}')$ of $\br$ with Zariski-dense image, such that (fixing an embedding $\mc{O}' \into \Qpb$) $\rho$ belongs to a strictly pure compatible system $\rho_\iota \colon \gal{F} \to \mr{GSp}_{2n}(\Qlb)$ of $\ell$-adic representations, indexed over $\iota \colon \CC \xrightarrow{\sim} \Qlb$, and each $\rho_{\iota}$ has Zariski-dense image in $\mr{GSp}_{2n}$.
\end{prop}
\begin{proof}
Fix $\iota \colon \CC \xrightarrow{\sim} \Qpb$. We claim that \cite[Theorem
A]{fkp:reldef} implies there is a lift $\rho \colon \gal{F, \widetilde{{\mc{S}}}}
\to \mr{GSp}_{2n}(\mc{O})$ of $\br$ with similitude character $\mu= \kappa^{1-2n}$ that in $\mr{GL}_{2n}$ is ordinary of weight
$\xi= (\xi_v)_{v \vert p}$; moreover we can arrange that for some place $v_0$ of $F$ such that $\br|_{\gal{F_{v_0}}}=1$ and $N(v_0) \equiv 1 \pmod p$, $\rho|_{\gal{F_{v_0}}}$ is isomorphic to an unramified twist of the Steinberg parameter. In light of the hypotheses of Proposition \ref{CS}, we just have to check the local lifting hypotheses. For $v \not \mid p$ in ${\mc{S}}$, Booher's work (\cite{booher:minimal}) shows that $\br|_{\gal{F_v}}$ has a lift $\gal{F_v} \to \mr{GSp}_{2n}(\mc{O}')$ with multiplier $\kappa^{1-2n}$, and that such a lift lies on an irreducible component of $R^{\square, \kappa^{1-2n}}_{\br|_{\gal{F_v}}}$ isomorphic to $\mc{O}'[[X_1, \ldots, X_{\dim(G^{\mr{der}})}]]$. For $v \vert p$, we first note that the given lifts $\rho_v \colon \gal{F_v} \to \mr{GSp}_{2n}(\ov{\Q}_p)$ are ordinary of some regular weight in the sense of Definition \ref{orddef} below. Indeed, by the regularity of the $\tau$-labeled Hodge--Tate weights, there is a unique $\gal{F_v}$-stable filtration $0 \subsetneq F_1 \subsetneq \cdots \subsetneq F_{2n}= (\ov{\Q}_p)^{2n}$ such that $F_{i}/F_{i-1}$ is a line on which $\gal{F_v}$ acts with $\tau$-labeled Hodge--Tate weights $-(\xi_{\tau, 2n-i+1}+i-1)$. As the Hodge--Tate weights increase with $i$, with respect to the pairing $\rho_v \times \rho_v \to \kappa^{1-2n}$, $F_1 \subsetneq \cdots \subsetneq F_n$ must be a maximal isotropic flag. Thus $\rho_v$ factors through a Borel subgroup $B \subset \mr{GSp}_{2n}$ and is ordinary in the sense of Definition \ref{orddef}. To apply Theorem \ref{mainthmintro} we must specify a quotient $\ov{R}_v[1/p]$ of $R^{\square, \kappa^{1-2n}}_{\br|_{\gal{F_v}}}[1/p]$ containing $\rho_v$ that has an open dense regular subscheme and is equidimensional of dimension $\dim(G^{\mr{der}})+[F_v:\Q_p]\dim(\mr{Fl}_G)$; moreover, for the application we must ensure that $\ov{\Q}_p$-points of $\ov{R}_v[1/p]$ correspond to ordinary representations. In Lemma \ref{ordinary}, deferred until Appendix \ref{ordappendix}, we explain how to do this: in a nutshell, we show (following \cite{geraghty:ordinary}) that imposing an ordinarity condition on potentially semistable deformation rings for any reductive group $G$ cuts out a union of irreducible components. Finally, at an auxiliary place $v_0$ of $F$, \cite[Theorem A]{fkp:reldef} (or Theorem \ref{mainthmintro} of the present paper) allows us to construct our global lift $\rho$ such that $\rho|_{\gal{F_{v_0}}}$ lies on (and only on) a Steinberg component of $R_{\br|_{\gal{F_{v_0}}}}^{\square, \kappa^{1-2n}}$. More precisely, we choose a trivial prime $v_0$ and include it in $\mc{S}$. We then fix a lift $\rho_{v_0}$ whose associated Weil--Deligne representation is isomorphic to a twist of the Steinberg parameter. This gives a formally smooth point of $R_{\br|_{\gal{F_{v_0}}}}^{\square, \kappa^{1-2n}}$, as is easily verified using \cite{bellovin-gee-G} (see \cite[Lemma 3.7]{fkp:reldef} for a similar argument). By prescribing (as in the proof of Theorem \ref{mainthmintro}) our lift of $\br|_{\gal{F_{v_0}}}$ modulo a sufficiently high power $\vpi^t$, we can then guarantee that $\rho|_{\gal{F_{v_0}}}$ has associated Weil--Deligne representation isomorphic to a twist of the Steinberg parameter.

Regarded as a $\mr{GL}_{2n}$-representation, $\br$ is semisimple, and we decompose $\br$ as $\mr{GL}_{2n}$-representation into a direct sum of absolutely irreducible representations $\br= \oplus_{i=1}^r \br_i$. Note that our assumption that $\br|_{\gal{F(\zeta_p)}}$ is $\mr{GSp}_{2n}$-irreducible implies that each $\br_i$ satisfies $\br_i \cong \br_i^\vee(\bar{\kappa}^{1-2n})$ and remains absolutely irreducible when restricted to $\gal{F(\zeta_p)}$.

By \cite[Theorem 3.1.2]{blggt:potaut}, there is a Galois totally real
extension $F'/F$ linearly disjoint from $F(\br, \zeta_p)$ such that
$\br|_{\gal{F'}}$ is automorphic. More precisely, we apply
\textit{loc.~cit.}~to $\br \colon \gal{F} \to \mr{GSp}_{2n}(k)$ (thus
the set $I$ is a singleton---note that $\br$ is allowed to be
reducible in \textit{loc.~cit.}) and with $F^{\mr{avoid}}/F$ equal to
$F(\br, \zeta_p)$. We also make a small alteration to the proof to produce an automorphic representation that is Steinberg at places above $v_0$. The proof of \cite[Theorem 3.1.2]{blggt:potaut} applies Moret-Bailly's theorem
(\cite[Proposition 3.1.1]{blggt:potaut}) to a scheme (over
$F(\zeta_N)^+$ for a judiciously-chosen integer $N$) there denoted
$\widetilde{T}$, which in our case is in the notation of
\cite[pg. 549]{blggt:potaut} simply $T_{\br \times \br'}$ for a
suitable mod $p'$ (for a well-chosen $p' \neq p$) representation
$\br'$. There is a morphism $t \colon \widetilde{T} \to \mathbb{P}^1-\{\mu_N, \infty\}$ of $F(\zeta_N)^+$-schemes. Moret-Bailly's theorem allows one to find a point $P \in
\widetilde{T}(F')$ for some finite Galois extension $F'/F$ (containing
$F(\zeta_N)^+$ and linearly disjoint over $F(\zeta_N)^+$ from $F(\br,
\br', \mu_{Npp'})$) such that $P$ satisfies a finite collection of
pre-specified local conditions (see the statement of \cite[Proposition
3.1.1]{blggt:potaut}). In \cite[Proposition 3.1.2]{blggt:potaut}, the
local conditions are only specified at places of $F$ above $p$, $p'$,
and $\infty$ (see the bullet-points on pg. 549,
\textit{loc.~cit.}). We further add the requirement that for all
places $v$ above $v_0$, $P \in \widetilde{T}(F')$ should also satisfy
the local condition $v(t(P))<0$.
Then there is a regular algebraic, self-dual, cuspidal automorphic representation $\pi$ of $\mr{GL}_{2n}(\mathbb{A}_{F'})$ such that:
\begin{itemize}
\item $\br|_{\gal{F'}} \cong \bar{r}_{\iota}(\pi)$;
\item $\pi$ is $\iota$-ordinary of weight 0;
\item $\pi_{v_0}$ is isomorphic to an unramified twist of the Steinberg representation; this desideratum follows from the choice of $P$, \cite[Lemma 5.1(2)]{blght:cy2}, and local-global compatibility for the Galois representations associated to $\pi$.
\end{itemize}
(Note that the proof of \cite[Theorem 3.1.2]{blggt:potaut} shows that $\pi$ is self-dual rather than merely essentially self-dual.)

Choose a quadratic CM extension $L/F$, setting $L'=LF'$ such that $L'$ (also Galois over $F$) is linearly disjoint over $F$ from $F(\br, \zeta_p)$, and consider the restriction $\rho|_{\gal{L'}}$ and the base-change $\mr{BC}_{L'/F'}(\pi)$. We verify that $\rho|_{\gal{L'}}$ satisfies the hypotheses of \cite[Theorem 1.1]{allen-newton-thorne:reducible} and is therefore automorphic: there is a RACSDC automorphic representation $\Pi$ of $\mr{GL}_{2n}(\mathbb{A}_{L'})$ such that $\rho|_{\gal{L'}} \cong r_{\iota}(\Pi)$. Indeed, by construction hypotheses (1)-(6) of \textit{loc.~cit.}~are clear (recall we have chosen $L'$ linearly disjoint over $F$ from $F(\br)$). Hypothesis (8) can be taken to be part of our $p \gg_G 0$ assumption. The various parts of hypothesis (7) are treated as follows:
\begin{itemize}
\item By the assumption $[F(\zeta_p):F]>a_{\mr{GSp}_{2n}}$ and the linear disjointness of $L'/F$ from $F(\zeta_p, \br)/F$, we see that $L'(\zeta_p)$ is not contained in $L'(\br(\fgder))$, $L'$ is not contained in $F'(\zeta_p)$ and the $\br_i|_{\gal{L'(\zeta_p)}}$ remain absolutely irreducible and distinct.
\item The integer $a_{\mr{GSp}_{2n}}$ is defined to guarantee that the image of $\br$ modulo the center of $\mr{GSp}_{2n}$ has no cyclic quotient of order greater than $a_{\mr{GSp}_{2n}}$. In particular, for $p > a_{\mr{GSp}_{2n}}$, the image of $\br= \br^{\mr{ss}}$ has no quotient of order $p$, and by the linear disjointness the same remains true after restriction to $\gal{L}$.
\item We have assumed that $\br$, viewed as $\mr{GL}_{2n}$-valued, is
  not induced from a proper subgroup of $\gal{F}$. Again since $L'$ is
  linearly disjoint from $F(\br)$ over $F$, the same holds for
  $\br|_{\gal{L'}}$, since the image of the representation determines whether or not it is induced.
\end{itemize}
Thus we can invoke \textit{loc.~cit.}~to produce such a $\Pi$. Moreover, $\Gal(L'/F')$-invariance of $\rho|_{\gal{L'}}$ implies the same for $\Pi$, and therefore $\Pi$ descends to RAESDC automorphic representation $\Pi_{F'}$ of $\mr{GL}_{2n}(\mathbb{A}_{F'})$. Since $\Pi_{F'}$ has an associated Galois representation, it is easy to see that we can alter the choice of descent (by $\delta_{L'/F'}$) to arrange that $\rho|_{\gal{F'}} \cong r_{\iota}(\Pi_{F'})$. 

Now we can apply the argument of \cite[Theorem 5.5.1]{blggt:potaut} to
conclude that $\rho$ belongs to a strictly pure compatible system (in
the terminology of \textit{loc.~cit.}) of $\ell$-adic representations
$\{\rho_{\iota'}\}$ of $\gal{F}$, indexed over primes $\ell$ and
isomorphisms $\iota' \colon \CC \xrightarrow{\sim} \Qlb$: the argument
here is greatly simplified since $\rho$, and hence $\rho|_{\gal{F'}}$
has Zariski-dense image in $\mr{GSp}_{2n}$, and the purity moreover
implies that each $\rho_{\iota'}$ also has Zariski-dense image in
$\mr{GSp}_{2n}$. For the latter point, note that each $\rho_{\iota'}$
is essentially self-dual (since $\rho$ is), and compatibility at $v_0$
implies that the image of each $\rho_{\iota'}$ contains a regular
unipotent element of $\mr{GL}_{2n}$ (which equals a regular unipotent
element of $\mr{GSp}_{2n}$). It follows from a theorem of Dynkin that
the Zariski-closure of the image of $\rho_{\iota'}$ is either a
principal $\mr{GL}_2$ or all of $\mr{GSp}_{2n}$ (note the self-duality
is necessarily symplectic, since $\mr{SO}_{2n}$ does not contain a
regular unipotent element of $\mr{GL}_{2n}$). Independence of $\ell$
of the formal character (and in particular the rank of a maximal
torus) excludes the possibility of $\mr{GL}_2$ (unless $n=1$).
\end{proof}

See also the discussion in \cite{fkp:SOeg} for the difficulties in
applying the Khare--Wintenberger lifting method to construct lifts
(and thus compatible systems) in $\mr{GL}$-reducible settings, and for
more ``endoscopic'' examples where our lifting results apply. That said, as the referee points out, \cite[Corollary 5.2]{newton-thorne:symmetric1} constructs essentially conjugate self-dual lifts of $\br$ when restricted to a CM extension of $F$. We don't know whether this method can at present produce the $\mr{GSp}_{2n}$ compatible system over $F$ itself of Proposition \ref{CS}.

\appendix
\section{} \label{appendix}

In this appendix we discuss some explicit conditions on $\br$ which
imply that Assumptions \ref{generalhyp}, \ref{modp^Nhyp}, and (the
first part of) \ref{finalhyp} hold.  For ease of reference we label
the properties we shall check as follows:
\begin{enumerate}[C]
\item $\br(\fgder)$ does not contain the trivial representation as a
  submodule.
\item $\br(\fgder)^*$ does not contain the trivial representation as a
  submodule. 
\item There is no surjection of $\Fp[\gal{F}]$-modules
  $\br(\fgder) \onto W$ with $W$ a nonzero $\Fp[\gal{F}]$-module
  subquotient of $\br(\fgder)^*$.
\item $H^1(\Gal(K/F), \br(\fgder)^*)=0$.
\end{enumerate}

The case that $\br$ is irreducible was discussed in \cite{fkp:reldef}
so here we assume that $\br$ factors through $P(k)$, where $P$ is a
proper parabolic subgroup of a (split) connected reductive group
$G$. We may and do assume that $P$ is minimal with respect to this
property (though it need not be unique in general).  Let
$\pi: P \to M$ be the Levi quotient of $P$, and
$\br_M := \pi \circ \br$ the induced map $\Gamma_F \to M(k)$. Let
$U$ be the unipotent radical of $P$, $\mf{u} = \Lie(U)$, $U^-$
the unipotent radical of the opposite parabolic and
$\mf{u}^- = \Lie(U^-)$. We let $\br(\mf{u})$ (resp.~$\br(\mf{u}^-)$
be $\mf{u}$ (resp.$~\mf{u}^-$) viewed as a $\Gamma_F$-module (via
the adjoint action of $P(k)$).

Let $\{\mf{u}_i\}$ be the ascending central series of the Lie
algebra $\mf{u}$ and $\{U_i\}$ the ascending central series of the
unipotent group $U$ (over the field $k$). In order to formulate some
of our criteria cleanly, we will use the following:
\begin{lemma} \label{l:exp}%
  If $p \gg_G 0$, there is a $P$-equivariant isomorphism of algebraic
  varieties over $k$ $\exp:\mf{u} \to U$ (with $\mf{u}$ viewed
  as an affine space). For each $i$, this restricts to an isomorphism
  of varieties $\exp: u_i \to U_i$ which induces a
  $P$-equivariant isomorphism of algebraic groups
  $\mf{u}_{i+1}/\mf{u}_{i} \to U_{i+1}/U_i$.
\end{lemma}
We denote the inverse of $\exp$ by $\log: U \to \mf{u}$.

\begin{proof}
  The lemma is well-known (for arbitrary unipotent groups) over fields
  of characteristic zero and since $G$ is split reductive it follows
  for all $p \gg_G 0$ by spreading out using the split reductive model
  of $G$ over $\Z$.
\end{proof}

\begin{rmk}
  For classical groups, how large $p$ has to be for the lemma to hold
  can easily be made effective, e.g., for $\mr{GL}_n$ it suffices to
  take $p>n$.
\end{rmk}

Let $\Gamma$ be any group, $\Sigma$ a finite $p$-group and
$f:\Gamma \to \Aut(\Sigma)$ a homomorphism giving rise to an action of
$\Gamma$ on $\Sigma$. If  $\Sigma' \subset \Sigma''$ are normal
$\Gamma$-invariant subgroups of $\Sigma$ such that $\Sigma''/\Sigma'$
is contained in the centre of $\Sigma/\Sigma'$, then $f$ induces the
structure of $\Z[\Gamma]$-module on $\Sigma''/\Sigma'$.

Suppose
\[
  \{1\} = \Sigma_0 \subsetneq \Sigma_1 \subsetneq \Sigma_2 \subsetneq \cdots
  \subsetneq \Sigma_n = \Sigma
\]
is a chain of normal $\Gamma$-invariant subgroups of $\Sigma$ such
that $\Sigma_{i+1}/\Sigma_i$ is contained in the centre of
$\Sigma/\Sigma_i$ for all $i$; note that such chains exist because
$\Sigma$ is nilpotent and any chain can be refined to a maximal
one. If the chain is maximal, then each $\Sigma_{i+1}/\Sigma_i$,
$i=1,\dots,n$, is a simple $\Fp[\Gamma]$-module.

\begin{lemma} \label{l:JH}%
  The Jordan--H\"{o}lder property holds for maximal chains as above, i.e.,
  the set (with multiplicities) of simple $\Fp[\Gamma]$-modules
  occurring as subquotients is independent of the choice of maximal
  chain.
\end{lemma}

\begin{proof}
  This is proved using standard methods, see, e.g., \cite[Theorem
  10.5]{isaacs-algebra}.
\end{proof}

Now let $\Gamma$ be any subgroup of $P(k)$, $\ov{\Gamma}$ its image in
$M(k)$ and $\Sigma$ the kernel of the surjection
$\Gamma \to \ov{\Gamma}$. Since $\Sigma \subset U(k)$, it is a
$p$-group and we let $\{\Sigma_i\}$ be the subgroups given by the
ascending central series of $\Sigma$. The group $\Gamma$ acts on each
quotient $\Sigma_{i+1}/\Sigma_i$.
\begin{lemma} \label{l:subq}%
  Suppose $p \gg_G 0$.  Then for all $i$, any irreducible
  $\Gamma$-subquotient of $\Sigma_{i+1}/\Sigma_i$ is isomorphic to an
  $\Fp[\Gamma]$-subquotient of $\mf{u}$ (with the $\Gamma$ action
  induced from the natural action of $P(k)$).
\end{lemma}

\begin{proof}
  Since $U(k)$ is a $p$-group, we may replace $\Gamma$ with
  $\pi^{-1}(\ov{\Gamma})$ and assume that $\Sigma = U(k)$. Let
  $\{\mf{u}_i\}$ be the ascending central series of the $k$-Lie
  algebra $\mf{u}$. By Lemma \ref{l:exp} the exponential map
  identifies $\mf{u}_{i+1}/u_i$ with $U_{i+1}(k)/U_{i}(k)$,
  compatibly with the $P(k)$-actions, hence \emph{a fortiori} with the
  $\Gamma$-actions, so the lemma follows from the Jordan--H\"older
  property (Lemma \ref{l:JH}).
\end{proof}

Properties (C3) and (C4) are quite subtle and not easy to verify in
general, but the following lemma provides a simple sufficient
condition for all $G$ and $P$.
\begin{lemma} \label{l:allP} %
  If $p \gg_G 0$ and the $p^{th}$ roots of unity are not contained in $F(\br(\fgder))$
  then (C3) and (C4) hold.
\end{lemma}

\begin{proof}
  We will prove that the stronger version of properties (C3) and (C4)
  hold, with $\br(\fgder)$ replaced by its semisimplification. We also
  note that by enlarging $k$ if necessary (which clearly doesn't
  change the hypothesis) we may assume that $\br_M$ is absolutely
  irreducible and any irreducible
  $k[\Gamma_F]$-subquotient of $\br(\fgder)$ is absolutely
  irreducible.

  Suppose there is a nonzero $\Fp[\Gamma_F]$-linear map from a
  subquotient of $\br(\fgder)$ onto a submodule $W$ of
  $\br(\fgder)^*$, which we may assume is irreducible as an
  $\Fp[\Gamma_F]$-module. Equivalently, $\br(\fgder)$ has an
  $\Fp[\Gamma_F]$-subquotient $V$ which is isomorphic to $W$. Since
  both $\br(\fgder)$ and $\br(\fgder)^*$ are $k$-modules, it follows
  that there exists a $k[\Gamma_F]$-subquotient $V'$ of $\br(\fgder)$
  and a $k[\Gamma_F]$-subquotient $W'$ of $\br(\fgder)^*$ which are
  isomorphic as $\Fp[\Gamma_F]$-modules.  Clearly $W'$ must be
  isomorphic to $V'' \otimes_k \kbar$, where $V''$ is a
  $k[\Gamma_F]$-subquotient of $\br(\fgder)$. Thus, $V''$ and
  $V'' \otimes_k \kbar$ are both subquotients of $\fgder$ for a
  nonzero $V''$, hence the $p^{th}$ roots of unity must be contained
  in $F(\br(\fgder))$. This contradiction proves that (C3) holds.

  Let $F' = F(\br_M, \kbar)$. Clearly, $F \subset F' \subset K$.
  Since $\br_M$ is absolutely irreducible and $p \gg_G 0$, it follows
  from \cite[Theorem A]{guralnick:CR} (see also the argument of
  \cite[Lemma A.1]{fkp:reldef}) that $H^1(\Gal(F'/F), W) = 0$ for any
(absolutely) irreducible 
  $k[\Gamma_F]$-subquotient $W$ of $\br(\fgder)^*$; here we use the
  fact that the kernel of the map $\Gal(F'/F) \to G^{\mr{ad}}(k)$
  (induced by $\br_M$) is of order prime to $p$. By the inflation-restriction sequence, it follows that
  (C4) will hold if we can show that
  $H^1(\Gal(K/F'), W))^{\Gal(F'/F)} = 0$. By the choice of $F'$, the
  Sylow $p$-subgroup $Q$ of $\Gal(K/F')$ is a subgroup of $U(k)$.
  Furthermore, $Q$ acts trivially on $W$ since it is an irreducible
  $\Gal(K/F)$-module and $Q$ is a normal $p$-subgroup. Thus it
  suffices to show that there are no non-trivial
  $\Gal(F'/F)$-equivariant homomorphisms from $Q^{\mr{ab}}$ to $W$. By
  Lemma \ref{l:subq}, any irreducible $\Fp[\Gal(F'/F)]$-quotient of
  $Q$ is isomorphic to an $\Fp[\Gal(F'/F)]$-subquotient of
  $\mf{u}$. Since $\mf{u} \subset \fgder$ and $W$ is a subquotient of
  $\br(\fgder)^*$, it follows from the strong form of property (C3)
  proved above that any such homomorphism must be trivial.
\end{proof}
  
\begin{rmk} \label{r:roots}%
  Whether or not the $p^{th}$ roots of unity are contained in
  $F(\br(\fgder))$ only depends on $\br_M$ since the kernel of the map
  $\pi: P(k) \to M(k)$ is a $p$-group.
\end{rmk}

 \begin{lemma} \label{l:abelian} Let $G$ be a semisimple group, $P$ a
  parabolic subgroup of $G$, $U$ the unipotent radical of $P$ and $M$
  a Levi subgroup. If $p \gg_G 0$, there exists a subgroup $A$ of $U$
  which is a product of root subgroups, $A$ is normalised by $M$, and
  $\mf{z}(M) : = \Lie(Z(M))$ acts faithfully on $\mf{a} := \Lie(A)$.
\end{lemma}

\begin{proof}
  Since $p \gg_G 0$, it suffices to find an abelian subalgebra
  $\mf{a}$ of $\mf{u} := \Lie(U)$ which is an $M$-submodule and on
  which $\mf{z}(M)$ acts faithfully since we can then get $A$ by
  exponentiating $\mf{a}$. We claim that if $\mf{a}$ is any maximal abelian
  subalgebra of $\mf{u}$ which is an $M$-submodule then $\mf{z}(M)$ acts
  faithfully on $\mf{a}$. Such a maximal $\mf{a}$ exists by the finite
  dimensionality of $\mf{u}$, so the claim implies the lemma.

  To prove the claim we consider a maximal $\mf{a}$ and assume that
  the kernel $\mf{k}$ of the map $\mf{z}(M) \to \End(\mf{a})$ is
  nonzero. The action of $\mf{k}$ on $\mf{u}$ is diagonalisable, and
  it acts trivially on $\mf{a}$. Let $R$ be the set of $k$-linear
  functions $\mf{k} \to k$ that occur in a decomposition of $\mf{u}$
  as a sum of one dimensional $\mf{k}$-invariant subspaces. The set
  $R$ does not consist of the singleton $\{0\}$ since the centre of
  $\mf{p}$ is trivial. Let $f \in R$ be such that $2f \notin R$. Such
  an $f$ always exists: the cardinality of $R$ is bounded above by a
  constant depending only on the root system of $G$ while the order of
  $2 \in k^{\times}$ goes to infinity with $p$. Let $\mf{u}_f$ be the
  subspace of $\mf{u}$ on which $\mf{k}$ acts by $f$; it is nonzero
  since $f \in R$ and it is an $M$-submodule of $\mf{u}$ since
  $\mf{k} \subset \mf{z}(M)$.

  The Lie bracket on $\mf{u}$ induces a map
  $\beta:\mf{a} \otimes \mf{u}_f \to \mf{u}$ of $\mf{m}$-modules. The
  image of $\beta$ is contained in $\mf{u}_f$ as
  $\mf{k} \subset \mf{m}$, it acts trivially on $\mf{a}$ and by
  multiplication by $f$ on $\mf{u}_f$. Since the adjoint action of
  $\mf{u}$ on itself is nilpotent, it follows that there is a nonzero
  $\mf{m}$-submodule $\mf{a}' \subset \mf{u}_f$ such that
  $\beta|_{\mf{a} \times \mf{a}'}$ is identically $0$, i.e., $\mf{a}$
  and $\mf{a}'$ commute. Since $\mf{k}$ acts on $\mf{u}_f$ by $f$, it
  acts on $[\mf{u}_f, \mf{u}_f]$ by $2f$, so by the choice of $f$ it
  follows that $[\mf{u}_f, \mf{u}_f]$, hence also
  $[\mf{a}', \mf{a}']$, is $\{0\}$. Since $\mf{a}$ is an abelian Lie
  algebra by assumption, we conclude that $\mf{a} \oplus \mf{a}'$ is
  an abelian subalgebra of $\mf{u}$ which is normalised by $M$,
  contradicting the maximality of $\mf{a}$.
  
\end{proof}

\begin{rmk} \label{r:shahidi} For many $G$ and $P$ one can give
  explicit examples of $A$ as in Lemma \ref{l:abelian}. For a maximal
  parabolic $P$, we say that $P$ is \emph{$L$-faithful} if the Lie
  algebra of any Levi subgroup of $P$ acts faithfully on
  $\mf{z}(U) := \Lie(Z(U))$. For groups of type $A_n$, one easily sees
  that any maximal parabolic is $L$-faithful.  More generally, from
  \cite[Appendix A]{shahidi:eisseriesbook} one sees that for any
  simple $G$ (and $p \gg_G 0$), there always exists at least one
  $L$-faithful parabolic. If $P'$ is any parabolic contained in an
  $L$-faithful parabolic $P$ then it follows that we may take $A$ to be
  $\mf{z}(U)$ (where $U$ is the unipotent radical of $P$). This gives
  another proof of Lemma \ref{l:abelian} for such $P'$; in particular,
  this applies to the Borel subgroup $B$.
\end{rmk}

\subsection{The case $G = \mr{GL}_n$ or $\mc{G}_n$} \label{s:gln}
\subsubsection{$\mr{GL}_n$}
Let $P$ be a standard parabolic of $\mr{GL}_n$ corresponding to a
partition $n = n_1 + n_2 + \dots + n_r$, with $r>1$ and let $M$ be the
Levi subgroup of $P$ consisting of block diagonal matrices, so
$M = \mr{GL}_{n_1} \times \mr{GL}_{n_2} \times \dots \times
\mr{GL}_{n_r}$. Given any representation $\br: \Gamma_F \to \mr{GL}_n$
with $\br(\gal{F}) \subset P(k)$, by projecting to $M(k)$ we obtain
representations $\br_i: \gal{F} \to \mr{GL}_{n_i}$, $i=1,2,\dots,r$.
  
\begin{lemma} \label{l:gln} Let the notation be as above and suppose
  that $p \gg_n 0$ and the following holds:
  \begin{enumerate}
  \item $\br_{i}$ is absolutely irreducible for $i=1,2\dots,r$.
  \item $\br_i$ is not isomorphic to $\br_j \otimes \bar{\psi}$ for
    $\bar{\psi} \in \{1, \kbar, \kbar^{-1}\}$ and all
    $i \neq j \in \{1,2,\dots,r\}$. (In particular, this condition
    always holds if $n_i \neq n_j$.)
  \item Let $\bar{\chi}_{i,j}$ be the character
    $\det(\br_i)^{n_j} \cdot \det(\br_j)^{-n_i}$.
    Then $[F(\zeta_p):F(\zeta_p) \cap F(\bar{\chi}_{i,j})]$ is greater
    than a constant $b$ depending only on $n$, for all
    $i \neq j \in \{1,2,\dots,r\}$.
\end{enumerate}
Then $H^0(\gal{F}, \br(\fgder)) \subset \mf{z}(M) \cap \fgder$ and
(C2), (C3) and (C4) hold.
\end{lemma}

\begin{proof}

 We first note that the assumptions of the lemma continue
  to hold if we replace $\br$ by the representation
  $\br' := \br_1 \oplus \br_2 \oplus \dots \oplus \br_r$.  Moreover,
  $H^0(\gal{F}, \br(\fgder)) \subset H^0(\gal{F}, \br'(\fgder))$ so
  one easily sees that if the conclusions of the lemma, except (C4),
  hold for $\br'$ then they also hold for $\br$.  

  The adjoint representation of $\mr{GL}_n$
  and $\br'$ make $\mf{gl}_n$ into a $\gal{F}$-module, and as such it
  is isomorphic to $\oplus_{i, j=1}^r \br_i \otimes \br_j^{\vee}$.
  From this it is clear that (1) and (2) imply that
  $H^0(\gal{F}, \br'(\fgder))$ is contained in $\mf{z}(M) \cap \fgder$,
  and (C2) will hold if the $p^{th}$ roots of unity are not contained
  in $F(\br(\fgder))$.

  We now show that conditions (1) and (3) imply that the $p^{th}$
  roots of unity are not contained in $F(\br(\fgder))$, and then we
  may apply Lemma \ref{l:allP} to conclude that (C3) and (C4) also
  hold.  Since the kernel of the projection from $P(k)$ to $M(k)$ is a
  $p$-group, the condition on the $p^{th}$ roots of unity
  will hold for $\br$ iff it holds for $\br'$, so to check this
  condition we may assume that $\br = \br'$. From (1)
  and \cite[Lemma A.6]{fkp:reldef} it follows that the order of the
  maximal cyclic quotient of $\br_i(\gal{F})$ differs from the order
  of $\det(\br_i(\gal{F}))$ 
  by at most a constant depending only on
  $n_i$. This implies that if we set $F'$ to be extension of $F$ cut
  out by the image of $\gal{F}$ in
  $M(k)/ k^{\times}\cdot M^{\mr{der}}(k)$, where $k^{\times}$ is
  embedded in $M(k)$ as the centre of $\mr{GL}_n(k)$, then it suffices
  that $[F(\zeta_p): F' \cap F(\zeta_p)]$ should be larger than a
  constant depending only on $n$. The map
  $\gal{F} \to M(k) \to (M(k)/M^{\mr{der}}(k)) \cong (k^{\times})^r$
  is simply the map given by the product of all the $\det(\br_i)$, and
  the central $k^{\times}$ maps to $(k^{\times})^r$ by the map
  $x \mapsto (x^{n_1}, x^{n_2},\dots,x^{n_r})$. Now since
  $k^{\times}$ is a cyclic group, the kernel of the map
  $(k^{\times})^r$ to $(k^{\times})^{r(r-1)}$ given by
  \[
    (x_1,x_2,\dots,x_r) \mapsto (x_i^{n_j} x_j^{n_i})_{i \neq j}
  \]
  contains the image of the central $k^{\times}$ as above with index
  bounded in terms of $n$. Since $F(\zeta_p)$ is a
  cyclic extension of $F$ (and $r$ is bounded by $n$), it follows that
  there exists a constant $b$ for which (3) implies that 
  $[F(\zeta_p): F' \cap F(\zeta_p)]$ is larger than any fixed constant
  and therefore the $p^{th}$ roots of unity are not contained in
  $F(\br(\fgder))$.
 
\end{proof}

\subsubsection{$\mc{G}_n$}
In this subsection we take the group $G$ to be the group $\mc{G}_n$ as
defined in \cite[\S 2.1]{clozel-harris-taylor} and consider a
representation $\br: \Gamma_F \to G(k)$. The group $G^0$ is isomorphic
to $\mr{GL}_n \times \mbb{G}_m$ and $G/G^0$ is isomorphic to
$\Z/2\Z$. The group $G^{\mr{der}}$ is the $\mr{GL}_n$ factor of $G^0$
and we denote its Lie algebra by $\mf{g}_n$. We let $L$ be the field
cut out by composing $\br$ with the projection to $G/G^0$ and we
assume that it is a quadratic extension of $F$.  We let
$\bar{r}: \Gamma_{L} \to \mr{GL}_n(k)$ be the representation obtained
from $\br|_{\Gamma_L}$ by projecting to the $\mr{GL}_n$ factor of
$G^0$. We assume (for simplicity) that
$\bar{r} \cong \oplus_{i=1}^s \bar{r}_i$, where each $\bar{r}_i$ is an
irreducible representation of $\Gamma_L$ of dimension $n_i$.

\begin{lemma} \label{l:gn} Let the notation be as above and
  $p \gg_n 0$. Suppose that $L$ is a CM extension of $F$ (so the
  character $\Gamma_F \to \Z/2\Z$ induced by $\br$ is totally odd) and
  the following holds:
  \begin{enumerate}
  \item $\bar{r}_{i}$ is absolutely irreducible for $i=1,2\dots,r$.
  \item $\bar{r}_i$ is not isomorphic to $\bar{r}_j \otimes \bar{\psi}$ for
    $\bar{\psi} \in \{1, \kbar, \kbar^{-1}\}$ and all
    $i \neq j \in \{1,2,\dots,s\}$. (In particular, this condition
    always holds if $n_i \neq n_j$.)
  \item Let $\bar{\chi}_{i,j}$ be the character
    $\det(\bar{r}_i)^{n_j} \cdot \det(\bar{r}_j)^{-n_i}$.
    Then $[F(\zeta_p):F(\zeta_p) \cap F(\bar{\chi}_{i,j})]$ is greater
    than a constant $b$ depending only on $n$, for all
    $i \neq j \in \{1,2,\dots,s\}$.
\end{enumerate}
Then (C1) (C2), (C3) and (C4) hold.
\end{lemma}

\begin{proof}
  The proof is very similar to that of Lemma \ref{l:gln} so we only
  indicate the modifications that need to be made.

  For (C1), the proof of Lemma \ref{l:gln} shows that
  $(\mf{g}_n)^{\Gamma_L}$ is contained in the diagonal matrices. The
  action of complex conjugation on $\mf{g}_n$ is given by $x \mapsto -
  {}^tx$, so since $p \neq 2$ it follows that (C1) holds.

  The rest of the conditions follow in the same way as in Lemma
  \ref{l:gln}; the condition (C4) follows from the semisimplicity
  assumption on $\bar{r}$ using \cite{guralnick:CR}
  and (3) is not needed for that.

\end{proof}

\subsection{$P$ is a maximal parabolic} \label{s:maxpar} We now give
concrete conditions under which properties (C1) - (C4) hold in the
case that $P$ is a maximal parabolic of a classical $G$ with abelian
unipotent radical.
 
\begin{lemma} \label{lem:h^0} %
  Suppose $[F(\zeta_p):F] \gg_G 0$, $P$ is a maximal parabolic, and $\br_M$ is
  absolutely irreducible. Consider the following cases:
 \begin{enumerate}
 \item $G = \mr{GL}_n$ and $P$ is the parabolic corresponding to the
   partition $n = n' +n''$. We may then view $\br_M$ as a direct sum
   of two representations $\br_M'$ and $\br_M''$ of $\Gamma_F$ of
   dimension $n'$ and $n''$, and $\br$ is an extension of $\br_M''$ by
   $\br_M'$.  Consider the conditions:
   \begin{enumerate}[(a)]
   \item $n$ is odd, or $n' = n''$ and $\br_M' \ncong \br_M''$;
   \item $n$ is odd, or $n' = n''$ and
     $\br_M' \ncong \br_M'' \otimes \kbar^{-1}$.
   \end{enumerate}
 \item $G = \mr{GSp}_{2n}$ and $P$ is the Siegel parabolic, i.e., the
    stabilizer of a Lagrangian subspace: in this case
    $M \cong \mr{GL}_n \times \mathbb{G}_m$ and we let $\br_M'$ be the
    representation of $\Gamma_F$ corresponding to the first
    projection. Consider the conditions:
    \begin{enumerate}[(a)]
    \item $\br_M'$ does not preserve any nondegenerate symmetric bilinear form;
    \item $\br_M'$ does not preserve any nondegenerate
      symmetric bilinear form up to a scalar with multiplier
      character $\kbar$.
    \end{enumerate}
  \item $G = \mr{SO}_n$ and $P$ is the stabilizer of an isotropic
    line in the standard representation of $G$: in this case
    $M \cong \mr{SO}_{n-2} \times \mathbb{G}_m$ and we let $\br_M'$ be
    the $(n-2)$-dimensional linear representation induced from $\br_M$ by the
    tensor product of the standard representations of the two
    factors. Consider the conditions
    \begin{enumerate}[(a)]
    \item  $\br_M'$ does not have an invariant vector;
    \item $\br_M'$ does not have a $\kbar^{\pm 1}$-invariant
      vector.
    \end{enumerate}
  \item $G = \mr{SO}_{2n}$ and $P$ is the stabilizer of a maximal
    isotropic subspace: in this case $M \cong \mr{GL}_n$ so we may
    view $\br_M$ as an $n$-dimensional linear representation. Consider
    the conditions
    \begin{enumerate}[(a)]
    \item  $\br_M$ does not preserve any nondegenerate
      skew-symmetric form;
    \item $\br_M$ does not preserve a nondegenerate skew-symmetric
      form up to a scalar with multiplier character $\kbar$.
    \end{enumerate}
    If $(G,P)$ is as in one of the above cases and the map
    $\br(\Gamma_F) \to M(k)$ induced by $\pi$ is not injective then
    (C1) holds if (a) does and (C2) holds if (b) does.
 \end{enumerate}
\end{lemma}

\begin{proof}
  Consider the semisimplification $(\fgder)^{\mr{ss}}$ of $\fgder$ as
  a $P(k)$-module. The $P(k)$-action on $(\fgder)^{\mr{ss}}$ factors
  through an action of $M(k)$ and as an $M(k)$-module
  $(\fgder)^{\mr{ss}}$ has a direct sum decomposition
  $\mf{m}' \oplus \mf{u} \oplus \mf{u}^-$, where $\mf{m}'$ is the Lie
  algebra of a Levi subgroup $M'$ of the
  image of $P$ in $G^{\mr{ad}}$.  Each of the summands is semisimple
  by \cite[Theorem A]{guralnick:CR} since $[F(\zeta_p):F] \gg_G 0$
  implies $p \gg_G 0$. Furthermore, the Killing form induces a duality
  of $\mf{u}$ and $\mf{u}^-$ as $M$-representations.
 
  The above decomposition induces a decomposition of the
  semisimplification of $\br(\fgder)$ as a $\Gamma_F$-module, which we
  write as
  \[
    \br(\fgder)^{\mr{ss}} = \br(\mf{m}') \oplus \br(\mf{u})
    \oplus \br(\mf{u}^-);
  \]
  each of the three summands is semisimple by \cite[Theorem
  A]{guralnick:CR} because $p \gg 0$ and $\br_M$ is absolutely
  irreducible. All the above holds for any $G$ and any parabolic
  $P$. If $P$ is maximal, then the centre $Z(\mf{m}')$ of $\mf{m}'$ is
  one dimensional and is equal to $H^0(\Gamma_F, \br(\mf{m}'))$ by
  \cite[Lemma A.2]{fkp:reldef}.

  Now suppose $G$ is one of the groups in items (1) - (4). In each of
  these cases, one easily works out the action of $M$ on $ \mf{u}$
  and this is given as follows:
  \begin{enumerate}[(i)]
  \item  $G = \mr{GL}_n$: in this case $M = \mr{GL}_{n_1} \times
    \mr{GL}_{n_2}$ and its representation on $\mf{u}$ is
    $\lambda_1 \otimes \lambda_2^{\vee}$ where $\lambda_i$ is the
    standard representation of $\mr{GL}_{n_i}$.
  \item $G= \mr{GSp}_{2n}$: in this case the representation of $M$ on
    $\mf{u}$ is given by the projection to the $\mr{GL}_n$ factor
    and the second symmetric power of the standard representation.
  \item $G = \mr{SO}_n$ and $P$ as in (3): in this case the
    representation of $M$ on $\mf{u}$ is the tensor product of the
    standard representations of the two factors.
  \item $G = \mr{SO}_{2n}$ and $P$ as in (4): in this case the
    representation of $M$ on $\mf{u}$ is the second exterior power
    of the standard representation of $\mr{GL}_n$.
  \end{enumerate}

  From the list above it is clear that
  $H^0(\Gamma_F, \br(\mf{u})) = 0$ in these cases if $\br$
  satisifies condition (a).
  By semisimplicity the same holds for $\mf{u}^-$ since it is dual to
  $\mf{u}$. We thus see that
  $H^0(\Gamma_F, \br(\fgder)^{\mr{ss}}) = Z(\mf{m}')$.  Suppose
  $x \in H^0(\Gamma_F, \br(\fgder))$ is a nonzero element. By the
  foregoing, $x = z + u$, where $ 0 \neq z \in Z(\mf{m}')$ and
  $u \in \mf{u}$. Let $\gamma$ in $\br(\Gamma_F)$ be a nontrivial
  element in $\ker(\pi)$. Since $U$ is commutative, it follows that
  the image of $\gamma$ in $G^{\mr{ad}}(k)$ is in the centralizer of
  $z$ in $G^{\mr{ad}}(k)$.  By \cite[Theorem 3.14]{steinberg:torsion}
  the centralizer in $G^{\mr{ad}}$ of $z$ (in fact any semisimple
  element) is a connected reductive subgroup, so by a standard Lie
  algebra computation (recall $p \gg_G 0$) one sees that the centralizer
  of $z$ in $G^{\mr{ad}}$ is in fact $M'$.  Since the image of
  $\gamma$ in $G^{\mr{ad}}(k)$ is not in $M'(k)$, it follows that we
  must have $H^0(\Gamma_F, \br(\fgder)) = 0$, i.e., (C1) holds.

  For (C2) we first note that it follows from \cite[Corollary
  A.6]{fkp:reldef} that any one dimensional $k[\Gamma_F]$-quotient of
  $\br(\mf{m}')$ must be a character of order bounded by a constant
  depending only on $M$. Given this, the assumption that 
  $[F(\zeta_p):F] \gg_G 0$ implies that
  $H^0(\gal{F}, \br(\mf{m'})(1)) = 0$. As in case (C1), the assumption
  (b) implies that $H^0(\gal{F}, \br(\mf{u})(1)) = 0$. In case (3)
  assumption (b) also implies that
  $H^0(\gal{F}, \br(\mf{u}^-)(1)) = 0$. In the other cases an element
  $u$ of $H^0(\gal{F}, \br(\mf{u}^-)(1))$ may be viewed as an
  invariant in the tensor product of two absolutely irreducible
  representations, so if it is nonzero it must be nondegenerate, i.e.,
  induces an isomorphism of one of the representations with the dual
  of the other one (so both representations must have the same
  dimension). Let $\gamma \in \Gamma_F$ be any element such that
  $\kbar(\gamma) = 1$ and $\br(\gamma)$ is a nontrivial element of
  $U(k)$.  A simple matrix computation shows that in each case
  $\gamma(u) - u$ is a nonzero element of $\mf{m}'(1)$.  On the other
  hand, for any element $u'$ in $\mf{m}'(1) \oplus \mf{u}(1)$,
  $\gamma(u') - u'$ is an element of $\mf{u}(1)$. This implies that
  given a nonzero $u$, there exists no element $u'$ as above such that
  $\gamma(u + u') = u + u'$ so $u$ does not lift to a
  $\Gamma_F$-invariant element, as required.
\end{proof}

\begin{rmk} \label{shahidi}
  One may derive similar, but more complicated, conditions for any
  reductive $G$ and maximal parabolic $P$ using the tables in
  Appendices A and B of \cite{shahidi:eisseriesbook} which give the
  structure of $\mf{u}$ as an $M$-module.
\end{rmk}

For $P$ a parabolic subgroup of $G$, let $\ov{M}$ be the Levi quotient
of the image of $P$ in $G^{\mr{ad}}$, the adjoint group of $G$. If $P$ is
maximal, $\ov{M}$ modulo its derived subgroup is a one dimensional
torus which we identify with $\mathbb{G}_m$. This quotient gives rise
to a character
$\chi: \Gamma_F \to \mathbb{G}_m(k)$ obtained as
the composite
\[
  \Gamma_F \stackrel{\br}{\to} P(k) \to \ov{M}(k) \to
  \mbb{G}_m(k) \ .
\]

\begin{defn} \label{dstN} We say that $\br$ as above satisfies
  property $D(s,t,N)$, where $s,t,N$ are positive integers, if for all
  $\sigma \in \Aut(k)$,
  $(\chi^{m_1})^{\sigma}|_{\Gamma_{F'}} \neq (\chi^{m_2} \otimes
  \kbar^s)|_{\Gamma_F'}$, with $[F':F] \leq N$ and all
  $m_1,m_2 \in \Z$ such that $0 \leq |m_1|,|m_2| \leq t$.
\end{defn}
 
\begin{lemma} \label{lem:h^1} Suppose $p \gg_G 0$, $P$ is a maximal
  parabolic, and $\br_M$ is absolutely irreducible.  For any pair
  $(G,P)$ as above there exist positive integers $b$, $s$, $t$ and $N$
  (depending only on the root datum of $(G,P)$) such that if
  \begin{enumerate}
  \item $[F(\kbar): F(\kbar) \cap F(\chi)] > b$, or
  \item $\br$ satisfies property $D(s,t,N)$,
  \end{enumerate}
  then both (C3) and (C4) hold for $\br$. Furthermore, in Case (1), $\mu_p$ is not contained in $F(\br(\fgder))$.
\end{lemma}

\begin{proof}
  We will prove that in both cases (1) and (2), the stronger versions
  of properties (C3) and (C4) with $\br(\fgder)$ and $\br(\fgder)^*$
  replaced by their semisimplifications hold.

  In case (1), we will reduce to Lemma \ref{l:allP}.  By assumption
  $\br_M$ is absolutely irreducible, so it suffices to show that there
  exists $b$ such that if (1) holds then the $p^{th}$ roots of unity are
  not contained in $F(\br(\fgder))$. Using Remark \ref{r:roots}, this
  follows by applying  \cite[Lemma A.6]{fkp:reldef} to the semisimple
  group $M^{\mr{der}}$ and the definition of the
  character $\chi$.
  
  We now prove the existence of $s,t,N$ such if $\br$ satisfies
  $D(s,t,N)$ then (C3) and (C4) hold.
  Let $L$ be the extension of $F$ cut out by the image of $\Gamma_F$
  in $(\ov{M})^{\mr{ad}}(k)$, where $(\ov{M})^{\mr{ad}}$ is the
  adjoint group of $\ov{M}$. To verify property (C3) it suffices to do
  so after replacing $\Gamma_F$ by $\Gamma_{L}$.  The action of
  $\Gamma_F$ on $(\fgder)^{\mr{ss}}$ factors through $\ov{M}(k)$, so
  the action of $\Gamma_L$ factors through the action of
  $Z(\ov{M})(k)$, the centre of $\ov{M}(k)$. Consider the
  decomposition
  $(\fgder)^{\mr{ss}} = \mf{m}' \oplus \mf{u} \oplus \mf{u}^-$ as in
  the proof of Lemma \ref{lem:h^0}.  The group $Z(\ov{M})(k)$ acts
  trivially on $\mf{m}'$, by a sum of non-trivial characters, say
  $\chi_1, \chi_2,\dots, \chi_r$, on $\mf{u}$ (which may be
  determined explicitly), and by the inverses of these characters on
  $\mf{u}^-$.

  By the irreducibility of $\br_M$ and \cite[Corollary
  A.6]{fkp:reldef}, it follows that there exists an integer $N$
  depending only on $G$ such that the degree $[L \cap F^{\mr{ab}}:F]$ is 
  bounded above by $N$. Let $s$ be the order of
  the kernel of the map $Z(\ov{M}) \to \ov{M}/\ov{M}^{\mr{der}}$. In
  order for property (C3) to hold, it suffices that the set of
  characters of $\Gamma_L$ induced by the set
  $\{(\chi_1^{\sigma})^{\pm}, (\chi_2^{\sigma})^{\pm}, \dots,
  (\chi_r^{\sigma})^{\pm}\}$ is disjoint from the set obtained by
  tensoring those induced by
  $\{\chi_1^{\pm}, \chi_2^{\pm},\dots, \chi_r^{\pm}\}$ with $\kbar$,
  for all $\sigma \in \Aut(k)$. This in turn follows if we have
  disjointness when we raise each character of $\Gamma_L$ above to its
  $s^{th}$ power. Since each $\chi_i^s$ is a power of the character
  $\chi$ and the exponents that occur only depend on $G$, it follows
  that for a suitable choice of $t$, if $D(s,t,N)$ holds then (C3)
  holds.

  Let $K' = F(\br(\fgder), \kbar)$. It is a subfield of $K$ and
  $p \nmid [K:K']$, so to prove that (C4) holds it suffices to show
  that $H^1(\Gal(K'/F), \br(\fgder)^*) = 0$.  Let $L'$ be the
  extension of $F$ obtained by adjoining $\mu_p$ to the extension cut
  out by $\br_M$. Since $\br_M$ is absolutely irreducible and
  $p \gg_G 0$, it follows from \cite[Theorem A]{guralnick:CR} as in
  \cite[Lemma A.1]{fkp:reldef} that
  $H^1(\Gal(L'/F), (\br(\fgder)^*)^{\mr{ss}}) = 0$. By the
  inflation-restriction sequence it then follows that
  $H^1(\Gal(K'/F), (\br(\fgder)^*)^{\mr{ss}}) = 0$ if
  $\Hom_{\Fp[\Gal(L'/F)]}(\Gal(K'/L'), (\br(\fgder)^*)^{\mr{ss}}) = 0$
  and this holds if
  $\Hom_{\Fp[\Gal(L'/L)]}(\Gal(K'/L'), (\br(\fgder)^*)^{\mr{ss}}) =
  0$, with $L$ as above.

  There is a canonical $\Fp[\Gal(L'/F)]$-equivariant injection from
  $\Gal(K'/L')$ into $U(k)$. The possibly nonabelian group $U(k)$ has
  an invariant filtration with abelian associated graded such that
  $\Gal(L'/L)$ acts via the characters
  $\{\chi_1, \chi_2,\dots, \chi_r\}$ as before. Using this, if $\br$
  satisifies $D(s,t,N)$ then
  $\Hom_{\Fp[\Gal(L'/L)]}(\Gal(K'/L'), (\br(\fgder)^*)^{\mr{ss}})$
  indeed vanishes since the abelianisation of $\Gal(K'/L')$ and
  $(\br(\fgder)^*)^{\mr{ss}}$ have no common
  ${\Fp[\Gal(L'/L)]}$-subquotients (if $t \gg_G 0$).
\end{proof}

\begin{rmk}
  If $\chi$ takes values in $\Fp^{\times}$ then we may ignore $\sigma$
  in condition $D(s,t,N)$. Since $s$, $t$ and $N$ are all independent
  of $p$, there exists a constant $n_G$ such that the condition in (2)
  is satisfied when $\chi = \kbar^n$, for all $n$ (modulo 
  $p-1$) except
  for those in a subset of $\mbb{Z}/(p-1)\mbb{Z}$ of size at most
  $n_G$. In particular, condition (2) can hold even when
  $F(\chi) = F(\kbar)$.
\end{rmk} 

\begin{eg} \label{exgsp4a} %
  Let $G = \mr{GSp_4}$ and let $P$ be the Siegel parabolic of $G$, so
  the Levi quotient $M$ of $P$ is $\mr{GL}_2 \times \mbb{G}_m$, the
  group $\ov{M}$ is $\mr{GL}_2/\{\pm 1\}$ and $\ov{M}^{\mr{ad}}$ is
  $\mr{PGL}_2$. From Dickson's classification of finite subgroups of
  $\mr{PGL_2}(k)$ it follows that the integer $N$ can be taken to be
  $4$, and even $2$ in all cases except the $A_4$ case (where it is
  $3$) and the even dihedral case.
  
  The group $Z(\ov{M})$ acts on $\mf{u}^{\pm}$ by $\mr{det}^{\pm 1}$
  and the integer $s$ occurring in the proof of Lemma \ref{lem:h^1} is
  $1$. If we write $\br_M$ as $(\br_1, \delta)$, the character
  $\chi: \Gamma_F \to \mbb{G}_m(k)$ is $\mr{det} \circ \br_1$. If
  $\chi$ takes values in $\Fp^{\times}$ then using (the proof of)
  Lemma \ref{lem:h^1} we see that the properties (C3) and (C4) hold
  for $\br$ as long as the characters $\chi^{a}\otimes \kbar$
  have order at least $5$ for $|a| \leq 2$.
\end{eg}

\begin{eg}
  Let $G = \mr{GL}_{2n}$ and let $P$ be the parabolic corresponding to
  the partition $2n = n + n$. In this case there is only one character
  of $Z(\ov{M}(k))$ occuring in $\mf{u}$, a generator of the
  character group, and $s = n$. Moreover, $(\ov{M})^{\mr{ad}}$ is
  $\mr{PGL}_n \times \mr{PGL}_n$, so we may take $N$ to be the maximal
  size of the abelianisation of an absolutely irreducible subgroup of
  $\mr{PGL}_n(k) \times \mr{PGL}_n(k)$.
  Therefore (C3) and (C4) hold if for all
  $\sigma \in \Aut(k)$,
  $(\chi^{\pm 1})^{\sigma}|_{\Gamma_{F'}} \neq (\chi^{\pm 1} \otimes
  \kbar^n)|_{\Gamma_F'}$, with $F'$ any finite extension of $F$ of
  degree $\leq N$.
\end{eg}

\subsection{$P$ is a Borel subgroup} \label{s:borel}

Assume that $P=B$ is a Borel subgroup of $G$. In this case
$\br(\fgder)$ has a filtration with one dimensional subquotients, so
it is not difficult to analyze the conditions under which properties
(C1) - (C4) hold for $\br$. We give examples of such conditions
without attempting to be exhaustive. We note here that in the preprint
\cite{ray:reducible}, A.~Ray has proved a lifting theorem for
representations with values in the Borel subgroup of
$\mr{GSp}_{2g}(k)$; our methods give stronger results even in this
case.

In all that follows we shall assume that $p \gg_G 0$.  Under this
assumption, by Lemma \ref{l:exp}, any element $x \in U(k)$ has a
logarithm $\log(x) \in \mf{u}$. Let
$U(\br) = U(k) \cap \im(\br)$, the Sylow $p$-subgroup of
$\im(\br)$. We let $\mf{u}(\br)$ be the $k$-subspace of $\mf{u}$
spanned by $\{ \log(x)\, |\, x \in U(\br)\}$.  This is a
$k[\Gamma_F]$-submodule of $\br(\fgder)$. Then $Z(\mf{u}(\br))$, the
centralizer in $\fgder$ of $\mf{u}(\br)$, is also a
$k[\Gamma_F]$-submodule of $\br(\fgder)$.

Let $T$ be the Levi quotient of $B$ and let $\br_T$ be the induced
representation $\Gamma_F \to T(k)$. The $\Gamma_F$-action on
$(\fgder)^{\mr{ss}}$ (hence also $Z(\mf{u}(\br))^{\mr{ss}}$),
factors through $\br_T$, so it is a sum of characters.

\begin{lemma} \label{l:borel}%
  Let $\br:\Gamma_F \to G(k)$ be a continuous representation with
  $\im(\br)\subset B(k)$. Consider the following conditions:
\begin{enumerate}[(1)]
\item The trivial character does not appear in
  $Z(\mf{u}(\br))^{\mr{ss}}$.
\item The character $\kbar^{-1}$ does not appear in
  $Z(\mf{u}(\br))^{\mr{ss}}$.
\item For all characters $\chi$ and $\chi'$ occurring in the action of
  $\;\Gamma_F$ on $(\fgder)^{\mr{ss}}$ we always have
  $\sigma(\chi') \ncong \chi \otimes \kbar$, for any automorphism
  $\sigma$ of $k$.
\end{enumerate}
Then (1) implies (C1) holds, (2) implies (C2) holds, and (3)
implies that (C3) and (C4) hold.
\end{lemma}
Note that (3) only depends on $\br_T$.
\begin{proof}
  We will use the formula
  \begin{equation} \label{e:ad} \tag{$*$}
    \Ad(\exp(X))(Y) = \sum_{i=0}^{\infty} \frac{\ad(X)^i(Y)}{i!} ,
  \end{equation}
  which holds for all $X \in \mf{u}$ and all $Y \in \fg$; the sum
  is actually finite because $X$ is nilpotent.  This is well-known
  over fields of characteristic zero and so also holds if $p \gg_G 0$
  by spreading out.

  If there exists $Y \in (\fgder)^{\Gamma_F}$ such that
  $Y \notin Z(\mf{u}(\br))$, then there exists $x \in U(\br)$ such
  that $[\log(x),Y] \neq 0$. Formula \eqref{e:ad} and the fact that
  $\log(x)$ is nilpotent, imply that $\Ad(x)(Y) \neq Y$, a
  contradiction. The condition (1) clearly implies that
  $ Z(\mf{u}(\br))^{\Gamma_F} = 0$, so (C1) holds.

  The group $U(\br)$ is also the Sylow $p$-subgroup of $\Gal(K/F)$,
  and as such its action on $\br(\fgder)$ and $\br(\fgder)^*$ is
  identical. Therefore, it follows from the above that the $\Gamma_F$-invariants 
  of $\br(\fgder)^*$ must be contained in
  $ Z(\mf{u}(\br))\otimes \kbar$. This shows that (2) implies that
  (C2) holds.

   It is clear that (3) implies that a stronger form of (C3), with
  $\br(\fgder)$ replaced by $\br(\fgder)^{\mr{ss}}$.

  We now show that (3) implies that
  $H^1(\Gal(K/F), \chi \otimes \kbar)=0$, for $\chi$ any character
  occurring in $\br(\fgder)^{\mr{ss}}$, which clearly implies that
  (C4) holds. Since $U(\br)$ is the Sylow $p$-subgroup of $\Gal(K/F)$,
  it follows from the inflation-restriction sequence that it suffices
  to show that $\Hom(U(\br), \chi \otimes \kbar)^{\Gal(K/F)} = \{0\}$.
  By Lemma \ref{l:subq}, this will follow if we can show that
  $\Hom(\mf{u}, \chi \otimes \kbar)^{\Gal(K/F)} = \{0\}$. But this is
  an immediate consequence of (3), so (C4) holds.
\end{proof}


\section{} \label{ordappendix}

In this appendix we finish the proof of Proposition \ref{CS} by
providing the promised lemma about generic fibers of ordinary
deformation rings. We can work in somewhat greater generality than
needed in Proposition \ref{CS}, but in essence we follow the proof in
the $\mr{GL}_n$ case due to Geraghty (\cite{geraghty:ordinary}). We
repeat some of our notation here, to make this appendix relatively
self-contained. Let $\mc{O}$ be the ring of integers, with residue
field $k$, in a finite extension $E$ of $\Q_p$. We continue to let $F$
denote a number field, for consistency with the body of this paper,
but in this appendix we will only use the completion $F_v$ of $F$ at
some prime $v \vert p$ and its absolute Galois group $\gal{F_v}$. Let
$G$ be a split connected reductive group over $\mc{O}$ (we emphasize
that in contrast to the bulk of the paper, here $G$ is connected),
equipped with a split torus $T_0$. We will repeatedly use the
following fact:
\begin{prop}[See Corollary 5.2.13 of \cite{conrad:luminy}]\label{borelconj}
For any $\mc{O}$-algebra $A$, any two Borel subgroups $B_1, B_2$ of the base-change $G_A$ are $G$-conjugate Zariski-locally on $\Spec(A)$. Moreover, if $B_1$ and $B_2$ both contain $T_0$, they are $N_G(T_0)$-conjugate Zariski-locally on $\Spec(A)$.
\end{prop}
Let $T_G$ denote the ``canonical torus" of $G$: for any Borel subgroup $B \subset G$, $T_G$ is equal to the quotient $B/R_u(B)$ by the unipotent radical of $B$. The sense in which this is canonical is that any two Borel subgroups $B'$ and $B$ are $G(\mc{O})$-conjugate since $\mc{O}$ is a local ring, and $N_G(B)=B$ (this follows from the analogue over fields and smoothness of the normalizer, and it is valid for Borel subgroups in a reductive group scheme over an arbitrary base: \cite[Corollary 5.2.8]{conrad:luminy}) implies that a choice of $G(\mc{O})$-conjugation $B' \xrightarrow{\sim} B$ induces an isomorphism $B'/R_u(B') \xrightarrow{\sim} B/R_u(B)$ that is independent of the choice, since $B$ acts trivially by conjugation on $B/R_u(B)$. We similarly obtain a canonical identification of the character groups $\Hom(B, \mathbb{G}_m)$, which allows us to define unambiguously dominant regular cocharacters of $T_G$.

To define the ordinary condition we will, roughly speaking, impose the
condition that a representation factors through a Borel subgroup and
has pre-specified projection to the canonical torus; a correct
treatment of the moduli problem will require working over general (not
necessarily local) bases, and even though the projection of a Borel
subgroup modulo its unipotent radical makes sense over an arbitrary
base, we unfortunately don't know how to formulate the condition on
the projection without localizing.  We fix a reference Borel
$B_0 \subset G$ (over $\mc{O}$) containing the split maximal torus
$T_0$, so for any $\mc{O}$-algebra $A$ we can identify
$T_G \times_{\Spec(\mc{O})} \Spec(A)$ with
$B_0/R_u(B_0) \times_{\Spec(\mc{O})} \Spec A$.

Fix a prime $v \vert p$ of $F$ (or more generally, fix a finite extension of $\Q_p$), and fix a collection $\lambda_ \tau \colon \mathbb{G}_m \to T_G$ of co-characters, indexed over $\tau \in \Hom_{\Q_p}(F_v, E)$; here we if necessary enlarge $E= \mr{Frac}(\mc{O})$ so that it contains all $\ov{E}$-embeddings of $F_v$. Let $\rec_v \colon F_v^\times \to \gal{F_v}^{\mr{ab}}$ denote the local Artin reciprocity map, normalized to take uniformizers to geometric Frobenii, and define a character $\chi_{\lambda} \colon I_{F_v} \to T_G(\mc{O})$ by
\[
\chi_{\lambda}(\sigma)= \prod_{\tau \colon F_v \to E} \lambda_{\tau}(\tau(\rec_v^{-1}(\sigma))).
\]

\begin{defn}\label{orddef}
Let $A$ be a finite local $E$-algebra and $\lambda=(\lambda_\tau)_{\tau \colon F_v \to E}$ as above. Fix a finite extension $F_v'/F_v$. Let $\rho \colon \gal{F_v} \to G(A)$ be a continuous representation. We say that $\rho$ is $F_v'$-ordinary of weight $\lambda$ if there is a Borel subgroup $B \subset G_A$ such that $\rho$ factors through $B(A)$, and, choosing $g \in G(A)$ such that $gBg^{-1}= B_0$ the projection 
\[
\gal{F_v} \xrightarrow{\rho} B(A) \xrightarrow{x \mapsto gxg^{-1}} B_0(A) \to T_G(A)
\] 
is equal to $\chi_{\lambda}$ on the open subgroup $I_{F_v'}$ of $I_{F_v}$.
\end{defn}
Here we use Proposition \ref{borelconj} to choose such $g \in G(A)$, and note that the condition thus formulated is independent of the choice of $g$ by the remarks ($N_G(B_0)=B_0$) following the Proposition.

While we can formulate this assumption for any $\lambda$, in what follows we assume that $\lambda$ is dominant regular, i.e. that each $\lambda_{\tau}$ is a dominant regular cocharacter of $T_G$. The collection of cocharacters $\lambda$ gives a $p$-adic Hodge type $\mbf{v}_{\lambda}$. Fix a residual representation $\br \colon \gal{F_v} \to G(k)$, and let $R_{\br}^{\square}$ represent the lifting functor $\Lift_{\br}$ of \S \ref{prelims}. Recall from \cite[Theorem A]{bellovin-gee-G} that 
the quotient $R_{\br}^{\square, \mbf{v}_{\lambda}}$ of
$R_{\br}^{\square}$ whose points in finite local $E$-algebras
parametrize potentially semistable representations that become
semistable over $F'_v$ and have $p$-adic Hodge-type
$\mbf{v}_{\lambda}$ is equidimensional, with all irreducible
components of $R_{\br}^{\square, \mbf{v}_{\lambda}}[1/p]$ having
dimension $\dim(G)+[F_v: \Q_p]\dim(\mr{Fl}_G)$, and it admits an open
dense regular subscheme (note that we omit $F_v'$ from the notation
for $R_{\br}^{\square, \mbf{v}_{\lambda}}$). We wish to construct a
further quotient parametrizing ordinary representations and show that
it is a union of irreducible components of
$R_{\br}^{\square, \mbf{v}_{\lambda}}$.

Let $\rho^{\lambda} \colon \gal{F_v} \to G(R_{\br}^{\square,
  \mbf{v}_{\lambda}})$ denote the universal lift, and for any $R_{\br}^{\square, \mbf{v}_{\lambda}}$-algebra $A$, let $\rho^\lambda_A$ denote the push-forward to $A$ of $\rho^\lambda$. Let $\mc{G}$ be the subfunctor (on $R_{\br}^{\square, \mbf{v}_{\lambda}}$-algebras) of $\mr{Fl}_G$ such that $\mc{G}(A)$ is the subset of $B\in \mr{Fl}_G(A)$ fixed under the canonical action of $\rho^{\lambda}_A(\gal{F_v})$ on $\mr{Fl}_G(A)$. $\mc{G}$ is clearly representable by a closed subscheme, and it is clear that if $B \in \mc{G}(A)$, then $\rho^{\lambda}_A$ factors through $B(A) \subset G(A)$: indeed, $\rho^{\lambda}_A(\gal{F_v})$ is contained in $N_G(B)(A)=B(A)$. 

Let $\mc{G}_{\lambda} \subset \mc{G}$ be the subfunctor defined by taking $\mc{G}_{\lambda}(A)$ to be the set of $B \in \mc{G}(A)$ such that
\begin{itemize}
\item choosing (by Proposition \ref{borelconj}) a Zariski cover $\Spec(\wt{A}) \to \Spec(A)$ and a $g \in G(\wt{A})$ such that $gB_{\wt{A}}g^{-1}=B_0$, 
for all $\sigma \in I_{F_v'}$, the projection of $g\rho^\lambda_{\wt{A}}(\sigma)g^{-1}$ to $T_G(\wt{A})$ is equal to (the image in $\wt{A}$ of) $\chi_{\lambda}(\sigma)$.
\end{itemize}
We stress again that this condition is independent of the choice of conjugating element $g \in G(\wt{A})$.
\begin{lemma}\label{glambdaclosed}
$\mc{G}_{\lambda}$ is representable by a closed subscheme of $\mc{G}$.
\end{lemma}
\begin{proof}
Let $A$ be an $R_{\br}^{\square, \mbf{v}_{\lambda}}$-algebra, and fix an element $B \in \mc{G}(A)$. It suffices to show that there is an ideal $J \subset A$ such that for any ring homomorphism $f \colon A \to A'$, $f(B)$ belongs to $\mc{G}_{\lambda}(A')$ if and only if $A \to A'$ factors through $A/J$. Let $\Spec(\wt{A}) \to \Spec(A)$ be a Zariski cover such that for some $g \in G(\wt{A})$, $gB_{\wt{A}}g^{-1}= B_0$. Consider the root space decomposition of $\fg$ with respect to $T_0$, $\fg= \Lie(T_0) \oplus \bigoplus_{\beta} \fg_{\beta}$ into a direct sum of the Lie algebra of $T_0$ and free rank one $\mc{O}$-module root spaces $\fg_{\beta}$ with respect to characters $\beta \colon T_0 \to \mathbb{G}_m$. Let $X_{\beta}$ generate $\fg_{\beta}$, and let $p_{\beta}$ be the $\mc{O}$-module projection $\fg \to \fg_{\beta}$ (and likewise for the base-change to $\wt{A}$). For $B_0$-positive roots $\beta$ and $\sigma \in I_{F'_v}$, let $c_{\beta, \sigma} \in \wt{A}$ be the unique element such that 
\[
c_{\beta, \sigma} X_\beta= p_{\beta}(\Ad(g\rho^{\lambda}_{\wt{A}}(\sigma)g^{-1})X_{\beta}))-\beta(\chi_{\lambda}(\sigma))X_{\beta}.
\] 
Note that these constants are independent of the conjugating element $g \in G(\wt{A})$. We let $\wt{J} \subset \wt{A}$ be the ideal generated by $c_{\beta, \sigma}$ for all positive roots $\beta$ and $\sigma \in I_{F'_v}$; it is the smallest ideal of $\wt{A}$ such that $B_{\wt{A}/\wt{J}}$ belongs to $\mc{G}_{\lambda}(\wt{A}/\wt{J})$. This independence of the conjugating element implies that $\wt{J}$ is the extension of an ideal $J$ $A$: concretely, if $\Spec(\wt{A})= \sqcup_i \Spec(A_i)$ for Zariski-open subsets $\Spec(A_i) \subset \Spec(A)$, $\wt{J}$ is a collection of ideals $J_i \subset A_i$, and by independence of choice of conjugating element $J_i$ and $J_j$ generate the same ideal of $\Spec(A_i) \cap \Spec(A_j)$, and therefore the collection $(J_i)_i$ arises by extending an ideal $J \subset A$. 

We now verify the requisite condition on $J$. Suppose that $f \colon A \to A'$ is such that $f(B)$ belongs to $\mc{G}_{\lambda}(A')$. Then by definition there is an Zariski cover $\Spec(\wt{A}') \to \Spec(A')$ and $g' \in G(\wt{A}')$ conjugating $B_{\wt{A}'}$ to $B_0$ such that $g' \rho^{\lambda}_{A'}g^{-1}$ has the correct projection to $T_G(A')$. By definition, the analogues $c'_{\beta, \sigma}$ of the $c_{\beta, \sigma}$ (but now defined using $g'$) are equal to zero in $\wt{A}'$. We can also use the pull-back $\wt{A} \otimes_A A'$ of the Zariski-cover of $\Spec(A)$ used to define $J$. On the common refinement of the covers $\Spec(\wt{A}') \to \Spec(A') \leftarrow \Spec(\wt{A} \otimes_A A')$ of $\Spec(A')$, the $c_{\beta, \sigma}$ must equal the $c'_{\beta, \sigma}$ and therefore are 0. It follows that $f(J)=0$, i.e. $A \to A'$ factors through $A/J$. 

Conversely, if $A \to A'$ factors through $A/J$, then we witness $B_{A'}$ belonging to $\mc{G}_{\lambda}(A')$ by passing to the cover $\Spec(A' \otimes_{A/J} \wt{A}/\wt{J})$ (recall that $\wt{A}/\wt{J}= \wt{A} \otimes_A A/J$, so is still an open cover of $A/J$---this is where it is crucial that the ideal $\wt{J}$ descend to an ideal of $A$).
\end{proof}
We now define $R^{\triangle_{\lambda}}_{\br}$ to be the (global functions on the) scheme-theoretic image of $\mc{G}_{\lambda}[1/p] \to \Spec(R_{\br}^{\square, \mbf{v}_{\lambda}})$; it is the quotient of $R_{\br}^{\square, \mbf{v}_{\lambda}}$ by the kernel of the map $R^{\square, \mbf{v}_{\lambda}}_{\br} \to \mc{O}(\mc{G}_{\lambda}[1/p])$. As in \cite{geraghty:ordinary}, we can now observe the following:
\begin{itemize}
\item The morphism $\Lambda \colon \mc{G}_{\lambda} \to \Spec(R_{\br}^{\square, \mbf{v}_{\lambda}})$ is proper, so $\Lambda[1/p]$ is proper, and its image is therefore closed and equal (as topological space) to the scheme-theoretic image of $\Lambda[1/p]$, i.e. equal to $\Spec(R_{\br}^{\triangle_{\lambda}}[1/p]) \subset \Spec(R_{\br}^{\square, \mbf{v}_{\lambda}}[1/p])$.
\item If $y$ is a closed point of $\Spec(R_{\br}^{\triangle_{\lambda}}[1/p])$, then $y$ lifts uniquely to a closed point $\wt{y}$ of $\mc{G}_{\lambda}[1/p]$: indeed, it lifts by the previous bullet point, and the lift is unique because the co-characters $\lambda_\tau$ are all regular.
\item The induced map on complete local rings 
\[
(R^{\square, \mbf{v}_{\lambda}}_{\br}[1/p])^{\wedge}_y \onto (R^{\triangle_{\lambda}}_{\br}[1/p])^\wedge_{y} \into \mc{O}_{\mc{G}_{\lambda}[1/p], \wt{y}}^\wedge
\]
admits the following moduli interpretation as functors on the category $\mc{C}^f_{E_y}$ of finite local $E_y$-algebras $A$ with residue field $E_y$, where $E_y$ is the residue field of $y$: a local $E_y$-algebra homomorphism $(R^{\square, \mbf{v}_{\lambda}}_{\br}[1/p])^{\wedge}_y \to A$ corresponds to an object $\rho \in R^{\square, \mbf{v}_{\lambda}}_{\br}(A)$ such that $\rho \colon \gal{F_v} \to G(A)$ lifts the representation $\rho_y$ associated to $y$, a local $E_y$-algebra homomorphism $ \mc{O}_{\mc{G}_{\lambda}[1/p], \wt{y}}^\wedge \to A$ corresponds to a pair $(\rho, B) \in \mc{G}_{\lambda}(A)$ where $\rho$ lifts $\rho_y$, $B \in \mr{Fl}_G(A)$ lifts the Borel $B_{\wt{y}} \in \mc{G}_{\lambda}(E_y)$ corresponding to $\wt{y}$, and the composite $(R^{\square, \mbf{v}_{\lambda}}_{\br}[1/p])^{\wedge}_y \to \mc{O}_{\mc{G}_{\lambda}[1/p], \wt{y}}^\wedge$ is induced by the forgetful map $(\rho, B) \mapsto \rho$. Regularity of the characters $\lambda_\tau$ again implies that a choice of $B$ lifting $B_y$ with the specified projection to $T_G$ must be unique; we deduce that $\Hom_{\mc{C}^f_{E_y}}(\mc{O}_{\mc{G}_{\lambda}[1/p], \wt{y}}^\wedge, -)$ is a subfunctor of $\Hom_{\mc{C}_{E_y}^f}((R^{\square, \mbf{v}_{\lambda}}_{\br}[1/p])^{\wedge}_y, -)$, and thus from tangent space considerations the map $(R^{\square, \mbf{v}_{\lambda}}_{\br}[1/p])^{\wedge}_y \to \mc{O}_{\mc{G}_{\lambda}[1/p], \wt{y}}^\wedge$ is a surjection.
\end{itemize}
We omit the details, as they are very similar to the arguments of \cite{geraghty:ordinary}, taking into account the framework established in Lemma \ref{glambdaclosed}. We summarize these results in part (1) of the following lemma, and in part (2) we deduce our desired application to the structure of $R^{\triangle_{\lambda}}_{\br}$:
\begin{lemma}[cf. Lemma 3.10 of \cite{geraghty:ordinary}]\label{ordinary}
\begin{enumerate}
\item Let $A$ be a finite local $E$-algebra, and let $f \colon R_{\br}^{\square, \mbf{v}_{\lambda}} \to A$ be an $\mc{O}$-algebra homomorphism. Then $f$ factors through $R^{\triangle_{\lambda}}_{\br}$ if and only if $f \circ \rho^{\lambda}$ is $F_v'$-ordinary of weight $\lambda$ in the sense of Definition \ref{orddef}. 
\item $R^{\triangle_{\lambda}}_{\br}$ is a union of irreducible components of $R^{\square, \mbf{v}_{\lambda}}_{\br}$. In particular, $R^{\triangle_{\lambda}}_{\br}[1/p]$ admits an open dense regular subscheme, and all of its irreducible components have dimension $\dim(G)+[F_v: \Q_p]\dim(\mr{Fl}_G)$.
\item If we fix a lift $\mu \colon \gal{F_v} \to G/G^{\mr{der}}(\mc{O})$ of the multiplier character $\bar{\mu}$ and replace all the lifting rings with their fixed multiplier $\mu$ analogues, 
then the same assertions hold with $\dim(G^{\mr{der}})$ in place of $\dim(G)$.
\end{enumerate}
\end{lemma}
\begin{proof}
In (1), the ``if" direction is clear since the Borel witnessing ordinarity of $f \circ \rho^{\lambda}$ gives rise to a point of $\mc{G}_{\lambda}(A)$. The converse follows from the remarks before the proof: Let $y$ be the closed point underlying $f \pmod{\mf{m}_A}$, so that $f$ factoring through $R^{\triangle_{\lambda}}_{\br}[1/p]$ implies it induces a homomorphism $\mc{O}^{\wedge}_{\mc{G}_{\lambda}[1/p], \wt{y}} \to A$ corresponding to a pair $(\rho, B) \in \mc{G}_{\lambda}(A)$, and therefore $\rho=f \circ \rho^{\lambda}$ is ordinary of weight $\lambda$.

To show that $R^{\triangle_{\lambda}}_{\br}$ is a union of irreducible components of $R^{\square, \mbf{v}_{\lambda}}_{\br}$, it suffices now to compute for each closed point $y \in \Spec(R^{\triangle_{\lambda}}_{\br}[1/p])^\wedge_{y})$ a lower bound on the dimension of each irreducible component of $R^{\triangle_{\lambda}}_{\br}[1/p])^\wedge_{y}) \cong \mc{O}_{\mc{G}_{\lambda}[1/p], \wt{y}}^\wedge$ that is greater than or equal to $\dim((R^{\square, \mbf{v}_{\lambda}}_{\br}[1/p])^{\wedge}_y)= \dim(G)+[F_v:\Q_p]\dim(\mr{Fl}_G)$. We begin by observing that $\mc{O}^\wedge_{\mc{G}_{\lambda}[1/p], \wt{y}}$ also pro-represents the functor $\mc{F}_{\wt{y}} \colon \mc{C}^f_{E_y} \to \mbf{Sets}$ such that $\mc{F}_{\wt{y}}(A)$ is the set of pairs $(\rho, B)$ such that $\rho \colon \gal{F_v} \to G(A)$ lifts $\rho_y$, $\rho$ factors through $B(A)$ where $B$ lifts $B_{\wt{y}}$, and, for $g \in G(A)$ such that $gBg^{-1}=B_0$, $I_{F_v'} \xrightarrow{g \rho g^{-1}} B_0(A) \to T_G(A)$ is equal to $\chi_{\lambda}$. The difference from the previous assertion is that we do not assume $\rho|_{\gal{F'_v}}$ is semistable with $p$-adic Hodge type $\mbf{v}_{\lambda}$: but this is automatic (again via \cite[Proposition 1.28]{nekovar:p-adicheight}; to use this for general $G$, see for instance \cite[Lemma 4.8]{stp:exceptional}) from the fact that the above composite $I_{F'_v} \to T_G(A)$ is still $\chi_{\lambda}$. The dimension calculation proceeds as in \cite[Corollary 3.6, Lemma 3.7]{geraghty:ordinary}: there is some $g \in G(\mc{O}_{E_y})$ such that $g \rho_y g^{-1}$ takes values in $B_0(\mc{O}_{E_y})$, and the functor $\mc{F}_{\wt{y}}$ can then be identified with the product of functors of $B_0$-lifts of $g \rho_y g^{-1}$ and the functor of lifts of $B_{\wt{y}}$ that is represented by $\mc{O}_{\mr{Fl}_G, B_{\wt{y}}} \cong E_y[[X_1, \ldots, X_{\dim (\mr{Fl}_G)}]]$ ($\mr{Fl}_G$ is smooth over $\mc{O}$).

We are thus reduced to the simple problem of computing a lower bound on the dimension of the lifting ring $R^{B_0, \lambda}_{g \rho_y g^{-1}}$ parametrizing lifts $\rho \colon \gal{F_v} \to B_0(A)$ of $r:= g \rho_y g^{-1}$ such that the projection of $\rho$ to $T_G(A)$ is equal to $\chi_{\lambda}$ on $I_{F'_v}$. The tangent space of this functor is identified with $\ker(Z^1(\gal{F_v}, r(\mf{b}_0)) \to Z^1(I_{F'_v}, r(\mf{b}_0/\mf{n}_0))$, where $\mf{b}_0= \Lie(B_0)$ and $\mf{n}_0= \Lie(R_u(B_0))$. Via the exact sequence
\[
0 \to \mf{b}_0^{\gal{F_v}} \to \mf{b}_0 \to Z^1(\gal{F_v}, r(\mf{b}_0)) \to H^1(\gal{F_v}, r(\mf{b}_0)) \to 0,
\]
we must compute
\[
-\dim(\mf{b}_0^{\gal{F_v}}) + \dim(\mf{b}_0)+ \dim \left(\ker\left(H^1(\gal{F_v}, r(\mf{b}_0)) \to H^1(I_{F'_v}, r(\mf{b}_0/\mf{n}_0))^{\gal{F_v}/I_{F'_v}}\right)\right),
\]
which is at least
\[
-\dim(\mf{b}_0^{\gal{F_v}}) + \dim(\mf{b}_0) +\dim(H^1(\gal{F_v}, r(\mf{b}_0)))-\left(\dim(H^1(\gal{F_v}, r(\mf{b}_0/\mf{n}_0))- \dim(H^1(\gal{F_v}/I_{F'_v}, r(\mf{b}_0/\mf{n}_0))) \right).
\]
By the local Euler characteristic formula (as well as the vanishing of $H^2(\gal{F_v}, r(\mf{b}_0/\mf{n}_0))$, which follows from local duality, and the vanishing of the cohomology of any finite group with coefficients in a $\Q$-vector space), this last expression equals
\begin{align*}
&\dim(\mf{b}_0)+h^2(\gal{F_v}, r(\mf{b}_0))+\dim(\mf{b}_0)[F_v: \Q_p]-\left(\dim(\mf{b}_0/\mf{n}_0)+[F_v: \Q_p]\dim(\mf{b}_0/\mf{n}_0)-\dim(\mf{b}_0/\mf{n}_0) \right)\\
&=\dim(\mf{b}_0)+h^2(\gal{F_v}, r(\mf{g}_0))+ \dim(\mf{n}_0)[F_v:\Q_p].
\end{align*}
Adding in the $\dim(\mf{n}_0)$ contribution from the dimension of $\mr{Fl}_G$, we conclude that $\mc{O}_{\mc{G}_{\lambda}[1/p], \wt{y}}$ is isomorphic to $E[[X_1, \ldots, X_h]]/(f_1, \ldots, f_{h'})$, where $h\geq \dim(G)+\dim(\mr{Fl}_G)[F_v:\Q_p]+h^2(\gal{F_v}, r(\mf{b}_0))$ and $h'= h^2(\gal{F_v}, r(\mf{b}_0))$. This implies the desired lower bound on the dimension of each irreducible component of this ring, and part (2) of the Lemma follows.

The fixed multiplier case is similar, and we omit the details.
\end{proof}

\bibliographystyle{amsalpha}
 \bibliography{biblio.bib}

\end{document}